\documentclass[11pt,a4paper]{article}
\usepackage{amssymb,amsmath,amsthm}
\usepackage{enumitem}
\usepackage{mathtools}
\numberwithin{equation}{section}
\newtheorem{thm}{Theorem}[section]
\newtheorem{df}[thm]{Definition}
\newtheorem{prop}[thm]{Proposition}
\newtheorem{lem}[thm]{Lemma}
\newtheorem{rem}[thm]{Remark}

\newtheorem{cor}[thm]{Corollary}
\let\oldproofname=\proofname
\renewcommand{\proofname}{\rm\bf{\oldproofname}}

\newcommand{\N}{\mathbb{N}}
\newcommand{\Z}{\mathbb{Z}}

\newcommand{\R}{\mathbb{R}}
\newcommand{\C}{\mathbb{C}}

\newcommand{\cB}{\mathcal{B}}

\newcommand{\cE}{\mathcal{E}}
\newcommand{\cF}{\mathcal{F}}

\newcommand{\cK}{\mathcal{K}}
\newcommand{\cL}{\mathcal{L}}
\newcommand{\cM}{\mathcal{M}}

\newcommand{\cO}{\mathcal{O}}

\newcommand{\cS}{\mathcal{S}}
\newcommand{\cT}{\mathcal{T}}
\newcommand{\cU}{\mathcal{U}}

\newcommand{\cX}{\mathcal{X}}
\newcommand{\cY}{\mathcal{Y}}
\newcommand{\cZ}{\mathcal{Z}}

\renewcommand{\Re}{\mathop{\mathrm{Re}}}
\renewcommand{\Im}{\mathop{\mathrm{Im}}}
\newcommand{\dd}{\,{\rm d}}
\newcommand{\D}{{\rm d}}
\renewcommand{\div}{\mathop{\mathrm{div}}\nolimits}
\newcommand{\curl}{\mathop{\mathrm{curl}}}

\newcommand{\supp}{\mathop{\mathrm{supp}}}
\newcommand{\1}{\mathbf{1}}
\newcommand{\weakto}{\rightharpoonup}

\newcommand{\DS}{\displaystyle}
\newcommand{\BS}{\mathrm{BS}_+}
\newcommand{\BSS}{\mathrm{BS}}
\newcommand{\QED}{\mbox{}\hfill$\Box$}
\renewcommand{\:}{\thinspace :}

\newcommand{\Rey}{\mathrm{Re}}
\newcommand{\bom}{{\boldsymbol\omega}}
\newcommand{\bw}{{\boldsymbol w}}

\begin{document}

\title{Viscous evolution of a point vortex in a half-plane}

\author{Anne-Laure Dalibard and Thierry Gallay}

\maketitle

\begin{abstract}
As a model for vortex-wall interactions, we consider the two-dimensional
incompressible Navier--Stokes equations in the half-plane $\R^2_+$ with no-slip
boundary condition and point vortices as initial data. We focus
on the paradigmatic example of a single vortex in an otherwise stagnant fluid,
which is already quite challenging from a mathematical point of view. We prove
that this system has a unique global solution for all values of the Reynolds
number $\Rey = |\Gamma|/\nu$, where $\Gamma$ is the circulation of the vortex
and $\nu$ the kinematic viscosity of the fluid. The solution we construct has
finite energy for positive times and converges to zero in energy norm as
$t \to +\infty$. Uniqueness holds under the assumption that the solution is
close to a Lamb--Oseen vortex for small times. To our knowledge, all
previous results in domains with boundaries assume that the initial vorticity
has small or zero atomic part. In our particular situation, we remove the
smallness condition by decomposing the solution into a vortex and a boundary
layer term, so that we can apply the techniques developed in the whole plane
$\R^2$ to avoid the difficulties related to the large circulation of
the vortex. 
\end{abstract}

\section{Introduction and main results}\label{sec1}

Interactions of vortices with rigid walls play an important role in fluid
dynamics, and have serious consequences in real-world applications \cite{DSW}. A
typical example is the {\em rebound effect} that is observed when a pair of
trailing vortices in the wake of a landing aircraft bounce up near the ground
due to their interaction with secondary vortical structures created in the boundary
layer \cite{PSW1,PSW2}.  A similar phenomenon occurs during the head-on
collision of a vortex ring with a rigid wall that is normal to the symmetry axis
\cite{LNC}.  Although three-dimensional effects such as the Crow or the elliptic
instability have to be taken into account in real experiments, the rebound is a
genuinely two-dimensional phenomenon, in the sense that it can be observed
already when a vortex dipole approaches the boundary in a two-dimensional domain
\cite{Or}. In the perspective of obtaining rigorous mathematical results, an
even simpler situation that is worth studying is the {\em self-induced motion}
of a single vortex in the presence of a boundary, which results from the
interaction with the boundary layer created by the vortex itself \cite{DSW}.

The goal of this paper is to provide a mathematical framework in which the
above-mentioned phenomena could potentially be studied. To keep the situation as
simple as possible, we suppose that the fluid domain is the upper half-plane
$\R^2_+ := \R \times (0,+\infty)$, and that the fluid velocity vanishes on the
boundary $\partial\R^2_+ = \R \times \{0\}$. It is well-known that the
initial-boundary value problem for the incompressible Navier--Stokes equations
in $\R^2_+$, with no-slip boundary condition for the velocity, is globally
well-posed in the energy class; see, for instance, \cite{KO}. More generally, a
recent result of Abe \cite{Abe} shows the existence of a unique global solution
when the initial vorticity is a Radon measure whose atomic part is small
compared to the kinematic viscosity of the fluid, which we denote by $\nu >
0$. In particular, the initial vorticity can be a finite collection of point
vortices of circulations $\Gamma_1, \dots, \Gamma_N$ provided
$|\Gamma_1| + \dots + |\Gamma_N| \le c_0\nu$ for some universal constant
$c_0 > 0$. Unfortunately, this last restriction is not appropriate for our
purposes, because the self-induced motion of a vortex or the rebound of a vortex
dipole only occur in the opposite regime where the circulation of the vortices
is large compared to the viscosity.  It thus appears necessary to extend the
result of \cite{Abe} so as to remove the smallness assumption on the atomic part
of the initial vorticity.

In this work, we make an important step in this direction by showing that
initial-boundary value problem for the two-dimensional Navier--Stokes equations
in the half-plane $\R^2_+$, with no-slip boundary condition, is globally
well-posed when the initial vorticity is a single point vortex of {\em
  arbitrary} circulation $\Gamma \in \R$, located at any point $z \in
\R^2_+$. This is the analogue of the well-posedness result established in
\cite{GW2} for the Cauchy problem in the whole plane $\R^2$, which opened the
way to the general case considered in \cite{GG}. An important ingredient of the
proof is a semi-explicit expression for the Stokes semigroup $S(t)$ in $\R^2_+$,
in vorticity formulation, which was derived by Maekawa \cite{Mae} and Abe
\cite{Abe} following previous work by Solonnikov \cite{So1,So2} and Ukai
\cite{Ukai}. That expression suggests a decomposition of the form
$S(t) = S_1(t) + S_2(t)$, where $S_1(t)$ is essentially the heat semigroup in
$\R^2$, and $S_2(t)$ is a correction that takes into account the creation
of vorticity within the boundary layer. Using the integral formulation of the
vorticity equation and a cut-off function, we obtain a similar decomposition of
the solution $\omega(t)$ into a concentrated vortex $\omega_1(t)$ and a boundary
layer correction $\omega_2(t)$. Roughly speaking, the vortex can be treated as in
the whole plane $\R^2$, see \cite{GW2,GG}, whereas the boundary term is smaller
and can be dealt with using a fixed point argument. As we shall see, it is
useful to further decompose $\omega_1(t)$ into a leading order term $\bar \omega_1(t)$,
which is the self-similar Lamb--Oseen vortex, and a remainder $\hat \omega_1(t)$
that can be treated perturbatively. Understanding the interactions between the various
terms in our decomposition is an important part of our analysis. 

At a more conceptual level, the difficulty of understanding the interaction
between a vortex and a solid wall, in the regime where the Reynolds number
$\Rey := \Gamma/\nu$ is large, is related to the development of instabilities
within the boundary layer. The unstable modes constructed in recent works such
as \cite{GGNg,GNg,BG} are responsible for the creation of vorticity near the
wall. They correspond to high tangential wave numbers, of order $\Rey^\beta$ for
some $\beta > 0$, and are typically located within a small sub-layer of the
Prandtl layer, which is of size $\Rey^{1/2}$.  These small scale structures,
which are believed to play an important role in the transition towards
turbulence, are sometimes deemed as ``vortices'', but they still lack a precise
description in physical space.  We hope that the model case considered in the
present paper will contribute to further developments in boundary layer theory.

\subsection{Formulation of the problem}\label{ss:statement}

We work in the upper half-plane $\R^2_+ := \bigl\{x = (x_1,x_2) \in \R^2\,;\,
x_2 > 0\bigr\}$. If $x = (x_1,x_2) \in \R^2_+$, we denote $x^\perp
= (-x_2,x_1)$, $x^* = (x_1,-x_2)$, and $|x|^2 = x_1^2 + x_2^2$. We consider
the vorticity equation
\begin{equation}\label{2Dvort}
  \partial_t \omega(x,t) + u(x,t)\cdot\nabla \omega(x,t) \,=\,
  \nu \Delta\omega(x,t)\,, \qquad x \in \R^2_+\,, \quad t \ge 0\,,
\end{equation}
where $u(\cdot,t)$ is the velocity field obtained from $\omega(\cdot,t)$
via the Biot--Savart formula in $\R^2_+$. At any fixed time, we have the
explicit expression
\begin{equation}\label{BSlaw}
  u(x) \,=\, \BS[\omega](x) \,:=\, \frac{1}{2\pi}\int_{\R^2_+}\biggl(
  \frac{(x-y)^\perp}{|x-y|^2} - \frac{(x-y^*)^\perp}{|x-y^*|^2}\biggr)
  \omega(y)\dd y\,,
\end{equation}
which makes sense for instance if $\omega \in L^p(\R^2_+)$ for some
$p \in (1,2)$. The velocity field $u = (u_1,u_2)$ satisfies $\div u :=
\partial_1 u_1 + \partial_2 u_2 = 0$ and $\curl u := \partial_1 u_2 -
\partial_2 u_1 = \omega$. In addition, the vertical component $u_2(x)$
vanishes when $x = (x_1,0) \in \partial\R^2_+$.

Assuming that the kinematic viscosity $\nu$ is positive, we need an
additional boundary condition to obtain a unique solution to the
parabolic equation \eqref{2Dvort}. We impose the {\em no-slip boundary condition},
which asserts that the horizontal velocity $u_1(x)$ also vanishes
when $x = (x_1,0) \in \partial\R^2_+$. In view of \eqref{BSlaw}, this
is equivalent to the following integral relation for the vorticity $\omega$\:
\begin{equation}\label{noslip}
  \int_{\R^2_+} \frac{y_2}{(x_1-y_1)^2 + y_2^2}\,\omega(y,t)
  \dd y \,=\, 0\,, \qquad \forall\,x_1 \in \R\,,\,~\forall\, t > 0\,.
\end{equation}
This condition means that $\omega$ is orthogonal in $L^2(\R^2_+)$ to a large
family of harmonic functions, see \cite{LM} for a detailed discussion of the
no-slip boundary condition in bounded domains. 

Our goal is to solve the initial-boundary value problem
\eqref{2Dvort}--\eqref{noslip} with singular initial data of the form
$\omega_0 \,=\, \Gamma \,\delta_z$ for some $\Gamma \in \R$ and some
$z \in \R^2_+$. In that case, the general result of Abe \cite{Abe} asserts the
existence of a unique solution to \eqref{2Dvort}--\eqref{noslip} {\em provided}
$|\Gamma|\le c_0\nu$, where $c_0 > 0$ is a universal constant. However, the
approach of \cite{Abe} relies on a fixed point argument which inevitably fails
when $|\Gamma|/\nu$ is large. It is interesting to note that existence
of solutions with general measures as initial data can be established, by an
approximation argument, if the no-slip condition \eqref{noslip} is replaced by
the homogeneous Dirichlet condition for the vorticity $\omega$, see \cite{MY}.
However, even in that situation, uniqueness of the solution has been established
only if the atomic part of the initial vorticity is small compared to the viscosity.

In what follows, we find it convenient to use dimensionless (dependent
and independent) variables for which $\nu = 1$ and $z = (0,1)$. The
initial data become
\begin{equation}\label{initdata}
  \omega_0 \,=\, \alpha\,\delta_z\,, \qquad \mbox{where}\quad
  \alpha \,=\, \frac{\Gamma}{\nu}\,.
\end{equation}
Without loss of generality, we may assume that $\Gamma > 0$. The
dimensionless quantity $\alpha = \Gamma/\nu$ is called the
{\em circulation Reynolds number}. 

\subsection{The Stokes semigroup in vorticity formulation}\label{ss:linear}

We will construct local solutions of \eqref{2Dvort} by means of a fixed point
argument based on the Duhamel formula. Therefore, our analysis relies
crucially on the properties of the Stokes equation, which is obtained by
linearizing \eqref{2Dvort} near $\omega = 0$. We introduce here the notation
that will be necessary to state our main result in Section~\ref{ss:results},
together with important estimates which will be used throughout the paper.

We consider the linear equation $\partial_t \omega = \Delta \omega$
in $\R^2_+$ with no-slip boundary condition \eqref{noslip} and initial data
$\omega_0 \in L^1(\R^2_+)$. As established in \cite{Abe}, see also
\cite{Mae,Ukai}, the solution takes the semi-explicit form
\begin{equation}\label{Stokes}
  \bigl(S(t)\omega_0\bigr)(x) \,=\, \int_{\R^2_+} K(x,y,t)\,\omega_0(y)\dd y\,,
  \qquad x \in \R^2_+\,,\quad t > 0\,,
\end{equation}
where the integral kernel $K(x,y,t)$ can be expressed in terms of the
Gaussian function $G_t$ and of the Poisson function $P$\:
\begin{equation}\label{GPdef}
  G_t(x) \,=\, \frac{1}{4\pi t}\,e^{-|x|^2/(4t)}\,, \qquad
  P(x) \,=\, \frac{1}{\pi}\,\frac{x_2}{x_1^2 + x_2^2}\,, \qquad
  x \in \R^2_+, \quad t > 0\,.
\end{equation}
More precisely, we have
\begin{equation}\label{Kdecomp}
  K(x,y,t) \,=\, G_t(x-y) \,-\, G_t(x-y^*) \,-\, K_0(x,y,t)\,,
\end{equation}
for all $x,y \in \R^2_+$ and all $t > 0$, where
\begin{equation}\label{K0def}
  K_0(x,y,t) \,=\, -2\partial_{x_2} \int_{\R}\int_0^{y_2} G_t(x-z^*)\,
  P(y-z)\dd z_2 \dd z_1\,.
\end{equation}
For the reader's convenience, we give a short proof of formula \eqref{Stokes}
in Appendix~\ref{appA}. 

As is well known the expression $G_t(x-y) - G_t(x-y^*)$ is the integral kernel
of the heat semigroup in $\R^2_+$ with Dirichlet boundary condition, which we
will denote by $e^{t\Delta_D}$. The additional term $K_0$ in \eqref{Kdecomp} is
a {\em boundary layer correction} which is needed to enforce the no-slip
boundary condition \eqref{noslip}. We thus have the decomposition
\begin{equation}\label{Sdecomp1}
  S(t) \,=\, e^{t\Delta_D} - S_0(t)\,, \qquad t > 0\,,
\end{equation}
where $S_0(t)$ is the linear operator with kernel $K_0(x,y,t)$\: 
\begin{equation}\label{SOdef}
  \bigl(S_0(t)\omega_0\bigr)(x) \,=\, \int_{\R^2_+} K_0(x,y,t)\,\omega_0(y)\dd y\,,
  \qquad x \in \R^2_+\,,\quad t > 0\,.
\end{equation}

It is important to observe that formula \eqref{Stokes} does not define a strongly continuous
semigroup in the whole space $L^1(\R^2)$. Indeed, if $\omega_0 \in L^1(\R^2)$ and if the
associated velocity $u_0 = \BS[\omega_0]$ does not satisfy \eqref{noslip}, then
$\|S(t)\omega_0 - \omega_0\|_{L^1}$ does not converge to zero as $t \to 0$ due to the
boundary layer term, see \cite[Lemma~4.1]{Abe}. For this reason, we introduce
the closed subspace 
\begin{equation}\label{L1perp}
  L^1_\perp(\R^2_+) \,=\, \bigl\{\omega \in L^1(\R^2_+)\,;\, \gamma[\omega] = 0
  \bigr\}\,,
\end{equation}
where $\gamma : L^1(\R^2_+) \to L^1(\R)$ is the trace operator defined by
\begin{equation}\label{gammadef}
  \bigl(\gamma[\omega]\bigr)(x_1) \,=\, \frac{1}{\pi}\,\int_{\R^2_+} \frac{y_2}{(x_1-y_1)^2 + y_2^2}
  \,\omega(y)\dd y\,, \qquad x_1 \in \R\,.
\end{equation}
As is easily verified, the linear operator $\gamma : L^1(\R^2_+) \to L^1(\R)$ is
bounded, see Appendix~\ref{appA}. The main properties of the Stokes semigroup
are collected in the following statement, which is essentially taken from
\cite{Abe}. 

\begin{prop}\label{prop:Stokes} {\rm (Properties of the Stokes semigroup)}\\[1mm]
1) The family $\bigl(S(t)\bigr)_{t \ge 0}$ defined by \eqref{Stokes} is a $C_0$-semigroup
in the space $L^1_\perp(\R^2_+)$. \\[1mm]
2) If $\omega_0 \in L^1(\R^2)$ the maps $t \mapsto S(t)\omega_0$ and
$t \mapsto S(t)\nabla\omega_0$ belong to $C^0((0,\infty),L^p(\R^2_+)\cap L^1_\perp(\R^2_+))$\\
\null\hspace{11pt} for any $p \in [1,+\infty]$, and there exists a constant $C > 0$ such that
\begin{equation}\label{Stokes2}
  \|S(t)\omega_0\|_{L^p} + t^{1/2}\|S(t)\nabla\omega_0\|_{L^p} \,\le\,
  C\,\frac{\|\omega_0\|_{L^1}}{t^{1-1/p}}\,, \qquad t > 0\,.
\end{equation}
\end{prop}

\noindent Here the operator $S(t)\nabla$ is defined by the integral formula
\begin{equation}\label{StokesDer}
  \bigl(S(t)\nabla\omega_0\bigr)(x) \,=\, -\int_{\R^2_+} \nabla_y K(x,y,t)\,
  \omega_0(y)\dd y\,, \qquad x \in \R^2_+\,,\quad t > 0\,.
\end{equation}

Let $C_0(\R^2_+)$ be the (Banach) space of all continuous functions on $\R^2_+$
which vanish at infinity and on the boundary $\partial\R^2_+$, equipped with the
$L^\infty$ norm. Its topological dual $\cM(\R^2_+) = C_0(\R^2_+)'$ is the space
of all finite (Radon) measures on the half-plane $\R^2_+$, equipped with the
total variation norm $\|\cdot\|_{\cM}$. Another topology on $\cM(\R^2_+)$
is also useful\: we say that a sequence $(\mu_n)_{n \in \N}$ in $\cM(\R^2_+)$ converges
weakly to $\mu \in \cM(\R^2_+)$ if $\langle \mu_n,\phi\rangle \to \langle \mu,\phi\rangle$
for all $\phi \in C_0(\R^2_+)$, where $\langle \cdot,\cdot\rangle$ denotes the
duality pairing. We then write $\mu_n \weakto \mu$.

As is easily verified, for any $x \in \R^2_+$ and any $t > 0$, the function
$y \mapsto K(x,y,t)$ belongs to $C_0(\R^2_+)$. This allows us to define
the action of the semigroup $S(t)$ on any measure $\mu \in \cM(\R^2_+)$
by the formula $\bigl(S(t)\mu\bigr)(x) = \langle \mu\,,K(x,\cdot,t)\rangle$. 
Then $t \mapsto S(t)\mu \in C^0((0,\infty),L^p(\R^2_+)\cap L^1_\perp(\R^2_+))$
and in analogy with \eqref{Stokes2} we have the estimate
\begin{equation}\label{Stokes3}
  t^{1-1/p}\|S(t)\mu\|_{L^p} \,\le\, C\,\|\mu\|_{\cM}\,, \qquad \forall\, t > 0\,.
\end{equation}

\subsection{Statement of the main results}\label{ss:results}

We first specify what we mean by a solution of \eqref{2Dvort} in the space
$L^1(\R^2_+)$. 

\begin{df}\label{def:sol}
Let $T > 0$. We say that a function $\omega \in C^0\bigl((0,T),L^1_\perp(\R^2_+)
\cap L^{4/3}(\R^2_+)\bigr)$ is a (mild) solution of equation \eqref{2Dvort} on
the interval $(0,T)$ if
\begin{equation}\label{IntEq}
  \omega(t) \,=\, S(t-t_0)\omega(t_0) - \int_{t_0}^t S(t-s)\div\bigl(
  u(s)\omega(s)\bigr)\dd s\,,
\end{equation}
whenever $0 < t_0 < t < T$, where $u(s) = \BS[\omega(s)]$ is the velocity 
field associated with $\omega(s)$. 
\end{df}

\begin{rem}\label{rem:mild} The assumption that $t \mapsto \omega(t)
\in C^0\bigl((0,T),L^{4/3}(\R^2_+)\bigr)$ is sufficient to give a meaning to 
the integral term in \eqref{IntEq}. Indeed, if $p = 1$ or $p = 4/3$, we have
\[
  \bigl\|S(t-s)\div \bigl(u(s)\omega(s)\bigr)\bigr\|_{L^p} \,\lesssim\,
  \frac{\|u(s)\omega(s)\|_{L^1}}{(t-s)^{3/2-1/p}} \,\lesssim\,
  \frac{\|u(s)\|_{L^4}\|\omega(s)\|_{L^{4/3}}}{(t-s)^{3/2-1/p}} \,\lesssim\,
  \frac{\|\omega(s)\|_{L^{4/3}}^2}{(t-s)^{3/2-1/p}}\,,
\]
where the first inequality follows from \eqref{Stokes2}, the second
from H\"older's inequality, and the last one from standard properties
of the Biot--Savart law, see Lemma~\ref{lem:HLS}. In fact, by parabolic 
smoothing, any solution in the sense of Definition~\ref{def:sol} is 
regular on $\R^2_+ \times (0,T)$ and satisfies the PDE \eqref{2Dvort}
in the classical sense. Moreover the associated velocity field given 
by the Biot--Savart formula \eqref{BSlaw} vanishes on the boundary 
$\partial \R^2_+$. 
\end{rem}

In what follows, we are interested in solutions of \eqref{2Dvort} for which
$\|\omega(t)\|_{L^1}$ stays bounded as $t \to 0$. This is typically the case
when $\omega(t) \weakto \mu$ as $t \to 0$ for some finite measure $\mu \in \cM(\R^2)$.
As already mentioned, we focus on the particular case $\mu = \alpha\,\delta_z$,
where $\alpha \in \R$ and $z = (0,1)$. In that situation, we expect that the solution
will be close for short times to the Lamb--Oseen vortex $\omega(x,t) =
\alpha G_t(x-z)$, where $G_t(x)$ is the Gaussian function defined in \eqref{GPdef}. 

The main result of this paper can be stated as follows\: 

\begin{thm}\label{thm:main} {\rm (Existence, uniqueness, and scale invariant
bounds)}\\
Fix $\alpha \in \R$ and $z = (0,1)$. For any sufficiently small $T > 0$, equation 
\eqref{2Dvort} has a unique mild solution $\omega \in C^0\bigl((0,T),L^1_\perp(\R^2_+)
\cap L^{4/3}(\R^2_+)\bigr)$ such that $\omega(t) \weakto \alpha\delta_z$ as
$t \to 0$ and
\begin{equation}\label{maincond}
  \sup_{0 < t < T}\|\omega(t)\|_{L^1} \,<\, \infty\,, \qquad 
  \sup_{0 < t < T}t^\beta\,\|\omega(t) - \alpha G_t(\cdot-z)\|_{L^{4/3}} \,<\, \infty\,,
\end{equation}
for some $\beta < 1/4$. Furthermore, estimate \eqref{maincond} actually holds with
$\beta = 1/8$, and we have the following bounds for the vorticity $\omega(t)$
and the associated velocity $u(t) = \BS[\omega(t)]$\:
\begin{equation}\label{mainest}
  \sup_{0 < t < T}t^{1-1/p}\|\omega(t)\|_{L^p} \,<\, \infty\,, \qquad
  \sup_{0 < t < T}t^{1/2-1/q}\|u(t)\|_{L^q} \,<\, \infty\,,
\end{equation}
for all $p \in [1,\infty]$ and all $q \in (2,+\infty]$. 
\end{thm}

\begin{rem}\label{rem:2D}
In the whole plane $\R^2$, the Lamb--Oseen vortex $\alpha G_t(x)$ is known
to be the unique solution of the Navier-Stokes equations for which the vorticity
is uniformly bounded in $L^1(\R^2)$ and converges weakly to $\alpha\delta_0$ as
$t \to 0$, see \cite{GW2,GG}. In other words, the second assumption in
\eqref{maincond} can be dispensed with in that case. The situation is less
clear in the half-plane $\R^2_+$, and it is quite possible that \eqref{maincond}
is not optimal as far as uniqueness is concerned. We also observe that
\eqref{maincond} implies the vorticity estimate in \eqref{mainest} for $p = 1$
and $p = 4/3$, hence for all all $p \in [1,4/3]$ by interpolation. The
velocity estimate for $q \in (2,4]$ follows by Lemma~\ref{lem:HLS}.
\end{rem}

Theorem~\ref{thm:main} holds for a small time $T > 0$ which is not precisely
quantified, but this smallness condition can easily be relaxed. Indeed, we
know that the vorticity $\omega(t)$ belongs to $L^1_\perp(\R^2_+)$ as soon as
$t > 0$, so that the corresponding measure $\mu(t) = \omega(t)\dd x$ has no
atomic part when $t > 0$. This implies that our solution can be extended to
the whole time interval $(0,+\infty)$ by using, for instance, the result of Abe
\cite{Abe}, which holds for initial data with small atomic part. In particular,
the statement of Theorem~\ref{thm:main} holds for
any fixed time $T > 0$.

Alternatively, it is possible to extend our solution by invoking classical
results on the 2D Navier-Stokes equations in velocity formulation. Indeed,
using weighted energy estimates, we prove in Proposition~\ref{prop:weight} that
the vorticity $\omega(t)$ in Theorem~\ref{thm:main} is strongly localized in the
vertical direction, in the sense that
\begin{equation}\label{x2weight}
  \sup_{0 < t < T}\,\int_{\R^2_+} (1+x_2)^m|\omega(x,t)|\dd x  \,<\, \infty\,,
  \qquad \forall\, m\in \N\,.
\end{equation}
The properties of the Biot--Savart law \eqref{BSlaw} then imply that the velocity
$u(t) = \BS[\omega(t)]$ belongs to $L^p(\R^2_+)$ for all $p \in (1,2)$, see
Lemma~\ref{lem:BS} below. Since we also know from \eqref{mainest} that
$u(t) \in L^q(\R^2_+)$ for some $q > 2$, we deduce by interpolation that
$u(t) \in L^2(\R^2_+)$. In other words, our solution coincides for $t \in (0,T)$
with a {\em finite energy solution} of the 2D Navier-Stokes equations in the
upper half-plane. Those solutions are globally defined, and a lot is known about
their long-time behavior, see for instance \cite{BM,KO,Han}. We thus
obtain\: 

\begin{prop}\label{prop:timedecay} {\rm (Global existence and long-time
behavior)}\\
The local solution constructed in Theorem~\ref{thm:main} can be extended to a
global solution which satisfies $\omega \in C^0\bigl((0,+\infty), L^1_\perp(\R^2_+)
\cap L^q(\R^2_+)\bigr)$ and $u\in C^0\bigl((0,+\infty),L^q(\R^2_+)\bigr)$
for any $q > 1$. Moreover,
\begin{equation}\label{timedecay}
  \|\omega(t)\|_{L^1} + \|u(t)\|_{L^2} \,=\, \cO\bigl(t^{-1/2}\bigr)\,,
  \qquad \text{as }\, t \to +\infty\,.
\end{equation}
\end{prop}

The decay rates in \eqref{timedecay} are optimal, and a direct calculation shows
that they coincide with the decay rates of the Stokes solution $\omega(t) =
S(t)\delta_z$, see Appendix~\ref{appB}. 

Another interesting question is the short-time behavior of the solution of
\eqref{2Dvort}. Under the assumptions of Theorem~\ref{thm:main}, we can prove
that
\begin{equation}\label{firstasym}
  \bigl\|\omega(t) - \alpha S(t)\delta_z \bigr\|_{L^1} + t^{1/4}
  \bigl\|\omega(t) - \alpha S(t)\delta_z \bigr\|_{L^{4/3}} \,=\,
  \cO\bigl(t^{1/8}\bigr)\,, \qquad \text{as }\,t \to 0\,,
\end{equation}
see Section \ref{ssec43} below. In other words, the Stokes solution
$\alpha S(t)\delta_z$ is the leading order term in a short time asymptotic
expansion of the full solution $\omega(t)$.  Since
$\bigl(S(t)\delta_z\bigr)(x) = K(x,z,t)$, an inspection of the formulas
\eqref{Kdecomp}, \eqref{K0def} reveals that $\omega(t)$ is well approximated for
small times by the superposition of a Lamb--Oseen vortex located at point
$z \in \R^2_+$ and a boundary layer correction supported in a neighborhood of
size $\sqrt{t}$ of $\partial \R^2_+$. What the Stokes approximation does not
account for is the motion of the vortex due to the interaction with the boundary
layer. The position of the vortex can be defined by the formula
\begin{equation}\label{def:Z}
  Z(t) \,=\, \frac{1}{\alpha}\int_{\R^2_+} \chi(x)\,x\,\omega(x,t)\dd x\,, \qquad
  t > 0\,,
\end{equation}
where $\chi$ is a smooth function that is equal to $1$ in a small neighborhood
of $z = (0,1)$ and to $0$ outside a larger neighborhood of $z$. The localization
estimates in Proposition~\ref{prop:weight2} below show that this definition is independent
of the choice of the cut-off $\chi$, up to corrections of size $\cO(t^\infty)$ as
$t \to 0$. Our final result is a computation of the velocity $Z'(t)$ for
small times. 

\begin{prop}\label{prop:speed} {\rm (Translation speed at initial time)}\\
Under the assumptions of Theorem~\ref{thm:main}, the velocity of the point
$Z(t)$ defined by \eqref{def:Z} satisfies $Z'(t) = V_\infty + \cO(t^{1/8})$
as $t \to 0$, where
\begin{equation}\label{Vasym}
  V_\infty \,=\, \frac{\alpha}{4\pi}\, \begin{pmatrix} 1 \\ 0\end{pmatrix} \,=\,
  -\frac{\alpha\,z^\perp}{4\pi|z|^2}\,.
\end{equation}
\end{prop}

Note that $V_\infty$ is exactly the speed of a point vortex of circulation
$\alpha$ located at point $z$ under the influence of a ``mirror vortex'' of
circulation $-\alpha$ located at the conjugate point $z^*$. In other words, to
leading order as $t \to 0$, the influence of the viscous boundary layer on the
position of the vortex is correctly described by the action of a fictitious
mirror vortex, as in the inviscid picture. To our knowledge, Proposition~\ref{prop:speed}
is the first rigorous justification of the motion of a concentrated vortex 
for the 2D Navier-Stokes equations in a domain with boundary. 

\subsection{Sketch of the proof of the main result}\label{ss:ideas}

If $\omega(x,t)$ is a solution of \eqref{2Dvort} satisfying \eqref{maincond}
and such that $\omega(t) \weakto \alpha\delta_z$ as $t \to 0$, we first take the limit
$t_0 \to 0$ in \eqref{IntEq} to obtain a representation formula that involves explicitly
the initial data\:
\begin{equation}\label{IntEqIni}
  \omega(t) \,=\, \alpha S(t) \delta_z - \int_0^t S(t-s) \div \bigl(u(s) \omega(s)
  \bigr)\dd s\,, \qquad t \in (0,T)\,.
\end{equation}
The main idea of the proof is then to decompose the solution $\omega(t)$ into
a concentrated vortex $\omega_1(t)$ originating from the Dirac mass in the
initial data, and a correction term $\omega_2(t)$ which takes into acount
the creation of vorticity in the boundary layer. At the linear level,
this corresponds to decomposing the Stokes semigroup as $S(t) = S_1(t) + S_2(t)$,
where
\begin{equation}\label{S12def}
\begin{split}
  \bigl(S_1(t)\omega_0\bigr)(x) \,=\, \int_{\R^2_+} G_t(x-y)\,\omega_0(y)\dd y\,, \quad
  \bigl(S_2(t)\omega_0\bigr)(x) \,=\, \int_{\R^2_+} K_2(x,y,t)\,\omega_0(y)\dd y\,,
\end{split}
\end{equation}
so that $K_2(x,y,t) = -G_t(x-y^*) - K_0(x,y,t)$.  Note that this decomposition
differs from \eqref{Sdecomp1} since the mirror vortex term $-G_t(x-y^*)$ in
\eqref{Kdecomp} is now included in the correction term $S_2(t)$. The
reason for that choice is that we want $S_1(t)$ to coincide with the
heat semigroup $e^{t\Delta}$ in the whole plane $\R^2$, when applied to
initial data that are supported in $\R^2_+$. 

We next define the vortex part $\omega_1(t)$ by the formula
\begin{equation}\label{defom1}
  \omega_1(t) \,=\, \alpha S_1(t) \delta_z - \int_0^t S_1(t-s) \div \bigl(\chi\,
  u(s) \omega(s)\bigr)\dd s\,, \quad t \in (0,T)\,,
\end{equation}
where, as in \eqref{def:Z}, the localization function $\chi$ is compactly
supported in $\R^2_+$ and equal to $1$ in a neighborhood of $z = (0,1)$. The
idea is now to perform a fixed point argument on the pair of integral equations
satisfied by the components $\omega_1$ and $\omega_2 := \omega - \omega_1$. The
advantage of the above decomposition is that equation \eqref{defom1} can be
easily extended to the whole plane $\R^2$, where we can use the techniques
introduced in \cite{GW2,GG} to circumvent the difficulty due to the large Dirac
mass in the initial data. This approach works provided $\omega_1(t)$ is
reasonably close to the Lamb--Oseen vortex $\alpha G_t(\cdot-z)$, and this is the
reason for the second assumption in \eqref{maincond}. On the other hand, the
integral equation for $\omega_2(t)$ involves either the operator $S_2(t)$, which
(according to Lemma~\ref{lem:Stokesbounds} below) satisfies better estimates
than $S_1(t)$ for small times, or the tamed nonlinearity $(1-\chi)u\omega$,
where the largest part of the solution is filtered out by the cut-off.
Summarizing, we can perform a fixed point argument for small times on the pair
of integral equations satisfied by the quantities $\hat\omega_1(t) := \omega_1(t)
- \alpha G_t(\cdot-z)$ and $\omega_2(t)$, and this gives at the same time the
existence and uniqueness claims in Theorem~\ref{thm:main}.

\subsection{Perspectives}\label{ss:perspect}

The present paper should stimulate further investigations of the initial-boundary
value problem for the 2D Navier-Stokes equations in unbounded domains, for
singular initial data including point vortices. In the case of a half-plane, it
is rather straightforward to modify the proof of Theorem~\ref{thm:main} so as to
include a finite sum of Dirac masses as initial vorticity.  Combining point
vortices with more regular structures such as vortex sheets or patches is more
challenging, but should be amenable to analysis using the techniques developed
in \cite{GG}. So there are reasons to hope that the restriction on the size of
the atomic part of the initial vorticity in the general result of Abe \cite{Abe}
will be removed in a near future.

In a different direction, it is natural to ask if our main result,
Theorem~\ref{thm:main}, remains valid in more general domains. The half-plane
is convenient because we have quite explicit formulas for the Biot--Savart
law and the Stokes semigroup, but from a broader perspective the questions
we investigate do not really depend on the shape of the domain. Note that,
in doubly connected domains such as exterior domains, the velocity is
not uniquely determined by the vorticity; in particular, its circulation
at infinity may be chosen arbitrarily. Therefore, what should be understood
by ``a point vortex in an exterior domain'' requires some clarification. 
This problem does not occur in the half-plane, because the boundary
condition \eqref{noslip} implies that the integral of the vorticity
vanishes. 

\subsection{Outline of the paper}\label{ss:organization}

The rest of this paper is organized as follows. In the preliminary
Section~\ref{sec2}, we prove accurate estimates on the kernel of the Stokes
semigroup $S(t)$, and we study the Biot--Savart formula \eqref{BSlaw} which
defines the velocity in terms of the vorticity. Part of this material can be
found in earlier works such as \cite{Abe}. Section~\ref{sec3} is devoted to the
linearized equation at the Lamb--Oseen vortex, which can be studied in the whole
plane $\R^2$ using self-similar variables. Here we follow quite closely the approach
introduced in \cite{GW1,GG}, but the main result
(Proposition~\ref{prop:semigroup}) is not contained in previous references.  The
core of the paper is Section~\ref{sec4}, where we formulate integral equations
for the quantities $\hat\omega_1(t) = \omega_1(t) - \alpha S_1(t)\delta_z$ and
$\omega_2(t)$, and subsequently solve them by a fixed point argument to obtain
the main conclusion of Theorem~\ref{thm:main}. Qualitative properties of the solution
are established in Section~\ref{sec5}, including scale invariant bounds and
localization in the vertical direction. We also show that the solution has
finite energy for any $t > 0$, and is concentrated for small times near the
initial position $z$ of the point vortex, or near the boundary. This allows us
to derive the formula \eqref{Vasym} for the speed of the vortex center near
initial time. Finally, Appendix~\ref{appA} presents a short derivation of the
explicit formula for the integral kernel of the Stokes semigroup,
Appendix~\ref{appB} collects a few properties of the Stokes solution
$\omega(t) = S(t)\delta_z$, and Appendix~\ref{appC} states and proves
an elementary Gr\"onwall-type lemma.

\bigskip\noindent{\bf Acknowledgements.}
This work was supported by the BOURGEONS project, grant ANR-23-CE40-0014-01 of
the French National Research Agency (ANR). Both authors acknowledge the support
of the Institut Universitaire de France.

%%%%%%%%%%%%%%%%%%%%%%%%%%%%%%%%%%%%%%%%%%%%%%%%%%%%%%%%%%%%%%%%%%%%%%%%%%%
%%%%%%%%%%%%%%%%%%%%%%%%%%%%%%%%%%%%%%%%%%%%%%%%%%%%%%%%%%%%%%%%%%%%%%%%%%%

\section{The Stokes semigroup and the Biot--Savart law}\label{sec2}

This section is mainly devoted to the proof of Proposition~\ref{prop:Stokes}.
We first obtain accurate estimates on the boundary term $K_0$ defined by
\eqref{K0def}, and we use them to derive bounds on the semigroups $S_1(t)$, 
$S_2(t)$ defined in \eqref{S12def}. We also prove estimates for the
Biot--Savart law in the half-space. 

\subsection{Estimates of the boundary term}\label{ss:boundary}

\begin{lem}\label{lem:K0bounds}
There exists a constant $C > 0$ such that, for all $x,y \in \R^2_+$ and
all $t > 0$, 
\begin{align}\label{Kbd1}
  0 \,< \, K_0(x,y,t) \,&\le\, \frac{C}{\sqrt{t}}\,e^{-x_2^2/(5t)}\,
  \frac{y_2}{(x_1-y_1)^2 + (y_2 + \sqrt{t})^2}\,,\\ \label{Kbd2}
  |\nabla_y K_0(x,y,t)| \,&\le\, \frac{C}{\sqrt{t}}\,e^{-x_2^2/(5t)}\,
  \frac{1}{(x_1-y_1)^2 + (y_2 + \sqrt{t})^2}\,.
\end{align}
\end{lem}

\begin{rem}\label{rem:Kbd}
Estimates \eqref{Kbd1}, \eqref{Kbd2} are similar to some of the bounds
established in \cite{Abe}, see in particular Lemma 3.4 therein. But they
are not stated in the same form in \cite{Abe}, nor in \cite{Ukai}. Since
our strategy heavily relies on the structure of the right-hand sides
of \eqref{Kbd1} and \eqref{Kbd2}, we provide here a complete proof of
these estimates.
\end{rem}

\begin{proof}
Starting from the definition \eqref{K0def}, it is easy to verify that the
kernel $K_0(x,y,t)$ satisfies the following {\em scaling invariance}\:
\[
  t\,K_0\bigl(x\sqrt{t},y\sqrt{t},t\bigr) \,=\, K_0(x,y,1)\,, \qquad
  \forall\, x,y \in \R^2_+\,, \quad \forall\, t > 0\,.
\]
As a consequence, it is sufficient to establish \eqref{Kbd1}, \eqref{Kbd2}
for $t = 1$. To do so, we denote
\begin{equation}\label{GamP}
  g(x_1) \,=\, \frac{1}{\sqrt{4\pi}}\,e^{-x_1^2/4}\,, \qquad
  P(x_1,x_2) \,=\, \frac{1}{\pi}\,\frac{x_2}{x_1^2 + x_2^2}\,, \qquad
  \forall\,x_1 \in \R\,, \quad \forall\, x_2 > 0\,, 
\end{equation}
and we observe that 
\begin{align}\nonumber
  K_0(x,y,1) \,&=\, -2\partial_{x_2}\int_{\R}\int_0^{y_2} g(x_1 - z_1) 
  g(x_2 + z_2)\,P(y-z)\dd z_2 \dd z_1 \\ \label{K0rep}
 \,&=\, -2\int_0^{y_2} g'(x_2 + z_2) V(x_1-y_1,y_2-z_2)\dd z_2\,,
\end{align}
where $V : \R^2_+ \to \R_+$ is the {\em Voigt function} defined by
\begin{equation}\label{Voigt}
  V(x) \,=\, \int_\R g(x_1-z_1)P(z_1,x_2)\dd z_1 \,=\, \frac{1}{\sqrt{4\pi}}
  \,\Re \Bigl(w\Bigl(\frac{x_1 + i x_2}{2}\Bigr)\Bigr)\,, \qquad x \in \R^2_+\,.
\end{equation}
The last member of \eqref{Voigt} involves the {\em Fadeeva function}
\begin{equation}\label{Fadeeva}
  w(z) \,=\, \frac{i}{\pi} \int_\R \frac{e^{-t^2}}{z-t}\,\dd t \,=\, e^{-z^2}
  \mathrm{erfc}(-i z)\,, \qquad z \in \C\,, \quad \Im(z) > 0\,,
\end{equation}
the properties of which are described in \cite[Section~7.1]{AS}. It follows
from \eqref{Voigt}, \eqref{Fadeeva} that the Voigt function extends to a smooth
map $V : \R^2 \to \R$, and since $1 - \mathrm{erfc}(-iz) \in i\R$ when $z \in \R$
we have $V(x_1,0) = g(x_1)$ for all $x_1 \in \R$. On the other hand, using
the relation $w'(z) = -2z w(z)+ 2i/\sqrt{\pi}$ and the asymptotic expansion of
the complementary error function given in \cite[Formula~7.1.23]{AS}, it is easy
to verify that
\begin{equation}\label{Fadexp}
  w(z) \,=\, \frac{i}{\sqrt{\pi}z}\Bigl(1 + \frac{1}{2z^2} + \cO\Bigl(
  \frac{1}{|z|^4}\Bigr)\Bigr)\,, \qquad
  w'(z) \,=\, \frac{i}{\sqrt{\pi}}\Bigl(-\frac{1}{z^2} + \cO\Bigl(
  \frac{1}{|z|^4}\Bigr)\Bigr)\,,
\end{equation}
as $|z| \to \infty$ with $\Im(z) \ge 0$. In view of \eqref{Voigt}, this in turn
implies that
\begin{equation}\label{Voigtbd}
  0 \,\le\, V(x) \,\lesssim\, \frac{x_2+1}{x_1^2 + (x_2+1)^2}\,, \qquad 
  |\nabla V(x)| \,\lesssim\, \frac{1}{x_1^2 + (x_2+1)^2}\,, \qquad 
  \forall\, x \in \overline{\R^2_+}\,.
\end{equation}

We are now ready to prove the bounds \eqref{Kbd1}, \eqref{Kbd2} for
$t = 1$. We use the inequality
\begin{equation}\label{Gamp}
	-g''(x_2+z_2) \,=\, \frac{x_2+z_2}{4\sqrt{\pi}}\,e^{-(x_2+z_2)^2/4}
	\,\lesssim\, \frac{e^{-x_2^2/5}\,e^{-z_2^2/5}}{(1+x_2^2)(1+z_2^2)}
	\,\leq\, \frac{e^{-x_2^2/5}\,e^{-z_2^2/5}}{1+z_2^2}\,,
\end{equation}
in which the factor $1/5$ in the exponential has been chosen somewhat
arbitrarily and could be any number stricly less than $1/4$.  Using the
first estimate in \eqref{Voigtbd} together with \eqref{Gamp}, we deduce from
\eqref{K0rep} that
\[
 0 \,<\, K_0(x,y,1) \lesssim \, e^{-x_2^2/5} \int_0^{y_2}
  \frac{e^{-z_2^2/5}}{1+z_2^2}\,\frac{y_2-z_2+1}{(x_1-y_1)^2 + (y_2-z_2+1)^2}
  \dd z_2  .
\]
Noting that for $z_2\in [0, y_2]$,
\begin{align*}
  (1+z_2)^2\bigl((x_1-y_1)^2 + (y_2-z_2+1)^2\bigr) \,&\ge\, (x_1-y_1)^2 + (1+z_2)^2
  (y_2-z_2+1)^2\\ \,&\ge\,  (x_1-y_1)^2 + (1+y_2)^2\,,
\end{align*}
we infer that
\[
  K_0(x,y,1) \,\lesssim\, e^{-x_2^2/5} \int_0^{y_2}
  \frac{e^{-z_2^2/5}\,(y_2+1)}{(x_1-y_1)^2 + (y_2+1)^2}\dd z_2 
  \,\lesssim\, \frac{e^{-x_2^2/5} \,y_2}{(x_1-y_1)^2 + (y_2+1)^2}\,,
\]
for all $x,y \in \R^2_+$, which is the desired result.

Similarly, we deduce from \eqref{K0rep} that
\[
  |\nabla_y K_0(x,y,1)| \,\lesssim\, g(x_1-y_1)\,|g'(x_2+y_2)| + 
  \int_0^{y_2} |g'(x_2 + z_2)| \,|\nabla V(x_1-y_1,y_2-z_2)|\dd z_2\,.
\]
Using the second estimate in \eqref{Voigtbd} together with \eqref{Gamp},
we obtain
\begin{align*}
  |\nabla_y K_0(x,y,1)| \,&\lesssim\, e^{-(x_1-y_1)^2/4}\, e^{-(x_2^2+y_2^2)/5}
  + \int_0^{y_2} \frac{e^{-(x_2^2+z_2^2)/5}}{1+z_2^2}\,\frac{1}{(x_1-y_1)^2
  + (y_2-z_2+1)^2}\dd z_2 \\ \,&\lesssim\, \frac{e^{-x_2^2/5}}{(x_1-y_1)^2
  + (y_2+1)^2}\,, \qquad \forall\, x,y \in \R^2_+\,,
\end{align*}
which concludes the proof of \eqref{Kbd2}. 
\end{proof}

Estimate \eqref{Kbd1} shows in particular that the contribution of $K_0(x,y,t)$
to the semigroup kernel \eqref{Kdecomp} is concentrated for small times in a
neighborhood of size $\cO(\sqrt{t})$ of the boundary $x_2 = 0$. As is easily
verified, the same property holds for the ``mirror term'' $G_t(x-y^*)$, with
the slight difference that the latter does not vanish at $y_2 = 0$. As a
consequence, if we consider the integral kernel $K_2(x,y,t) = -G_t(x-y^*)
- K_0(x,y,t)$ introduced in \eqref{S12def}, the following estimates hold\:

\begin{cor}\label{cor:K1bounds}
There exists a constant $C > 0$ such that, for all $x,y \in \R^2_+$ and
all $t > 0$, 
\begin{align}\label{K2bd1}
    |K_2(x,y,t)| \,&\le\, \frac{C}{\sqrt{t}}\,e^{-x_2^2/(5t)}\,
  \frac{y_2 + \sqrt{t}}{(x_1-y_1)^2 + (y_2 + \sqrt{t})^2}\,,\\ \label{K2bd2}
  |\nabla_y K_2(x,y,t)| \,&\le\, \frac{C}{\sqrt{t}}\,e^{-x_2^2/(5t)}\,
  \frac{1}{(x_1-y_1)^2 + (y_2 + \sqrt{t})^2}\,.
\end{align}
\end{cor}

\begin{rem}\label{rem:dummy}
In contrast, the kernel $K_1(x,y,t) = G_t(x-y)$ is localized for small times
in the region where $|x-y| = \cO(\sqrt{t})$, which is away from the boundary
if $y_2$ is bounded from below. 
\end{rem}

In the spirit of Lemma~\ref{lem:K0bounds}, we now establish a decomposition
of the kernel $K_0(x,y,t)$ which will be used in the proof of 
Proposition~\ref{prop:Stokes}.

\begin{lem}\label{lem:K0decomp}
The kernel $K_0(x,y,t)$ defined in \eqref{K0def} satisfies
\begin{equation}\label{K0decomp}
  K_0(x,y,1) \,=\, 2 g(x_2)\,P(x_1-y_1,y_2) \,+\, \tilde K_0(x,y,1)\,,
  \qquad \forall\, x,y \in \R^2_+\,,
\end{equation}
where $\tilde K_0$ has the property that $\varphi(y_2) \lesssim (1+y_2)^{-1}$
for all $y_2 > 0$, where
\begin{equation}\label{vphidef}
  \varphi(y_2) \,:=\, \sup_{y_1 \in \R} \,\int_{\R^2_+} |\tilde K_0(x,y,1)|
  \dd x\,.
\end{equation}
\end{lem}

\begin{proof}
Integrating by parts in the formula \eqref{K0rep} and recalling that
$V(x_1-y_1,0) = g(x_1-y_1)$, we obtain
\begin{align*}
  K_0(x,y,1) \,&=\, 2g(x_2)V(x_1-y_1,y_2) - 2g(x_2+y_2)g(x_1-y_1)\\
  &\quad\, -2\int_0^{y_2} g(x_2 + z_2) \partial_{y_2}V(x_1-y_1,y_2-z_2)\dd z_2\,.
\end{align*}
We thus have the decomposition \eqref{K0decomp} with $\tilde K_0(x,y,1) = \sum_{j=1}^3
\tilde K_{0j}(x,y,1)$, where
\begin{equation}\label{tK0dec}
\begin{split}
  \tilde K_{01}(x,y,1) \,&=\, 2g(x_2)\Bigl(V(x_1-y_1,y_2) - P(x_1-y_1,y_2)\Bigr)\,,\\
  \tilde K_{02}(x,y,1) \,&=\, -2g(x_2+y_2)g(x_1-y_1) \,=\, -2G_1(x-y^*)\,,\\
  \tilde K_{03}(x,y,1) \,&=\, -2\int_0^{y_2} g(x_2 + z_2) \partial_{y_2}
  V(x_1-y_1,y_2-z_2)\dd z_2\,.
\end{split}
\end{equation}
In analogy with \eqref{vphidef} we define, for $j = 1,2,3$, 
\[
  \varphi_j(y_2) \,=\, \sup_{y_1 \in \R} \,\int_{\R^2_+} |\tilde K_{0j}(x,y,1)|
  \dd x\,, \qquad \forall\, y_2 > 0\,.
\]

Since $\int_\R V(x_1,x_2)\dd x_1 = \int_\R P(x_1,x_2)\dd x_1 = 1$ and
$\int_0^\infty g(x_2)\dd x_2 = 1/2$, it is clear that $\varphi_1(y_2) \le 1$
for all $y_2 > 0$. It also follows from \eqref{Voigt}, \eqref{Fadexp} that
\[
  \Bigl|V(x_1-y_1,y_2) - P(x_1-y_1,y_2)\Bigr| \,\lesssim\, \frac{1}{\bigl((x_1-y_1)^2 +
  y_2^2\bigr)^{3/2}}\,, \qquad \forall\, y_2 \ge 1\,,
\]
hence
\[
  \varphi_1(y_2) \,\lesssim\, \int_\R \frac{1}{\bigl(x_1^2 + y_2^2\bigr)^{3/2}}
  \dd x_1 \,=\, \cO\Bigl(\frac{1}{y_2^2}\Bigr)\,, \qquad \text{as }\,
  y_2 \to +\infty\,.
\]
On the other hand, we have
\[
  \varphi_2(y_2) \,=\, 2\int_0^\infty g(x_2+y_2)\dd x_2 \,=\,
  \mathrm{erfc}\bigl(y_2/2\bigr)\,, \qquad \forall\,y_2 > 0\,.
\]
Finally, using the second estimate in \eqref{Voigtbd}, we find
\[
  |\tilde K_{03}(x,y,1)| \,\lesssim\, \int_0^{y_2} \frac{g(x_2 + z_2)}{
  (x_1-y_1)^2 + (y_2-z_2+1)^2}\dd z_2\,,
\]
so that
\[
  \varphi_3(y_2) \,\lesssim\, \int_0^{y_2} \frac{\mathrm{erfc}\bigl(z_2/2\bigr)}{
  y_2 - z_2 + 1}\dd z_2 \,\lesssim\, \frac{y_2}{(y_2+1)^2}\,, \qquad \forall\,y_2 > 0\,.
\]
Since $\varphi(y_2) \le \sum_{j=1}^3 \varphi_j(y_2)$, the conclusion follows. 
\end{proof}

\subsection{Smoothing properties}\label{ss:smoothing}

The action of the operator $S_1(t)$ defined in \eqref{S12def} is
especially easy to describe\: for all initial data $\omega_0 \in L^1(\R^2_+)$
and all $t > 0$, we have
\begin{equation}\label{S1action}
  S_1(t)\omega_0 \,=\, \Bigl(G_t \,\star\, \overline{\omega}_0\Bigr)
  \,\Big|_{\R^2_+}\,,  
\end{equation}
where $G_t$ is the Gaussian function \eqref{GPdef}, $\star$ is the convolution
product in $\R^2$, and $\overline{\omega}_0$ is the extension by zero of
$\omega_0$ to the whole plane $\R^2$. Thus, using the standard $L^1$--$L^p$
estimates for the heat kernel (and its gradient) on $\R^2$, we immediately
obtain\:

\begin{lem}\label{lem:S1}
There exists a constant $C > 0$ such that, for all $p \in [1,\infty]$
and all $\omega_0 \in L^1(\R^2_+)$,
\begin{equation}\label{S1est}
  \|S_1(t)\omega_0\|_{L^p} \,+\, t^{1/2}\|S_1(t)\nabla \omega_0\|_{L^p}
  \,\le\, C\,\frac{\|\omega_0\|_{L^1}}{t^{1-1/p}}\,, \qquad \forall\, t > 0\,.
\end{equation}
\end{lem}

In view of Corollary~\ref{cor:K1bounds}, the correction term $S_2(t)$
actually satisfies a {\em stronger estimate} when $p > 1$ and when $\omega_0$
is supported away from the boundary. 

\begin{lem}\label{lem:Stokesbounds}
There exists $C > 0$ such that, for all $p \in [1,\infty]$, all $\omega_0 \in
L^1(\R^2_+)$, and all $t > 0$, 
\begin{equation}\label{S2est}
  \|S_2(t)\omega_0\|_{L^p} \,\le\, \int_{\R^2_+} \frac{C\,|\omega_0(y)|}{(t+y_2
    \sqrt{t})^{1-1/p}}\dd y\,, \quad
  \|S_2(t)\nabla\omega_0\|_{L^p} \,\le\, \int_{\R^2_+} \frac{C\,\sqrt{t}\,
    |\omega_0(y)|}{(t+y_2\sqrt{t})^{2-1/p}}\dd y\,.
\end{equation}
\end{lem}

\begin{proof}
According to \eqref{K2bd1} there exists a constant $C > 0$ such that,
for all $y \in \R^2_+$ and all $t > 0$, 
\begin{equation}\label{K2est}
  \int_{\R^2_+}|K_2(x,y,t)|\dd x \,\le\, C\,, \qquad \text{and}\qquad
  \sup_{x \in \R^2_+}|K_2(x,y,t)| \,\le\, \frac{C}{t + y_2\sqrt{t}}\,. 
\end{equation}
We fix $t > 0$, $p \in [1,+\infty]$, and we consider the weight function
$w_{p,t} : \R^2_+ \to (0,+\infty)$ defined by 
\begin{equation}\label{wpt}
  w_{p,t}(y,t) \,=\, \|K_2(\cdot,y,t)\|_{L^p(\R^2_+)} \,\le\, \frac{C}{
 (t + y_2\sqrt{t})^{1-1/p}}\,, \qquad \forall\,y \in \R_2^+\,,
\end{equation}
where the inequality follows from \eqref{K2est} by interpolation.
If $\omega_0 \in L^1(\R^2_+)$, we have by \eqref{S12def}
\[
  \bigl(S_2(t)\omega_0\bigr)(x) \,=\, \int_{\R^2_+} K_2(x,y,t)\omega_0(y)\dd y
  \,=\, \int_{\R^2_+} \frac{K_2(x,y,t)}{w_{p,t}(y)}\,w_{p,t}(y)\omega_0(y)\dd y\,,
  \qquad \forall x \in \R^2_+\,.
\]
Applying Minkowski's integral inequality, we deduce that
\begin{equation}\label{Minko}
  \|S_2(t)\omega_0\|_{L^p} \,\le\, \int_{\R^2_+} |w_{p,t}(y)\omega_0(y)|\dd y\,,
  \qquad\text{because}\qquad \biggl\|\frac{K_2(\cdot,y,t)}{w_{p,t}(y)}
  \biggr\|_{L^p} \,=\, 1\,. 
\end{equation}
The estimate \eqref{S2est} for $S_2(t)\omega_0$ is now a direct consequence
of \eqref{wpt}, \eqref{Minko}. Since
\[
  \int_{\R^2_+}|\nabla_y K_2(x,y,t)|\dd x \,\le\, \frac{C}{y_2 + \sqrt{t}}\,, \qquad 
  \sup_{x \in \R^2_+}|\nabla_y K_2(x,y,t)| \,\le\, \frac{C}{\sqrt{t}\,(y_2 + \sqrt{t})^2}\,,
\]
the bound \eqref{S2est} on $S_2(t)\nabla\omega_0$ is obtained by a similar argument.
\end{proof}

\begin{rem}\label{rem:Stokes2} Estimate \eqref{Stokes2} for $S(t) = S_1(t) + S_2(t)$
follows immediately from \eqref{S1est}, \eqref{S2est}. More generally, Schur's bound
for integral operators shows that
\begin{equation}\label{Schur}
  \|S(t)\omega_0\|_{L^q} \,\le\, \sup_{x \in \R^2_+}\|K(x,\cdot,t)\|_{L^p}^{1-p/q}
  \,\sup_{y \in \R^2_+}\|K(\cdot,y,t)\|_{L^p}^{p/q} ~\|\omega_0\|_{L^r}\,,
\end{equation}
whenever $p,q,r \in [1,\infty]$ and $1 + 1/q = 1/p + 1/r$. However, as can be seen
from \eqref{Kbd1}, the quantity $\|K(x,\cdot,t)\|_{L^p}$ is finite if and
only if $p > 2$. So, besides the case $r = 1$ considered in \eqref{Stokes2}, 
estimate \eqref{Schur} shows that $S(t)$ maps $L^r(\R^2_+)$ into $L^q(\R^2_+)$
if $1 < r < 2$ and $1/q < 1/r - 1/2$, see also \cite[Remark~4.2]{Abe}. A similar
remark applies to the operator $S(t)\nabla$, with the important difference that
$\|\nabla_y K(x,\cdot,t)\|_{L^p}$ is finite if and only if $p > 1$. Thus $S(t)\nabla$
maps $L^r(\R^2_+)$ into $L^q(\R^2_+)$ if $q > r$. 
\end{rem}

\subsection{Continuity properties}\label{ss:continuity}

The following result is the second key step in the proof of Proposition~\ref{prop:Stokes}.
We recall that $L^1_\perp(\R^2_+)$ is the closed subspace of $L^1(\R^2_+)$
defined by \eqref{L1perp}. 

\begin{lem}\label{lem:L1cont}
If $\omega_0 \in L^1_\perp(\R^2_+)$, then $\|S(t)\omega_0 - \omega_0\|_{L^1} \to 0$
as $t \to 0$.
\end{lem}

\begin{proof}
As in the proof of Lemma~\ref{lem:Skernel} we write $S(t) = e^{t\Delta_D} - S_0(t)$,
where $e^{t\Delta_D}$ is the heat semigroup in $\R^2$ with Dirichlet boundary
condition, and $S_0(t)$ is the operator defined by \eqref{SOdef}.
It is well known that $e^{t\Delta_D}$ is a $C_0$-semigroup in $L^p(\R^2_+)$
for any $p \in [1,+\infty)$. In particular, for any $\omega_0 \in L^1(\R^2_+)$,
we have $\|e^{t\Delta_D}\omega_0 - \omega_0\|_{L^1} \to 0$ as $t \to 0$. To prove
Lemma~\ref{lem:L1cont}, it remains to verify that $\|S_0(t)\omega_0\|_{L^1} \to 0$
as $t \to 0$ if we assume in addition that $\omega_0 \in L^1_\perp(\R^2_+)$. 

Using Lemma~\ref{lem:K0decomp} and the scaling invariance, we decompose
the kernel \eqref{K0def} as
\begin{equation}\label{newdec}
  K_0(x,y,t) \,=\, \frac{2}{\sqrt{t}}\,g\Bigl(\frac{x_2}{\sqrt{t}}\Bigr)
  \,P(x_1-y_1,y_2) \,+\, \frac{1}{t}\, \widetilde{K}_0\Bigl(\frac{x}{\sqrt{t}}\,,
  \frac{y}{\sqrt{t}}\,,1\Bigr)\,. 
\end{equation}
If $\omega_0 \in L^1_\perp(\R^2_+)$, the definition \eqref{gammadef} shows that
\[
  \bigl(\gamma[\omega_0]\bigr)(x_1) \,=\, \int_{\R^2_+} P(x_1-y_1,y_2) \omega_0(y)
  \dd y \,=\, 0\,,
\]
for almost all $x_1 \in \R$. We can thus drop the first term in
the right-hand side of \eqref{newdec}, and we obtain the simpler
representation
\[
  \bigl(S_0(t)\omega_0\bigr)(x) \,=\, \frac{1}{t} \int_{\R^2_+} \widetilde{K}_0
  \Bigl(\frac{x}{\sqrt{t}}\,,\frac{y}{\sqrt{t}}\,,1\Bigr)\omega_0(y)\dd y\,,
  \qquad \forall\, x \in \R^2_+\,, \quad \forall\, t > 0\,.
\]
In particular, using the definition \eqref{vphidef}, we find
\[
  \|S_0(t)\omega_0\|_{L^1} \,\le\, \int_{\R^2_+} \Bigl\|\widetilde{K}_0\Bigl(\cdot\,,
  \frac{y}{\sqrt{t}},1\Bigr)\Bigr\|_{L^1} |\omega_0(y)|\dd y \,\le\, \int_{\R^2_+}
  \varphi\Bigl(\frac{y_2}{\sqrt{t}}\Bigr)|\omega_0(y)|\dd y\,, 
\]
for all $t > 0$. Since $\varphi(y_2) \lesssim (1+y_2)^{-1}$, Lebesgue's
dominated convergence theorem shows that $\|S_0(t)\omega_0\|_{L^1} \to 0$ as
$t\to 0$.
\end{proof}

\begin{proof}[\bf End of the proof of Proposition~\ref{prop:Stokes}]
By \eqref{Stokes2} the Stokes operator $S(t)$ defined by the integral formula
\eqref{Stokes} is uniformly bounded in $L^1(\R^2_+)$. Moreover, by construction,
$S(t)$ maps $L^1(\R^2_+)$ into the closed subspace $L^1_\perp(\R^2_+)$ for $t > 0$,
see Appendix~\ref{appA}, and the restriction of $S(t)$ to that subspace defines
a $C_0$-semigroup by Lemma~\ref{lem:L1cont}. This concludes the proof of item 1) in
Proposition~\ref{prop:Stokes}

On the other hand, as was observed in Remark~\ref{rem:Stokes2}, estimate
\eqref{Stokes2} follows from \eqref{S1est} and \eqref{S2est}. To prove that
$t \mapsto S(t)\omega_0$ belongs to $C^0((0,+\infty),L^p(\R^2_+))$ for any $p > 1$,
we fix $t_0 > 0$ and we observe that, for any $\tau > 0$,
\begin{align*}
  \bigl\|S(t_0+\tau)\omega_0 - S(t_0)\omega_0\bigr\|_{L^p} \,&=\,
  \bigl\|S(t_0/2)(S(\tau) - \1)S(t_0/2)\omega_0\bigr\|_{L^p} \\ \,&\le\,
  \frac{C}{t_0^{1-1/p}}\bigl\|(S(\tau) - \1\bigr)S(t_0/2)\omega_0\bigr\|_{L^1}
  \,\xrightarrow[\tau \to 0]{}\, 0\,,
\end{align*}
because $S(t_0/2)\omega_0 \in L^1_\perp(\R^2_+)$. This shows that $S(t)\omega_0$
is continuous to the right at $t = t_0$, and a similar argument gives
the continuity to the left too, as well as the continuity of $S(t)\nabla\omega_0$.
This concludes the proof of Proposition~\ref{prop:Stokes}.
\end{proof}

As explained in \cite{Abe}, Lemma~\ref{lem:L1cont} can be extended to initial data
$\omega_0 \in \mathcal M_\perp$, where $\mathcal M_\perp$ denotes the space of finite
measures on $\R^2_+$ satisfying the analogue of the condition $\gamma[\omega_0] = 0$. 
Weaker continuity properties can be established for initial measures which
do not satisfy that condition\:

\begin{lem}\label{lem:Mcont}
Let $z=(0,1)$. Then $S(t)\delta_z\rightharpoonup \delta_z$ in $\mathcal M(\R^2_+)$
as $t \to 0$. 
\end{lem}

\begin{proof}
Let $\phi\in C_0(\R^2_+)$. We extend $\phi$ by zero on the lower half-plane;
the extension, still denoted by $\phi$, is a continuous function in $\R^2$,
vanishing at infinity. Then
\begin{align*}
\langle (S(t) \delta_z - \delta_z, \phi\rangle\; &=\; \int_{\R^2_+} K(x,z,t) \phi(x)\dd x - \phi(z)\\
&=\; G_t\star \phi(z) - \phi(z) - G_t\star \phi (z^*) - \int_{\R^2_+} K_0(x,z,t) \phi(x)\dd x\,.
\end{align*}
Classical convolution results ensure that $G_t\star \phi(z) - \phi(z)\to 0$,
$G_t\star \phi (z^*)\to \phi(z^*)=0$ as $t\to 0$.  As for the last term, we use
the bound \eqref{Kbd1} and we obtain, for $t\leq 1$,
\begin{align*}
  \left| \int_{\R^2_+} K_0(x,z,t) \phi(x)\dd x\right| \; &\lesssim \; \frac{1}{\sqrt{t}}
  \int_{\R^2_+} e^{-x_2^2/(5t)} \frac{1}{x_1^2 + 1} |\phi(x_1,x_2)|\dd x\\
  &\lesssim \; \int_{\R^2_+} e^{-X_2^2/5} \frac{1}{x_1^2 + 1} |\phi(x_1,\sqrt{t}X_2)|
  \dd x_1\dd X_2\,,
\end{align*}
and the right-hand side vanishes as $t\to 0$ since $\phi(x_1,0)=0$ for all $x_1\in \R$
by assumption.
\end{proof}

\begin{rem}
We bring the reader's attention to the fact that the notion of weak
convergence used here is different from the one in the paper of Abe
\cite{Abe}: in the latter, the space of measures is the dual space of
functions vanishing at infinity, but not necessarily on the boundary. 
As a consequence, the statement of Lemma \ref{lem:Mcont} is different from
Lemma 4.1 in \cite{Abe}.
\end{rem}

\subsection{Estimates on the Biot--Savart law}\label{ss:BS}

Assume that $u = \BS[\omega]$ as in \eqref{BSlaw}. Since $|x-y| \le |x-y^*|$
for all $x,y \in \R^2_+$, it is clear that
\begin{equation}\label{BSR2}
  |u(x)| \,\le\, \frac{1}{\pi}\int_{\R^2_+}\frac{1}{|x-y|}\,|\omega(y)|
  \dd y\,, \qquad \forall\, x \in \R^2_+\,.
\end{equation}
We thus have the following classical estimates\:

\begin{lem}\label{lem:HLS} {\rm (Biot--Savart estimates)}
\begin{enumerate}[leftmargin=15pt,itemsep=2pt,topsep=1pt]

\item Assume that $1 < p < 2 < q < \infty$ and $1/q = 1/p - 1/2$. If
$\omega \in L^p(\R^2_+)$, then $u \in L^q(\R^2_+)$ and there exists
a constant $C > 0$ such that $\|u\|_{L^q} \le C \|\omega\|_{L^p}$.

\item If $\omega \in L^p(\R^2_+) \cap L^q(\R^2_+)$ for some $p \in [1,2)$
and some $q = (2,+\infty]$, then $u \in L^\infty(\R^2_+)$ and there exists
a constant $C > 0$ such that $\|u\|_{L^\infty} \le C \|\omega\|_{L^p}^\theta
\|\omega\|_{L^q}^{1-\theta}$ where $1/2 = \theta/p + (1-\theta)/q$. 

\end{enumerate}
\end{lem}

\begin{proof}
Extending $\omega$ by zero outside $\R^2_+$, we can integrate over $y \in \R^2$
in \eqref{BSR2}, and using the classical Hardy-Littlewood-Sobolev inequality, we
obtain the estimate $\|u\|_{L^q} \le C \|\omega\|_{L^p}$ when $1/q = 1/p - 1/2$.
To prove the second part, we fix $x \in \R^2_+$, and given $R > 0$ we split the
integral in \eqref{BSR2} into two parts, according to whether $|x-y| \le R$ or
$|x-y| > R$.  We easily obtain the estimate $|u(x)| \le CR^{1-2/q}\|\omega\|_{L^q}
+ CR^{1-2/p}\|\omega\|_{L^p}$, and optimizing over $R > 0$ gives the desired result.
\end{proof}

The next statement takes into account some cancellations
between the contributions of the point $y$ and the mirror image
$y^*$ in the Biot--Savart kernel.

\begin{lem}\label{lem:BS}
If $u = \BS[\omega]$, there exists a constant $C > 0$ such that
\begin{equation}\label{BSimproved}
  |u(x)| \,\le\, C\int_{\R^2_+}\frac{y_2}{|x-y|\,|x-y^*|}\,|\omega(y)|
  \dd y\,, \qquad \forall\, x \in \R^2_+\,.
\end{equation}
In particular, for any $p \in (1,2)$, there exists a constant $C_p > 0$ such that
\begin{equation}\label{uLp}
  \|u\|_{L^p} \,\le\, C_p \int_{\R^2_+} y_2^{\theta_p} \,|\omega(y)|\dd y\,,
  \qquad \text{where}\quad \theta_p \,=\, \frac{2}{p}-1\,.
\end{equation}
\end{lem}

\begin{proof}
It is straightforward to verify that the Biot--Savart kernel in $\R^2_+$
satisfies
\[
  \frac{1}{2\pi}\biggl(\frac{(x-y)^\perp}{|x-y|^2} - \frac{(x-y^*)^\perp}{
  |x-y^*|^2}\biggr)  \,=\, \frac{1}{\pi}\,\frac{y_2}{|x-y|^2 |x-y^*|^2}
  \begin{pmatrix} (x_1 - y_1)^2 - (x_2^2 - y_2^2) \\  2 x_2(x_1-y_1)
  \end{pmatrix}\,.
\]
Since $\max(|x_1-y_1|,|x_2-y_2|) \le |x-y| \le |x-y^*|$ and $x_2 \le x_2+y_2
\le |x-y^*|$, we immediately obtain the estimate \eqref{BSimproved}.

Next we fix $p \in (1,2)$ and, for any $y \in \R^2_+$, we consider the function
$\phi_y$ defined by
\[
  \phi_y(x) \,=\, \frac{y_2^{1-\theta_p}}{|x-y|\,|x-y^*|} \,\le\, \frac{y_2^{1-\theta_p}}{
  |x-y|\,(x_2+y_2)}\,, \qquad \forall\,x \in \R^2_+\,,\quad x \neq y\,.
\]
A straightforward calculation shows that, for all $y \in \R^2_+$,  
\begin{align*}
  \int_{\R^2_+}\phi_y(x)^p\dd x \,&\le\, \int_0^\infty \frac{y_2^{2p-2}}{(x_2+y_2)^p}
  \,\biggl\{\int_\R \frac{1}{|x-y|^p}\dd x_1\biggr\} \dd x_2 \\
  \,&=\, \int_0^\infty \frac{\,y_2^{2p-2}}{(x_2+y_2)^p}\,\frac{C_p}{|x_2-y_2|^{p-1}}\dd x_2
   \,=\, C_p\,,
\end{align*}
where $C_p$ denotes a generic positive constant depending only on $p$. 
This shows that the family $(\phi_y)_{y \in \R^2_+}$ is uniformly bounded in $L^p(\R^2_+)$.
Now we deduce from \eqref{BSimproved} that
\[
  |u(x)| \,\le\, C\int_{\R^2_+} \phi_y(x)\,y_2^{\theta_p}\,|\omega(y)|\dd y\,, \qquad
  \forall\,x \in \R^2_+\,,
\]
and applying Minkowski's integral inequality we obtain
\[
  \|u\|_{L^p} \,\le\, C \Bigl\{\,\sup_{y \in \R^2_+}\|\phi_y\|_{L^p}\Bigr\}
  \int_{\R^2_+} y_2^{\theta_p}\,|\omega(y)|\dd y \,\le\,  C_p
  \int_{\R^2_+} y_2^{\theta_p}\,|\omega(y)|\dd y\,,
\]
which is \eqref{uLp}. 
\end{proof}

Finally, we observe that, for vorticities satisfying the no-slip
condition \eqref{noslip}, the Biot--Savart formula \eqref{BSlaw}
actually reduces to the usual formula \eqref{BS2d} in the whole space
$\R^2$.

\begin{lem}\label{lem:2BS}
Assume that $\omega \in L^1_\perp(\R^2_+) \cap L^{4/3}(\R^2_+)$ and let $u = \BS[\omega]$.
Then
\begin{equation}\label{BSid}
  u(x) \,=\, \frac{1}{2\pi}\int_{\R^2_+} \frac{(x-y)^\perp}{|x-y|^2}
  \,\omega(y)\dd y\,, \qquad \text{for almost all } x \in \R^2_+\,. 
\end{equation}
In particular $\DS\int_{\R^2_+} \! u(x)\,\omega(x)\dd x = 0$. 
\end{lem}

\begin{proof}
Let $\widetilde u : \R^2 \to \R^2$ be the extension of $u$ by zero to the
whole plane $\R^2$, and $\widetilde\omega : \R^2 \to \R$ the extension by
zero of $\omega$. Since by assumption $\widetilde u$ vanishes on the
boundary $\partial\R^2_+$, we have $\div\widetilde u = 0$ and
$\curl \widetilde u = \widetilde \omega$ in $\R^2$. We also know
that $\widetilde\omega \in L^1(\R^2) \cap L^{4/3}(\R^2)$, so we conclude
that $\widetilde u = \BSS[\widetilde \omega]$, which gives \eqref{BSid} after
restriction  to $\R^2_+$. 

Now, since $\omega \in L^{4/3}(\R^2_+)$, we have $u \in L^4(\R^2_+)$ by
Lemma~\ref{lem:HLS}, hence $u\,\omega \in L^1(\R^2_+)$, and Fubini's
theorem implies that 
\[
  \int_{\R^2_+} u(x)\,\omega(x)\dd x \,=\, \frac{1}{2\pi} \int_{\R^2_+\times \R^2_+}
  \frac{(x-y)^\perp}{|x-y|^2}\,\omega(x)\, \omega(y)\dd x\dd y \,=\, 0\,,
\]
because the last integrand is an antisymmetric function of its arguments
$x,y$. 
\end{proof}

%%%%%%%%%%%%%%%%%%%%%%%%%%%%%%%%%%%%%%%%%%%%%%%%%%%%%%%%%%%%%%%%%%%%%%%%%%%
%%%%%%%%%%%%%%%%%%%%%%%%%%%%%%%%%%%%%%%%%%%%%%%%%%%%%%%%%%%%%%%%%%%%%%%%%%%

\section{The linearized evolution at the Lamb--Oseen vortex}\label{sec3}

As is explained in the introduction, our strategy for proving Theorem~\ref{thm:main}
relies crucially on a decomposition of the solution into an approximation of
the Lamb--Oseen vortex and a correction term that is concentrated near the boundary.
It is thus natural to consider the linearized equation at the Lamb--Oseen vortex,
which was already studied in previous works, including \cite{GG,GW1,GW2}. In
this section, we extend the existing results and derive long-time estimates
for the solutions of the linearized equation in the space $L^1(\R^2)$. 

Given $\alpha \in \R$ and $z = (0,1) \in \R^2_+$, we denote by $\bar\omega_1(x,t)$
and $\bar u_1(x,t)$ the vorticity and the velocity field of the Lamb--Oseen
vortex with circulation $\alpha$ centered at point $z$, namely
\begin{equation}\label{def:Oseen}
  \bar\omega_1(x,t) \,=\, \frac{\alpha}{t}\,G\Bigl(\frac{x-z}{\sqrt{t}}\Bigr)\,, \qquad
  \bar u_1(x,t) \,=\, \frac{\alpha}{\sqrt{t}}\,v^G\Bigl(\frac{x-z}{\sqrt{t}}\Bigr)\,,
  \qquad x \in \R^2\,, \quad t > 0\,. 
\end{equation}
Here the vorticity profile $G$ and the velocity profile $v^G$ are given by
\begin{equation}\label{GvGdef}
  G(\xi) \,=\, G_1(\xi) \,=\, \frac{1}{4\pi}\,e^{-|\xi|^2/4}\,, \qquad
  v^G(\xi) \,=\, \frac{1}{2\pi}\,\frac{\xi^\perp}{|\xi|^2}\Bigl(
  1 - e^{-|\xi|^2/4}\Bigr)\,, \qquad \xi \in \R^2\,.
\end{equation}
For later use we observe that $v^G = \BSS[G]$, where $\BSS$ denotes the Biot--Savart
operator in the whole space $\R^2$\:
\begin{equation}\label{BS2d}
  \BSS[w](\xi) \,=\, \frac{1}{2\pi}\int_{\R^2} \frac{(\xi-\eta)^\perp}{
  |\xi-\eta|^2}\,w(\eta)\dd \eta\,, \qquad \xi \in \R^2\,.
\end{equation}
Note that \eqref{BS2d} differs from the Biot--Savart formula \eqref{BSlaw},
which holds in the half-plane $\R^2_+$. 

In what follows we consider initial data $\omega_0 \in L^1(\R^2)$ at time
$t_0 > 0$, and we analyze the Cauchy problem for the linear non-autonomous equation
\begin{equation}\label{LinGam1}
  \partial_t \omega(t) + \bar u_1(t)\cdot\nabla\omega(t) 
  + u(t) \cdot\nabla\bar\omega_1(t) \,=\, \Delta\omega(t) \,,
  \qquad x \in \R^2\,, \quad t \ge t_0\,,
\end{equation}
with initial condition $\omega(t_0) = \omega_0$. We emphasize that \eqref{LinGam1}
is to be solved {\em in the whole space} $\R^2$, and that the velocity field
$u = \BSS[\omega]$ is given by the Biot--Savart formula \eqref{BS2d}. Since
$\bar u_1(t)$ and $\bar\omega_1(t)$ are smooth and bounded for all $t \ge t_0$,
classical results for linear parabolic equations assert that \eqref{LinGam1} has
a unique solution $\omega \in C^0([t_0,+\infty),L^1(\R^2))$. We write
\begin{equation}\label{Sigalphdef}
  \omega(t) \,=\, \Sigma_\alpha(t,t_0)\omega_0\,, \qquad \forall\,t \ge t_0\,,
\end{equation}
where the operator $\Sigma_\alpha(t,t_0)$ is, by definition, the two-parameter
semigroup associated with the linear equation \eqref{LinGam1}.

Following \cite{GW1,GW2}, we introduce the self-similar variables
\begin{equation}\label{ssvar}
  \xi \,=\, \frac{x-z}{\sqrt{t}}\,, \qquad \tau \,=\, \log\frac{t}{t_0}\,.
\end{equation}
The vorticity and the velocity are transformed according to
\begin{equation}\label{sschange}
  \omega(x,t) \,=\, \frac{1}{t}\,w\Bigl(\frac{x-z}{\sqrt{t}}\,,\,\log
  \frac{t}{t_0}\Bigr)\,, \qquad
  u(x,t) \,=\, \frac{1}{\sqrt{t}}\,v\Bigl(\frac{x-z}{\sqrt{t}}\,,\,\log
  \frac{t}{t_0}\Bigr)\,.  
\end{equation}
Since $\bar \omega_1$ and $\bar u_1$ have the self-similar form given
by \eqref{def:Oseen}, it is easy to verify that the rescaled vorticity
$w(\xi,\tau)$ and velocity $v(\xi,\tau)$ satisfy the relation $v = \BSS[w]$ and
evolve according to the autonomous equation
\begin{equation}\label{LinGam2}
  \partial_\tau w(\xi,\tau) + \alpha\Bigl(v^G \cdot \nabla w +
  v \cdot \nabla G\Bigr)(\xi,\tau) \,=\, \Bigl(\Delta w +
  \frac12\, \xi \cdot\nabla w + w\Bigr)(\xi,\tau)\,,
\end{equation}
with initial condition $w(\xi,0) = t_0 \omega_0\bigl(z + \sqrt{t_0}\xi\bigr)$. 
Here it is understood that the differential operators $\Delta$ and $\nabla$
act on the new space variable $\xi$. Introducing the diffusion operator $\cL$
and the (nonlocal) advection operator $\Lambda$ defined by
\begin{equation}\label{LLamdef}
  \cL \,=\, \Delta_\xi + \frac12\,\xi\cdot\nabla_\xi + 1\,, \qquad
  \Lambda w \,=\, v^G \cdot \nabla w + \BSS[w] \cdot \nabla G\,,
\end{equation}
we can write \eqref{LinGam2} in the compact form $\partial_\tau w = (\cL -
\alpha\Lambda)w$, which is convenient for analysis.

In the rest of this section we first derive accurate estimates on the semigroup
$\cS_\alpha(\tau)$ associated with the evolution equation \eqref{LinGam2},
which are collected in Proposition~\ref{prop:semigroup}. Then, returning to
the original variables, we deduce the corresponding estimates for the
two-parameter semigroup $\Sigma_\alpha(t,t_0)$, see
Corollary~\ref{cor:Sigma-alpha}.

\begin{prop}\label{prop:semigroup}
For any $\alpha \in \R$, the linear operator $\cL - \alpha\Lambda$
generates a strongly continuous semigroup $\cS_\alpha(\tau)$ in
$L^1(\R^2)$ which satisfies the following estimates. For any $\kappa > 0$
and any $p \in [1,2)$, there exists a constant $C_0 = C_0(\alpha,\kappa,p)$
such that, for all $\tau > 0$ and all $w_0 \in L^1(\R^2)$, 
\begin{equation}\label{sgest}
  \|\cS_\alpha(\tau)w_0\|_{L^p} \,\le\, \frac{C_0\,e^{\kappa \tau}}{
    a(\tau)^{1-1/p}}\,\|w_0\|_{L^1}\,, \qquad
  \|\cS_\alpha(\tau)\nabla w_0\|_{L^p} \,\le\, \frac{C_0\,e^{\left(\kappa
  -\frac{1}{2}\right) \tau}}{a(\tau)^{3/2-1/p}}\,\|w_0\|_{L^1}\,,
\end{equation}
where $a(\tau) = 1 - e^{-\tau}$. Moreover estimates \eqref{sgest} hold with
$\kappa = 0$ if $|\alpha|$ is small enough.  
\end{prop}

\noindent
The proof of Proposition~\ref{prop:semigroup} is quite involved, and
we divide it into four steps. 

\medskip\noindent{\bf Step 1\:} {\em The case where $\alpha = 0$.}\\
The diffusion operator $\cL$ is the generator of a strongly continuous
semigroup $\cS_0$ in $L^1(\R^2)$, given by the explicit representation
formula
\begin{equation}\label{sg0rep}
  \bigl(\cS_0(\tau)w_0\bigr)(\xi) \,=\, \frac{e^\tau}{4\pi a(\tau)}
  \int_{\R^2} e^{-\frac{|\xi-\eta|^2}{4a(\tau)}} w_0\bigl(\eta \,e^{\tau/2}\bigr)
  \dd \eta\,, \qquad \xi \in \R^2\,, \quad \tau > 0\,,
\end{equation}
where $a(\tau) = 1 - e^{-\tau}$, see \cite{GW1,GW2}. It follows that, 
for any $p \in [1,+\infty]$, 
\begin{equation}\label{sg0est1}
  \bigl\|\cS_0(\tau)w_0\bigr\|_{L^p} \,\le\, \frac{1}{\bigl(4\pi a(\tau)
  \bigr)^{1-1/p}}\,\|w_0\|_{L^1}\,, \qquad \tau > 0\,.
\end{equation}
Indeed, estimate \eqref{sg0est1} is a direct consequence of \eqref{sg0rep}
if $p = 1$ or $p = \infty$, and the general case follows by interpolation. 
Similarly, differentiating \eqref{sg0rep} with respect to $\xi$ and
observing that $\nabla \cS_0(t) = e^{\tau/2}\cS_0(\tau)\nabla$ , we obtain
the estimate
\begin{equation}\label{sg0est2}
  \bigl\|\nabla\cS_0(\tau)w_0\bigr\|_{L^p}  \,=\, e^{\tau/2}\bigl\|\cS_0(\tau)\nabla w_0
  \bigr\|_{L^p} \,\le\, \frac{C\,\|w_0\|_{L^1}}{a(\tau)^{3/2-1/p}}\,, \qquad \tau > 0\,.
\end{equation}
Altogether this proves \eqref{sgest} for $\alpha = 0$, with $\kappa = 0$. 

\medskip\noindent{\bf Step 2\:} {\em Short time estimates.}\\
We next construct a solution of \eqref{LinGam2} for small times.
Given $w_0 \in L^1(\R^2)$, $\alpha \in \R$, and $T > 0$, we consider the integral
equation
\begin{equation}\label{st0}
  w(\tau) \,=\, \cS_0(\tau)w_0 - \alpha \cF[w](\tau)\,, \qquad \tau \in [0,T]\,,  
\end{equation}
where $\cF$ is the linear map defined by
\begin{equation}\label{st1}
  \cF[w](\tau) \,=\, \int_0^\tau {\cS_0}(\tau-s)\,\Lambda w(s)\dd s 
  \,=\, \int_0^\tau {\cS_0}(\tau-s)\div\bigl(v^G w(s) + v(s) G\bigr)\dd s\,.
\end{equation}
We recall that $G, v^G$ are given by \eqref{GvGdef} and that $v(s) = \BSS[w(s)]$ as
in \eqref{BS2d}. For later use, we observe that 
\begin{align}\label{st2}
  \|v^G w(s) + v(s) G\|_{L^1} \,&\le\,  C\|w(s)\|_{L^1}\,, \\ \label{st3}
  \|v^G\cdot \nabla w(s) + v(s)\cdot \nabla G\|_{L^1} \,&\le\,  C\bigl(
  \|\nabla w(s)\|_{L^1}{+\| w(s)\|_{L^1}}\bigr)\,.
\end{align}
Indeed the terms $v^G w$ and $v^G\cdot \nabla w$ are straightforward to
bound since $v^G\in L^\infty(\R^2)$. On the other hand, we have the useful
estimate
\[
  \bigl\|\BSS[w]\bigr\|_{L^1(\R^2) + L^\infty(\R^2)} \,\le\, C\|w\|_{L^1(\R^2)}\,,
  \qquad \forall\, w \in L^1(\R^2)\,,
\]
which is easily obtained by decomposing the integral in \eqref{BS2d} in
two parts according to whether $|\eta-\xi| \le 1$ or $|\eta-\xi| > 1$.
Since both $G$ and $\nabla G$ belong to $L^1(\R^2) \cap L^\infty(\R^2)$,
we obtain the desired result. 

We now consider the map $\cF$ as acting on the space $X_T = C^0([0,T],L^1(\R^2))$
equipped with the norm
\[
  \|w\|_{1,T} \,=\, \sup_{\tau \in [0,T]} \|w(\tau)\|_{L^1}\,.
\]
From \eqref{sg0est1} we have $\| \cS_0(\cdot ) w_0\|_{1,T} \leq \| w_0\|_{L^1}$.
Using \eqref{sg0est2} and \eqref{st2} together with the monotonicity of
$s\mapsto a(s)$, we also find, for $\tau \in [0,T]$,
\begin{equation}\label{cFest}
\begin{split}
  \|\cF[w](\tau)\|_{L^1} \,&\le\, C \int_0^\tau \frac{e^{-(\tau-s)/2}}{a(\tau-s)^{1/2}}
  \,\bigl\|v^G w(s) + v(s) G\bigr\|_{L^1}\dd s \\ \,&\le\,
  C \,\|w\|_{1,T} \int_0^\tau \frac{e^{-s/2}}{a(s)^{1/2}}\dd s \,\le\, C_1\,a(\tau)^{1/2}
  \,\|w\|_{1,T}\,,
\end{split}
\end{equation}
so that $\|\cF[w]\|_{1,T} \le C_1\,a(T)^{1/2}\,\|w\|_{1,T}$ for some universal constant
$C_1 > 0$. If $T > 0$ is small enough so that $C_1 |\alpha| \,a(T)^{1/2} \le 1/2$, the integral
equation \eqref{st0} has a unique solution $w \in X_T$ which satisfies $\|w\|_{1,T} \le
2 \|w_0\|_{L^1}$. By construction this solution depends linearly on the initial vorticity
$w_0 \in L^1(\R^2)$, and will be denoted $w(\tau) = \cS_\alpha(\tau) w_0$. Note that $T > 0$
can be taken arbitrarily large if $C_1|\alpha| \le 1/2$. 

\begin{lem}\label{lem:short} {\em (Short time estimates)}\\
For all $\alpha \in \R$, there exists a time $T > 0$ such that the following
estimates hold. The solution $w(\tau) = \cS_\alpha(\tau) w_0$ of the integral equation
\eqref{st0} satisfies, for any $p \in [1,2)$,
\begin{equation}\label{shortest1}
  \|\cS_\alpha(\tau) w_0\|_{L^p} \,\le\, \frac{C_2}{a(\tau)^{1-1/p}}\,\|w_0\|_{L^1}\,,
  \qquad \forall\,\tau \in (0,T]\,,
\end{equation}
where the constant $C_2 > 0$ may depend on $\alpha$ and $p$. Similarly,
\begin{equation}\label{shortest2}
  \|\nabla \cS_\alpha(\tau)w_0\|_{L^1} + \|\cS_\alpha(\tau) \nabla w_0\|_{L^1}
  \,\le\, \frac{C_2}{a(\tau)^{1/2}}\,\|w_0\|_{L^1}\,, \qquad
  \forall\,\tau \in (0,T]\,.
\end{equation}
\end{lem}

\begin{proof}
Proceeding as in \eqref{cFest} we find, for $p \in [1,2)$ and $\tau \in [0,T]$,
\begin{align*}
  \|\cF[w](\tau)\|_{L^p} \,&\le\, \int_0^\tau \frac{e^{-(\tau-s)/2}}{a(\tau-s)^{3/2-1/p}}
  \,\|v^G w(s) + v(s) G\|_{L^1}\dd s \\ \,&\le\,
  C \,\|w\|_{1,T} \int_0^\tau \frac{e^{-s/2}}{a(s)^{3/2-1/p}}\dd s \,\le\,
  C\,a(T)^{1/p-1/2}\,\|w_0\|_{{L^1}}\,.
\end{align*}
Using in addition \eqref{sg0est1}, we can bound the right-hand side of \eqref{st0}
in $L^p$, and this gives \eqref{shortest1}. 

To obtain the estimate on $\cS_\alpha(\tau)\nabla w_0$ in \eqref{shortest2}, we
consider the modified integral equation $w(\tau) = \cS_0(\tau)\partial_j w_0 - \alpha
\cF[w](\tau)$, where $j \in \{1,2\}$. The exponential factor $e^{\tau/2}$ in front
of $\cS_0(\tau)\nabla w_0$ in \eqref{sg0est2} is irrelevant for short times, and
we neglect it. Repeating the arguments above and observing that
\[
  \int_0^\tau \frac{1}{a(\tau-s)^{1/2} a(s)^{1/2}}\dd s \,\lesssim\, \int_0^\tau
  \frac{1}{(t-s)^{1/2} s^{1/2}}\dd s \,\lesssim\, 1\,, \qquad \forall\, \tau \le 1\,,
\]
we can solve this equation by a fixed point argument in the function space
\[
  X_T' \,=\, \bigl\{w \in C^0\bigl((0,T],L^1(\R^2)\bigr)\,;\, \|w\|_{1,T}' < \infty\bigr\}\,,
  \qquad \|w\|_{1,T}' \,=\, \sup_{\tau \in (0,T]} a(\tau)^{1/2} \|w(\tau)\|_{L^1}\,,
\]
provided $T >0$ is sufficiently small, depending on $|\alpha|$. The calculations
are straightforward and can be left to the reader. 

Finally, to prove the estimate on $\nabla\cS_\alpha(\tau)w_0$ in \eqref{shortest2},
we differentiate both sides of \eqref{st0} with respect to $\xi$. This gives
an integral equation for $\nabla w(\tau)$ if we observe that
\[
  \nabla\cF[w](\tau) \,=\, \int_0^\tau \nabla \cS_0(\tau-s)\bigl(v^G \cdot \nabla w(s)
  + v(s) \cdot\nabla G\bigr)\dd s\,.
\]
Using \eqref{sg0est2} and \eqref{st3} we obtain in particular the estimate
\[
  \|\nabla \cF[w](\tau)\|_{L^1} \,\le\, \int_0^\tau \frac{C}{a(\tau-s)^{1/2}}
  \,\bigl(\|\nabla w(s)\|_{L^1} + \|w(s)\|_{L^1}\bigr)\dd s\,, 
\]
which implies that the integral equation has a unique solution $\nabla w \in X_T'$,
provided $T > 0$ is sufficiently small. Again we skip the details. This concludes
the proof of estimate \eqref{shortest2}. 
\end{proof}

\medskip\noindent{\bf Step 3\:} {\em Long time estimates.}\\
By construction, the linear map $\cS_\alpha(\tau)$ defined in the previous step
for $\tau \in [0,T]$ is strongly continuous and satisfies the semigroup property
$\cS_\alpha(\tau_1 + \tau_2) = \cS_\alpha(\tau_1)\cS_\alpha(\tau_2)$ whenever
$\tau_1 + \tau_2 \le T$. It can therefore be extended, in a unique way, to a $C_0$
semigroup $\bigl(\cS_\alpha(\tau)\bigr)_{\tau \ge 0}$ in $L^1(\R^2)$. When $|\alpha|$
is large, it is not known whether this semigroup is uniformly bounded in $L^1(\R^2)$,
but the following estimates will be useful.

\begin{lem}\label{lem:long} {\em (Long time estimates)}\\
For all $\alpha\in \R$ and all $\kappa>0$, there exists $C_3 > 0$ such that, 
for all $w_0\in L^1(\R^2)$ and all $\tau \ge T$,
\begin{equation}\label{longest}
  \bigl\| \cS_\alpha(\tau) w_0\bigr\|_{L^1} \,\le\, C_3\,e^{\kappa \tau} \|w_0\|_{L^1}\,, \qquad
  \bigl\| \cS_\alpha(\tau) \nabla w_0\bigr\|_{L^1} \,\le\, C_3\,e^{(\kappa-\frac{1}{2}) \tau}
  \|w_0\|_{L^1}\,.
\end{equation}
\end{lem}

The proof of Lemma \ref{lem:long} relies on the spectral analysis of the
operator $\cS_\alpha(\tau)$ in the spaces $L^1(\R^2)$ and $\mathcal X$, the latter being defined
as
 \[
   \mathcal X \,=\, \Bigl\{w\in L^1(\R^2)\,;\, w = \div f\,,\, f = (f_1,f_2) \in L^1(\R^2)^2\Bigr\}
   \,\subset\, L^1(\R^2)\,.
\]
It is easy to verify that $\mathcal X$ is a Banach space when equipped with the norm
\[
  \|w\|_{\mathcal X}\,=\, \|w\|_{L^1} + \inf\Bigl\{\|f_1\|_{L^1} + \|f_2\|_{L^1}\,;\, 
  f = (f_1,f_2) \in L^1(\R^2)^2\,,\, \div f = w\Bigr\}\,.
\]

We first observe that the restriction of $\cS_\alpha$ to $\mathcal X \subset L^1(\R^2)$
defines a strongly continuous semigroup in $\mathcal X$. Indeed, following \cite[Section~6.2]{GG},
we introduce the auxiliary equation
\begin{equation}\label{eq:auxiliaire}
  \partial_\tau f + \alpha v^G \div f + \alpha G\; \BSS [\div f ] \,=\,
  \Bigl(\cL - \frac{1}{2}\Bigr) f\,,
\end{equation}
which defines the evolution of a vector field $f=(f_1,f_2)$. Given initial data
$f_0 \in L^1(\R^2)^2$, the associated integral equation takes the form
\begin{equation}\label{eq:auxiliaire2}
  f(\tau) \,=\, e^{-\tau/2}\cS_0(\tau)f_0 - \alpha\int_0^\tau e^{-(\tau-s)/2}\cS_0(\tau-s)
  \Bigl(v^G \div f(s) + G\; \BSS [\div f(s)]\Bigr)\dd s\,.
\end{equation}
Arguing as in Step~2, on can show that \eqref{eq:auxiliaire2} has a unique solution in the
space defined by the norm $\sup_{\tau \in [0,T]}\|f(\tau)\|_{L^1} + \sup_{\tau \in (0,T]}
a(\tau)^{1/2} \|\div f(\tau)\|_{L^1}$, provided $T > 0$ is sufficiently small depending
on $|\alpha|$. This means that \eqref{eq:auxiliaire2} defines a strongly continuous
semigroup $\cT_\alpha$ in $L^1(\R^2)^2$ such that $f(\tau) = \cT_\alpha(\tau)f_0$. 
Assume now that $w_0\in \mathcal X$, so that $w_0=\div f_0$ for some $f_0 \in L^1(\R^2)^2$. 
If $f \in C^0([0,T],L^1(\R^2)^2)$ is the solution of \eqref{eq:auxiliaire2}, it
is straightforward to verify that the function $w(\tau) := \div f(\tau)$ satisfies
precisely the integral equation \eqref{st0} with $\sup_{\tau \in (0,T)} a(\tau)^{1/2}
\| w(\tau)\|_{L^1}<+\infty$. From there, we infer that $\sup_{\tau \in (0,T)}
\|w(\tau)\|_{L^1}<+\infty$, and by uniqueness in the space $X_T$, we deduce that
\begin{equation}\label{eq:intertwin}
  \cS_\alpha(\tau) w_0 \,=\, \cS_\alpha(\tau)\div f_0 \,=\, \div \bigl(\cT_\alpha(\tau)f_0
  \bigr)\,, \qquad \forall\,\tau \in [0,T]\,.
\end{equation}
We conclude in particular that $\cS_\alpha$ is a strongly continuous semigroup in $\mathcal X$.

For later use we also observe that the formula \eqref{eq:intertwin} can serve as
a definition of the quantity $S_\alpha(\tau)w_0$ when we only assume that 
$w_0 = \div f_0$ and $f_0 \in L^1(\R^2)^2$. In particular, if $w_0 \in L^1(\R^2)$,
the quantity $S_\alpha(\tau)\nabla w_0$ defined in this way belongs to $\mathcal X$
and satisfies
\begin{equation}\label{eq:Salphext}
  \|\cS_\alpha(\tau)\nabla w_0\|_{\mathcal X} \,\le\, \frac{C}{a(\tau)^{1/2}}\,\|w_0\|_{L^1}\,,
  \qquad \forall\,\tau \in (0,T]\,.
\end{equation}

\smallskip
We next show that the operator $\cS_\alpha(T)$ is a compact perturbation of
$\cS_0(T)$, which implies that the essential spectra of both operators coincide. 

\begin{lem}\label{lem:compact}
Let $T>0$ be given by Lemma \ref{lem:short}. The linear operator $\cK$ defined by
\begin{equation}\label{cTdef}
  \cK w_0 \,=\, \cS_\alpha(T)w_0 - \cS_0(T)w_0 \,=\, -\alpha \int_0^T \cS_0(T-\tau)
  \,\Lambda \bigl(\cS_\alpha(\tau) w_0\bigr)\dd \tau\,
\end{equation}
is compact in $L^1(\R^2)$ and in $\mathcal X$.
\end{lem}

\begin{proof}
We can focus on the compactness in $L^1(\R^2)$, because the compactness in $\mathcal X$
will follow easily. Fix $w_0 \in L^1(\R^2)$ such that $\|w_0\|_{L^1} \le 1$. In view
of \eqref{sg0est2} and \eqref{shortest2}, we have
\begin{equation}\label{comp1}
  \|\nabla \cK w_0\|_{L^1} \,\le\, \|\nabla\cS_\alpha(T)w_0\|_{L^1} +
  \|\nabla\cS_0(T)w_0\|_{L^1} \,\le\, \frac{C}{a(T)^{1/2}}\,\|w_0\|_{L^1}
  \,\le\, C'\,,
\end{equation}
where the constant $C'$ depends on $|\alpha|$ but not on $w_0$. On the other hand,
denoting $w(\tau) = \cS_\alpha(\tau)w_0$ and $v(\tau) = \BSS[w(\tau)]$ for $\tau \in [0,T]$, 
we deduce from \eqref{cTdef} and \eqref{sg0rep} that
\begin{equation}\label{repcT}
  \bigl(\cK w_0\bigr)(\xi) \,=\, \alpha \int_0^T \int_{\R^2} K(\xi,\eta,T-\tau)
  \bigl(v^G(\eta) w(\eta,\tau) + v(\eta,\tau) G(\eta)\bigr)\dd \eta\dd \tau\,,
\end{equation}
for all $\xi \in \R^2$ and all $\tau \in [0,T]$, where
\[
  K(\xi,\eta,s) \,:=\, \frac{e^{-s/2}}{8\pi a(s)^2}\,\bigl(\xi - \eta\,e^{-s/2}\bigr)\,
  \exp\Bigl(-\frac{|\xi-\eta\,e^{-s/2}|^2}{4a(s)}\Bigr)\,, \qquad
  \xi,\eta \in \R^2\,, \quad s \ge 0\,.
\]
In particular there exists a constant $C > 0$ such that
\begin{equation}\label{comp2}
  |K(\xi,\eta,s)| \,\le\, \frac{C}{a(s)^{3/2}}\,\exp\Bigl(-\frac{|\xi-\eta
  \,e^{-s/2}|^2}{5a(s)}\Bigr)\,, \qquad \xi,\eta \in \R^2\,, \quad s \ge 0\,.
\end{equation}
We now fix a large $R > 0$ and we assume that $|\xi| \ge 2R$. We claim that
\begin{equation}\label{comp3}
  |K(\xi,\eta,s)| \bigl(|v^G(\eta)| + G(\eta)^{1/2}\bigr) \,\le\,
  \widetilde{K}(\xi,\eta,s) \,:=\, \frac{C}{R}\,\frac{1}{a(s)^{3/2}}
  \,\exp\Bigl(-\frac{|\xi-\eta\,e^{-s/2}|^2}{6a(s)}\Bigr)\,.
\end{equation}
Indeed, if $|\eta| \le R$, then $|\xi - \eta\,e^{-s/2}|
\ge |\xi| - |\eta| \ge |\xi|/2 \ge R$, so that
\[
  \exp\Bigl(-\frac{|\xi-\eta\,e^{-s/2}|^2}{5 a(s)}\Bigr) \,\le\, 
  \exp\Bigl(-\frac{R^2}{30 a(s)}\Bigr)\,\exp\Bigl(-\frac{|\xi-\eta\,e^{-s/2}|^2}{6 a(s)}
  \Bigr)\,,
\]
and \eqref{comp3} follows since $a(s) \le a(T){\leq 1}$. If $|\eta| \ge R$, then \eqref{comp3}
is a direct consequence of \eqref{comp2} because $|v^G(\eta)| \le C/R$ and
$|G(\eta)|\le C/R^2$. Altogether we obtain 
\begin{align}\nonumber
  \int_{|\xi|\ge 2R} \bigl|\cK w_0(\xi)\bigr|\dd\xi \,&\le\, 
 | \alpha| \int_0^T \int_{\R^2}\int_{\R^2} \widetilde{K}(\xi,\eta,T-\tau)
  \bigl(|w(\eta,\tau)| + |v(\eta,\tau)|G(\eta)^{1/2}\bigr)\dd\eta \dd\xi \dd\tau \\
  \,&\le\, \nonumber
  \frac{|\alpha|}{R}\int_0^T \frac{C}{a(T{-}\tau)^{1/2}}\int_{\R^2}
  \bigl(|w(\eta,\tau)| + |v(\eta,\tau)|G(\eta)^{1/2}\bigr)\dd\eta \dd\tau \\ \label{comp4} 
  \,&\le\, \frac{|\alpha|}{R}\int_0^T \frac{C}{a(T{-}\tau)^{1/2}}\,
  \|\cS_\alpha(\tau)w_0\|_{L^1}\dd \tau \,\le\, \frac{C|\alpha|}{R}\,,
\end{align}
where the constant $C$ is independent of $R$ and $w_0$. Now, in view of
\eqref{comp1}, \eqref{comp4}, the Riesz criterion \cite[Theorem~XIII.66]{RSIV}
asserts that the image of the unit ball in $L^1(\R^2)$ by the operator $\cK$ is
a compact subset of $L^1(\R^2)$, which means that $\cK$ is a compact operator. 

The same arguments also apply to the operator $\cT_\alpha(T)$ associated with
the integral equation \eqref{eq:auxiliaire2}, and show that the difference
$\cT_\alpha(T)-\cT_0(T)$ is a compact operator in $L^1(\R^2)^2$. Now for any
$w = \div f \in \mathcal X$, it follows from \eqref{eq:intertwin} that
\[
  \cK w \,=\, \bigl(\cS_\alpha(T) - \cS_0(T)\bigr)\div f \,=\,
  \div\Bigl(\bigl(\cT_\alpha(T) - \cT_0(T)\bigr)f\Bigr)\,.
\]
This identity readily implies that the operator $\cK$ is compact in $\mathcal X$. 
\end{proof}

To conclude the proof of Lemma~\ref{lem:long}, it remains to show that the 
operator $\cS_\alpha(T) = \cS_0(T) + \cK$ acting on $L^1(\R^2)$ (resp. on $\cX$)
has no spectrum outside the unit disk $D_1 := \{\mu \in \C\,;\, |\mu| \le 1\}$
(resp. outside the disk $D_{e^{-T/2}}$). Using the estimates \eqref{sg0est1}
and \eqref{sg0est2}, it is easy to verify that
\[
  \|\cS_0(\tau )\|_{L^1\to L^1} \,\le\, 1\,, \quad \text{and}\quad
  \|\cS_0(\tau) \|_{\cX\to \cX} \,\le\, C_0\,e^{-\tau/2}\,, \quad \forall\, \tau \ge 0\,,
\]
for some positive constant $C_0$. It then follows from standard properties of
semigroups (see \cite[Chapter IV, Proposition 2.2]{EN}) that the spectrum of
$\cS_0(T)$ in $L^1$ (resp. in $\mathcal X$) is entirely contained in $D_1$ (resp. in
$D_{e^{-T/2}}$). As $\cK$ is a compact operator by Lemma~\ref{lem:compact}, the
spectrum of $\cS_\alpha(T)$ outside $D_1$ (resp. outside $D_{e^{-T/2}}$)
consists of isolated eigenvalues with finite multiplicity. Our goal is
to verify that such eigenvalues do not exist.

We first consider the $L^1$ case. Assume on the contrary that
$\mu \in \C \setminus D_1$ is an eigenvalue of $\cS_\alpha(T)$ in $L^1(\R^2)$. By
the spectral mapping theorem for point spectra \cite[Section IV.3.7]{EN}, there
exists an eigenvalue $\lambda \in \C$ such that $\mu = e^{\lambda T}$ and a
nonzero eigenfunction $\phi \in D(\cL)$ satisfying
\begin{equation}\label{eigeneq}
  \bigl(\cL - \alpha \Lambda\bigr)\phi \,=\, \lambda \phi\,.
\end{equation}
We now repeat the arguments in \cite[Section~4.1]{GW2}. Using polar coordinates
in $\R^2$ and observing that both operators $\cL$ and $\Lambda$ are invariant
under rotations about the origin, we can assume that $\phi(r,\theta) = \psi(r)
e^{im\theta}$ for some $m \in \Z$. The eigenvalue equation \eqref{eigeneq}
can thus be reduced to an ODE for the radial profile $\psi(r)$, and the
assumption that $\phi \in L^1(\R^2)$ implies that $\psi(r) = \cO\bigl(r^k
e^{-r^2/4}\bigr)$ as $r \to +\infty$, for some $k \in \N$. In particular,
we have $\phi \in \cY$ where
\[
  \cY \,=\, \Bigl\{w \in L^1(\R^2)\,;\, \|w\|_\cY^2 := \int_{\R^2}|w(\xi)|^2\,e^{|\xi|^2/4}
  \dd \xi < \infty\Bigr\}\,.
\]
But one of the main observations in \cite{GW2} is that the operator $\cL$
is self-adjoint and negative in the Hilbert space $\cY$, whereas $\Lambda$
is skew-adjoint. In particular, taking the scalar product of \eqref{eigeneq}
with $\phi$ yields
\[
  \Re(\lambda)\,\|\phi\|_\cY^2 \,=\, \Re\langle (\cL - \alpha \Lambda)\phi\,,\phi\rangle
  \,=\, \Re\langle \cL \phi\,,\phi\rangle \,\le\, 0\,.
\]
Thus $\Re(\lambda) \le 0$, hence $|\mu| = e^{T\Re(\lambda)} \le 1$, which contradicts
the assumption that $\mu \notin D_1$. It follows that the spectrum of
the operator $\cS_\alpha(T)$ in $L^1$ is entirely contained in the unit disk.
From \cite[Chapter IV, Proposition 2.2]{EN} we infer that, for any $\kappa>0$,
the exists a constant $c_\kappa > 0$ such that
\begin{equation}\label{specest1}
  \|\cS_\alpha(\tau)\|_{L^1\to L^1} \,\le\, C_\kappa e^{\kappa\tau}\,, \quad
  \forall\,\tau \ge 0\,.
\end{equation}

The spectral analysis it quite similar in the space $\mathcal X$, the main difference
being that all elements of $\mathcal X$ are integrable functions with zero mean.
As a consequence, using the same notation as above, we now have $\phi\in \cY_0$
where
\[
  \cY_0 \,=\, \Bigl\{w\in \cY\,;\, \int_{\R^2} w(\xi)\dd \xi = 0\Bigr\}\,.
\]
Since $\cL \le -1/2$ on $\cY_0$ by \cite[Proposition 4.1]{GW2}, it follows
that $\Re(\lambda)\le -1/2$, which shows that the spectrum of $\cS_\alpha(T)$
is entirely contained in $D_{e^{-T/2}}$. We deduce that, that for all $\kappa>0$,
there exists a constant $C_\kappa$ such that
\begin{equation}\label{specest2}
  \|\cS_\alpha(\tau)\|_{\mathcal X\to \mathcal X} \,\le\, C_\kappa e^{(\kappa-\frac{1}{2})\tau}\,,
  \quad\forall\,\tau \ge 0\,.
\end{equation}

The proof of Lemma~\ref{lem:long} is now easily concluded. The first estimate in
\eqref{longest} is implied by \eqref{specest1}. To prove the second one, we
assume that $\tau \ge T$ and we observe that
\[
  \|\cS_\alpha(\tau)\nabla w_0\|_{L^1} \,\le\, \|\cS_\alpha(\tau)\nabla w_0\|_{\mathcal X} \,\le\, 
  \|\cS_\alpha(\tau-T)\|_{X \to X} \|\cS_\alpha(T)\nabla w_0\|_{\mathcal X}\,\le\,
  C\,e^{(\kappa-\frac12)\tau}\|w_0\|_{L^1}\,,
\]
where the last inequality follows from \eqref{eq:Salphext} and \eqref{specest2}. \QED

\medskip\noindent{\bf Step 4\:} {\em End of the proof of Proposition~\ref{prop:semigroup}.}\\
If $\tau \in (0,2T)$, $p \in [1,2)$, and $j \in \{0,1\}$, it follows from \eqref{shortest1} and
\eqref{shortest2} that
\[
  \|\cS_\alpha(\tau)\nabla^j\|_{L^1\to L^p} \,\le\,  \|\cS_\alpha(\tau/2)\|_{L^1\to L^p}
  \,\|\cS_\alpha(\tau/2)\nabla^j\|_{L^1\to L^1} \,\le\, \frac{C}{a(\tau)^{1+j/2-1/p}}\,,
\]
where we used the fact that $a(\tau)/2 \le a(\tau/2) \le a(\tau)$ for all
$\tau \ge 0$. This proves \eqref{sgest} for short times. If $\tau \ge 2T$, we observe
that
\[
  \|\cS_\alpha(\tau)\nabla^j\|_{L^1\to L^p} \,\le\, \|\cS_\alpha(T)\|_{L^1\to L^p}
  \,\|\cS_\alpha(\tau-T)\nabla^j\|_{L^1\to L^1} \,\le\, C\,e^{(\kappa-\frac{j}{2})\tau}\,,
\]
where the last inequality follows from \eqref{shortest1} and \eqref{longest}.
This proves \eqref{sgest} for long times.\QED

\medskip
To conclude this section, we deduce from Proposition~\ref{prop:semigroup} the following
estimates on the two-parameter semigroup $\Sigma_\alpha(t,t_0)$ associated with
equation \eqref{LinGam1}\:

\begin{cor}\label{cor:Sigma-alpha}
Fix $\alpha \in \R$, $\kappa > 0$, and $p \in [1,2)$. There exists
a constant $C = C(\alpha,\kappa,p) > 0$ such that, for all $t > t_0 > 0$
and all $\omega_0 \in L^1(\R^2)$, 
\begin{align}\label{SigGamBd1}
  \bigl\|\Sigma_\alpha(t,t_0) \omega_0\bigr\|_{L^p} \,\le\, &\frac{C}{(t-t_0)^{1-1/p}}\,
  \Bigl(\frac{t}{t_0}\Bigr)^{\kappa}\,\|\omega_0\|_{L^1}\,\\ \label{SigGamBd2}
  \bigl\|\Sigma_\alpha(t,t_0)\nabla \omega_0\bigr\|_{L^p} \,\le\, &\frac{C}{(t-t_0)^{3/2-1/p}}\,
  \Bigl(\frac{t}{t_0}\Bigr)^{\kappa}\,\|\omega_0\|_{L^1}\,. 
\end{align}
\end{cor}

\begin{proof}
According to \eqref{sschange} and Proposition~\ref{prop:semigroup}, the
non-autonomous evolution equation \eqref{LinGam1} generates a
two-parameter semigroup $\Sigma_\alpha(t,t_0)$ given by the formula
\begin{equation}\label{SigGam}
  \Bigl(\Sigma_\alpha(t,t_0)\omega_0\Bigr)(x) \,=\, \frac{1}{t}
  \Bigl(\cS_\alpha\Bigl(\log\frac{t}{t_0}\Bigr)w_0\Bigr)
  \Bigl(\frac{x-z}{\sqrt{t}}\Bigr)\,, \qquad x \in \R^2\,, \quad t > t_0 > 0\,,
\end{equation}
where $w_0(\xi) = t_0 \omega_0\bigl(z + \sqrt{t_0}\xi\bigr)$.
Applying the first inequality in \eqref{sgest} we thus find
\begin{align*}
  \bigl\|\Sigma_\alpha(t,t_0)\omega_0\bigr\|_{L^p} \,&=\, \frac{1}{t^{1 - 1/p}}\,
  \Bigl\|\cS_\alpha\Bigl(\log\frac{t}{t_0}\Bigr)w_0\Bigr\|_{L^p} \\
  \,&\le\, \frac{C_0}{t^{1-1/p}}\,\Bigl(\frac{t}{t_0}\Bigr)^\kappa
  \Bigl(\frac{t}{t-t_0}\Bigr)^{1-1/p}\|w_0\|_{L^1}
  \,=\, \frac{C_0}{(t-t_0)^{1-1/p}}\,\Bigl(\frac{t}{t_0}\Bigr)^\kappa
  \|\omega_0\|_{L^1}\,.
\end{align*}
Similarly, translating the second inequality in \eqref{sgest} in the original
variables, we obtain \eqref{SigGamBd2}. 
\end{proof}

%%%%%%%%%%%%%%%%%%%%%%%%%%%%%%%%%%%%%%%%%%%%%%%%%%%%%%%%%%%%%%%%%%%%%%%%%%%
%%%%%%%%%%%%%%%%%%%%%%%%%%%%%%%%%%%%%%%%%%%%%%%%%%%%%%%%%%%%%%%%%%%%%%%%%%%

\section{Existence and uniqueness theory}\label{sec4}

In this section we prove the existence and uniqueness claims in 
Theorem~\ref{thm:main}. We first derive the integral equation \eqref{IntEqIni}
which is the starting point of our analysis. Then we decompose the vorticity
$\omega(x,t)$ into a principal part $\omega_1(x,t)$, which originates from the
Dirac mass at initial time, and a correction term $\omega_2(x,t)$, which
describes the boundary layer. These quantities satisfy a coupled
system of integral equations which can be solved by a fixed point argument
to obtain existence and uniqueness of the solution of \eqref{IntEqIni} under
the assumptions \eqref{maincond}.

\subsection{Integral equation and decomposition of the solution}
\label{ssec41}

The integral equation \eqref{IntEq} in Definition~\ref{def:sol} is difficult
to exploit directly, because we do not have precise information on the
solution $\omega(t_0)$ at ``initial time'' $t_0 > 0$. This problem can
be eliminated by taking the limit $t_0 \to 0$. 

\begin{lem}\label{lem:inteq}
Let $\omega\in C^0((0,T), L^1_\perp(\R^2_+) \cap L^{4/3}(\R^2_+))$ be a
solution of \eqref{2Dvort} in the sense of Definition~\ref{def:sol}. 
Assume furthermore that
\begin{equation}\label{InteqHyp}
  \|\omega\|_\star \,:=\, \sup_{t\in (0,T)}t^{1/4} \|\omega(t)\|_{L^{4/3}} \,<\, +\infty\,,
  \qquad \text{and} \quad \omega(t) \,\xrightharpoonup[t\to 0]{}\,\alpha\,\delta_z\,,   
\end{equation}
where $z = (0,1)$ and $\alpha \in \R$. Then the integral equation \eqref{IntEqIni}
holds for all $t \in (0,T)$. 
\end{lem}

\begin{proof}
For $0 < t_0 < t < T$, we write equation \eqref{IntEq} in the form  $\omega(t) \,=\,
\omega_L(t,t_0) - \omega_{NL}(t,t_0)$, where $\omega_L(t,t_0) = S(t-t_0)\omega(t_0)$
and $\omega_{NL}(t,t_0) = \int_{t_0}^tS(t-s)\div\bigl(u(s)\omega(s)\bigr)\dd s$. Proceeding
as in Remark~\ref{rem:mild}, we observe that
\[
  \int_{t_0}^t \bigl\|S(t-s)\div\bigl(u(s)\omega(s)\bigr)\bigr\|_{L^1}\dd s \,\lesssim\,
  \int_{t_0}^t \frac{\|\omega(s)\|_{L^{4/3}}^2}{(t-s)^{1/2}}\dd s \,\lesssim\, 
  \int_{t_0}^t \frac{\|\omega\|_\star^2}{(t-s)^{1/2}s^{1/2}}\dd s\,,
\]
where the last integral is bounded by $\pi\|\omega\|_\star^2$.  If $t > 0$ is
fixed and $t_0 \to 0$, this shows that $\omega_{NL}(t,t_0)$ converges in $L^1(\R^2_+)$
to a limit that we denote $\int_0^t S(t-s)\div\bigl(u(s) \omega(s)\bigr)\dd s$.
Since $\omega_{NL}(t,t_0)\in L^1_\perp(\R^2_+)$ for all $t_0>0$, the limit also belongs
to  $L^1_\perp(\R^2_+)$. It follows that $\omega_L(t,t_0) \equiv \omega(t) + \omega_{NL}(t,t_0)$
also converges in $L^1_\perp(\R^2_+)$ to some limit, which we denote $\bar\omega$. Since
\[
  S(t_0)\omega_L(t,t_0) - \bar\omega \,=\, S(t_0)\bigl(\omega_L(t,t_0) - \bar\omega\bigr)
  + \bigl(S(t_0) -  \1\bigr)\bar\omega\,,
\]
we deduce from Proposition~\ref{prop:Stokes} that $S(t_0)\omega_L(t,t_0) \equiv
S(t)\omega(t_0)$ converges to $\bar\omega$ in $L^1_\perp(\R^2_+)$ as $t_0 \to 0$.
But in view of \eqref{Stokes} we have, for any $x \in \R^2_+$, 
\[
  \bigl(S(t)\omega(t_0)\bigr)(x) \,=\, \int_{\R^2_+} K(x,y,t)\omega(y,t_0)\dd y
  \,\xrightarrow[t_0\to 0]{}\,\alpha\,K(x,z,t)\,,   
\]
where we used the second assumption in \eqref{InteqHyp} and the fact that the function
$y \mapsto K(x,y,t)$ belongs to $C_0(\R^2_+)$ for any $x \in \R^2_+$ and any $t > 0$.
This shows that $\bar\omega = \alpha S(t)\delta_z$, so that \eqref{IntEqIni} follows
from \eqref{IntEq} by taking the limit $t_0 \to 0$. 
\end{proof}

Assume that $\omega \in C^0\bigl((0,T),L^1_\perp(\R^2_+) \cap L^{4/3}(\R^2_+)\bigr)$
is a solution of \eqref{IntEqIni} satisfying \eqref{maincond} and such that
$\omega(t) \weakto 0$ as $t \to 0$. If the Reynolds number $|\alpha|$ is small
enough, it is possible to perform a fixed point argument directly on the
integral equation \eqref{IntEqIni} to prove that the solution is unique, and
existence can be established by a similar argument; see \cite{Abe} for a careful
study of the Cauchy problem for \eqref{2Dvort} in the perturbative regime. If
$|\alpha|$ is large, this approach does not work, but it is possible to
circumvent the problem by a suitable decomposition of the solution, which we now
describe.

Let $\zeta : \R_+ \to [0,1]$ be a smooth function satisfying
$\zeta(r) = 1$ for $r \le 1/2$ and $\zeta(r) = 0$ for $r \ge 1$. Given
$r_0 \in (0,1)$, we introduce the cut-off function $\chi : \R^2 \to [0,1]$
defined by
\begin{equation}\label{chidef}
  \chi(x) \,=\, \zeta\Bigl(\frac{|x-z|}{r_0}\Bigr)\,, \qquad \forall x \in \R^2\,,
\end{equation}
where $z = (0,1) \in \R^2_+$. In other words the function $\chi$ is radially symmetric
with respect to the initial position $z$ of the vortex, equal to $1$ inside the ball
$B(z,r_0/2)$ and to zero outside $B(z,r_0)$. The parameter $r_0$ will not play any role
in this Section, but we will take advantage of this extra degree of freedom in Section \ref{sec5}.

In view of \eqref{S12def}, the Stokes semigroup satisfies
$S(t) = S_1(t) + S_2(t)$, where the main term $S_1(t)$ is the restriction to $\R^2_+$
of the heat semigroup in the whole plane, in the sense of \eqref{S1action}, and
$S_2(t)$ takes into account the boundary condition.  At the nonlinear level,
we introduce a similar decomposition $\omega(t) = \omega_1(t) + \omega_2(t)$, where the partial
vorticities $\omega_1, \omega_2$ are defined by the integral formulas
\begin{align}
  \omega_1(t) \,&=\, \alpha S_1(t)\delta_z - \int_0^t S_1(t{-}s)\div \bigl(\chi
  u(s) \omega(s)\bigr)\dd s\,, \label{om1int} \\
  \omega_2(t) \,&=\, \alpha S_2(t)\delta_z - \int_0^t S_2(t{-}s)\div \bigl(\chi
  u(s) \omega(s)\bigr)\dd s 
  - \int_0^t S(t{-}s)\div \bigl((1{-}\chi)u(s) \omega(s)\bigr)\dd s\,.\label{om2int}
\end{align}

\begin{rem}\label{rem:domains}
It is extremely important to note that $\omega_1(t)$ and $\omega_2(t)$ {\em do
not satisfy the integral condition} \eqref{noslip}, although the total
vorticity $\omega(t)$ does.  A related observation is that the right-hand side
of \eqref{om1int} has a natural extension to the whole plane $\R^2$, which is
obtained by replacing $S_1(t)$ with the heat semigroup $e^{t\Delta}$ in $\R^2$
and by extending the nonlinearity $\chi u(s) \omega(s)$ by zero outside
$\R^2_+$. This point of view turns out to be quite useful, so we consider
from now on that $\omega_1(t)$ is defined on the whole plane $\R^2$. Of course,
the restriction of $\omega_1(t)$ to $\R^2_+$ is all we need to reconstruct
the original vorticity $\omega(t)$.
\end{rem}

The nonlinear term in \eqref{om1int} is still too large for a straightforward fixed point
argument, but this can be cured by decomposing further $\omega_1(t) = \bar\omega_1(t)
+ \hat\omega_1(t)$, where $\bar\omega_1(t) := \alpha e^{t\Delta} \delta_z$ is 
the self-similar Lamb--Oseen vortex with circulation parameter $\alpha$ centered at
$z = (0,1)$. As for the velocity field, we define $\bar u_1(t) = \BSS[\bar\omega_1(t)]$
and $\hat u_1(t) = \BSS[\hat\omega_1](t)$, which is quite natural since $\bar\omega_1(t)$
and $\hat\omega_1(t)$ are defined in the whole plane $\R^2$. In contrast, we 
set $u_2(t) = \BS[\omega_2(t)]$ where $\BS$ is the Biot--Savart operator \eqref{BSlaw}
in the half-plane $\R^2_+$. Summarizing, we have the following decompositions of
the vorticity $\omega(t)$ and the velocity field $u(t) = \BS[\omega(t)]$\:
\begin{equation}\label{omudecomp}
\begin{split}
  \omega(t) \,&=\, \hspace{20pt}\bar\omega_1(t) \hspace{15pt} + \hspace{19pt}\hat\omega_1(t)
  \hspace{16pt} \,+\, \omega_2(t)\,, \\ 
  u(t) \,&=\, \bar u_1(t) + \bar v_1(t) + \hat u_1(t) + \hat v_1(t) 
  \,+\, u_2(t)\,,
\end{split}
\end{equation}
where $\bar v_1(t)$ and $\hat v_1(t)$ are correction terms which take into account
the difference between the Biot--Savart formulas \eqref{BSlaw} and \eqref{BS2d}\:
\begin{equation}\label{vcordef}
   \bar v_1(t) \,=\, \BS[\bar\omega_1(t)] - \BSS[\bar\omega_1(t)]\,, \qquad
   \hat v_1(t) \,=\, \BS[\hat\omega_1(t)] - \BSS[\hat\omega_1(t)]\,. 
\end{equation}
In agreement with Remark~\ref{rem:domains}, the quantities $\bar\omega_1$, $\hat\omega_1$,
$\bar u_1$, and $\hat u_1$ are defined in the whole plane $\R^2$, while $\omega_2$, $u_2$,
$\bar v_1$, and $\hat v_1$ are defined in the upper half-plane only.  

\begin{lem}\label{lem:decomp}
Let $\omega\in C^0((0,T), L^1_\perp(\R^2_+) \cap L^{4/3}(\R^2_+))$ be a 
solution of \eqref{IntEqIni} satisfying \eqref{maincond}, and assume that
$\omega$ is decomposed as in \eqref{om1int}--\eqref{vcordef} with $\bar\omega_1(t)
= \alpha e^{t\Delta} \delta_z$. Then the vorticity components $\hat\omega_1(t)$ and
$\omega_2(t)$ satisfy the integral equations 
\begin{align}\label{def:hatom1}
  \hat \omega_1(t) \,&=\, \int_0^t \Sigma_\alpha (t,s) \div\bigl(F_1(s) - F_2(s)\bigr)\dd s\,,\\
  \omega_2(t) \,&=\, \alpha S_2(t) \delta_z - \int_0^t S_2(t-s)\div F_3(s)\dd s
  - \int_0^t S(t-s)\div F_4(s)\dd s\,, \label{def:om2}
\end{align}
where $\Sigma_\alpha(t,s)$ is the linear evolution operator introduced in
\eqref{Sigalphdef}, and 
\begin{equation}\label{Fjdef}
\begin{split}
  F_1 \,&=\, (1-\chi)\bigl(\bar u_1\,\hat\omega_1 + \hat u_1\,\bar\omega_1\bigr)\,,
  \qquad
  F_2 \,=\, \chi\bigl(u\,\omega - \bar u_1\,\bar\omega_1 - \bar u_1\,\hat\omega_1 -
  \hat u_1\,\bar\omega_1\bigr)\,, \\
  F_3 \,&=\, \chi\bigl(u\,\omega - \bar u_1\,\bar\omega_1\bigr)\,, \hspace{59pt}
  F_4 \,=\, (1-\chi)\bigl(u\,\omega - \bar u_1\,\bar\omega_1\bigr)\,.
\end{split}
\end{equation}
\end{lem}

\begin{rem}
The idea behind the decomposition \eqref{def:hatom1}, \eqref{def:om2} is the
following.  First, thanks to the estimates in Lemma~\ref{lem:Stokesbounds},
the terms involving $S_2$ in \eqref{def:om2} are small for short times when
measured in critical norms.  Moreover, the terms $(1-\chi)\bar u_1$ and
$(1-\chi)\bar \omega_1$ in $F_1$ and $F_4$ are uniformly bounded thanks to the
cut-off.  Eventually, in $F_2$, the potentially large terms
$\chi\bar u_1\,\omega_2$ and $\chi u_2\,\bar\omega_1$ are ``off-diagonal'',
and should therefore be harmless. Therefore, for any fixed value of
$\alpha \in \R$, the integral equation in \eqref{def:hatom1}-\eqref{def:om2} will
be amenable to a fixed point argument.
\end{rem}

\begin{proof}
As was observed in Remark~\ref{rem:domains}, the relation \eqref{om1int} makes
sense in the whole plane $\R^2$ and defines a function $\omega_1 : \R^2 \times (0,T)
\to \R$ which satisfies the equation
\begin{equation}\label{R2vort}
  \partial_t \omega_1(x,t) + \div\bigl(\chi(x)\,u(x,t)\,\omega(x,t)\bigr) \,=\,
  \Delta \omega_1(x,t)\,, \qquad x \in \R^2\,, \quad t \in (0,T)\,.
\end{equation}
Following \eqref{omudecomp}, we use the decomposition $\omega_1(t) =
\bar\omega_1(t) + \hat\omega_1(t)$, where the Lamb--Oseen vortex $\bar\omega_1(t)
= \alpha e^{t\Delta} \delta_z$ solves the heat equation $\partial_t \bar\omega_1 =
\Delta\bar\omega_1$. We also observe that
\begin{equation}\label{F12id}
  \div(\chi\,u\,\omega) \,=\, \div F_3 \,=\, \div\bigl(\bar u_1\,\hat\omega_1
  + \hat u_1\,\bar\omega_1\bigr) - \div(F_1-F_2)\,,
\end{equation}
where the first equality holds because $\div(\chi\,\bar u_1\,\bar\omega_1) = 0$
by symmetry, and the second one follows immediately from the definitions
\eqref{Fjdef}. We conclude that the correction term $\hat\omega_1(t)$ satisfies
\begin{equation}\label{hatR2}
  \partial_t \hat\omega_1(t) + \bar u_1(t)\cdot\nabla\hat\omega_1(t) 
  + \hat u_1(t)\cdot\nabla\bar\omega_1(t) \,=\, \Delta \hat\omega_1(t)
  + \div\bigl(F_1(t)-F_2(t)\bigr)\,, \quad t \in (0,T)\,.
\end{equation}
Integrating this equation over the time interval $[t_0,t]$ with $0 < t_0 < t < T$ 
and using the notation \eqref{Sigalphdef} for the two-parameter semigroup
defined by the linear equation \eqref{LinGam1}, we obtain the Duhamel
representation formula
\begin{equation}\label{R3vort}
  \hat \omega_1(t) \,=\, \Sigma_\alpha (t,t_0)\hat\omega_1(t_0) + 
  \int_{t_0}^t \Sigma_\alpha (t,s) \div\bigl(F_1(s) - F_2(s)\bigr)\dd s\,,
  \qquad 0 < t_0 < t < T\,.
\end{equation}

It remains to take the limit $t_0 \to 0$ to obtain \eqref{def:hatom1}.
Unfortunately, we cannot proceed here as in Lemma~\ref{lem:inteq} because the
right-hand side of \eqref{SigGamBd1} diverges as $t_0 \to 0$.  As a consequence,
identifying the limit of $\Sigma_\alpha (t,t_0)\hat\omega_1(t_0) $ is not
obvious.  To solve this problem, we rewrite \eqref{om1int} in the form
$\hat\omega_1(t) = -\int_0^t S_1(t-s)\div F_3(s)\dd s$, and we use the following
bound that will be established in the proof of Lemma~\ref{lem:estF} below\:
\begin{equation}\label{F3bound}
  \|F_3(t)\|_{L^1} \,\lesssim\, \alpha^2 + |\alpha| t^{-1/4} \|\omega(t) -
  \bar\omega_1(t)\|_{L^{4/3}} + \|\omega(t) - \bar\omega_1(t)\|_{L^{4/3}}^2\,. 
\end{equation}
Now, our assumption \eqref{maincond} asserts that $\|\omega(t) -
\bar\omega_1(t)\|_{L^{4/3}} \le C t^{-\beta}$ for some $\beta < 1/4$.
So using \eqref{F3bound} and estimate \eqref{S1est} with $p = 1$,
we easily obtain
\[
  \|\hat\omega_1(t)\|_{L^1} \,\le\, \int_0^t \frac{C}{(t-s)^{1/2}}\,
  \|F_3(s)\|_{L^1}\dd s \,\le\, C\Bigl(\alpha^2 t^{1/2} + |\alpha|
  t^{1/4-\beta} + C t^{1/2-2\beta}\Bigr)\,,
\]
hence $\|\hat\omega_1(t)\|_{L^1} = \cO\bigl(t^{1/4-\beta}\bigr)$ as $t \to
0$. So, if we assume that $0 < \kappa < 1/4 - \beta$ in
Corollary~\ref{cor:Sigma-alpha}, we conclude that
$\Sigma_\alpha (t,t_0) \hat\omega_1(t_0) \to 0$ in $L^1(\R^2)$ as $t_0 \to
0$. Therefore we obtain \eqref{def:hatom1} by taking the limit $t_0 \to 0$ in
\eqref{R3vort}, and \eqref{def:om2} is a simple rewriting of \eqref{om2int},
recalling that $\div(\chi u \omega)= \div F_3$ and $\div((1-\chi) u \omega)=
\div F_4$. 
\end{proof}
	
%%%%%%%%%%%%%%%%%

\subsection{Solving \eqref{def:hatom1}--\eqref{def:om2} by a fixed point argument}
\label{ssec42}

This section is the core of the proof of Theorem~\ref{thm:main}. Without loss
of generality, we assume henceforth that $\alpha > 0$.  We focus on the pair
of integral equations \eqref{def:hatom1}--\eqref{def:om2}, and we prove by
a fixed point argument that they have a unique solution in a small ball of a
suitable function space. More precisely, given a time $T > 0$ that will be taken
sufficiently small, we introduce the Banach space
\begin{equation}\label{cZdef}
  \cZ_T \,=\, \Big\{(\hat \omega_1,\omega_2) \in C^0\bigl((0,T],L^{4/3}(\R^2)
  \times L^{4/3}(\R^2_+)\bigr)\,;\, \|\hat \omega_1\|_\star + \|\omega_2\|_\star <
  \infty\Big\}\,,  
\end{equation}
where, with a slight abuse of notation, we define
\begin{equation}\label{starnorm}
  \|\hat\omega_1\|_\star \,:=\, \sup_{t\in (0,T)} t^{1/4} \|\hat\omega_1(t)\|_{
  L^{4/3}(\R^2)}\,, \qquad
  \|\omega_2\|_\star \,:=\, \sup_{t\in (0,T)} t^{1/4} \|\omega_2(t)\|_{
  L^{4/3}(\R^2_+)}\,.
\end{equation}

Note that the pair $(\hat \omega_1, \omega_2)$ denotes here \emph{any} element
of $\cZ_T$, and not necessarily a solution of \eqref{def:hatom1}--\eqref{def:om2}.
Given any $(\hat \omega_1, \omega_2)\in \cZ_T$, we define $(\omega, u)$ by
\eqref{omudecomp}--\eqref{vcordef}, with $\bar \omega_1(t) = \alpha e^{t\Delta} \delta_z$.
Next, we define the associated functions $F_j$ by \eqref{Fjdef}, which we estimate
in terms of $(\hat \omega_1, \omega_2)\in \cZ_T$ in Lemma~\ref{lem:estF} below.
Eventually, we denote by $\Phi_1(\hat\omega_1, \omega_2)$ and $\Phi_2(\hat\omega_1,
\omega_2)$ the right-hand sides of \eqref{def:hatom1} and \eqref{def:om2} respectively,
and we prove that $\Phi = (\Phi_1,\Phi_2)$ has a fixed point in an appropriate
subset of $\cZ_T$.

\begin{lem}\label{lem:estF}
Given any $(\hat \omega_1,\omega_2) \in \cZ_T$, define the vorticity $\omega$ and
the velocity $u$ by \eqref{omudecomp}, with $\bar\omega_1$ as in \eqref{def:Oseen},
and the quadratic terms $F_j$ by \eqref{Fjdef}. Then, for all $t\in (0,T)$, the
following estimates hold\: 
\begin{equation}\label{est:Fj}
\begin{split}
  \|F_1(t)\|_{L^1(\R^2)} \,&\lesssim\, \alpha t^{-1/4} \|\hat\omega_1\|_\star\,,\\
  \|F_2(t)\|_{L^1(\R^2)} \,&\lesssim\, \alpha^2 + \alpha t^{-1/4} \|\hat\omega_1\|_\star
    + \alpha t^{-1/2} \|\omega_2\|_\star + t^{-1/2}\bigl(\|\hat\omega_1\|_\star^2
    + \|\omega_2\|_\star^2\bigr)\,,\\
  \|F_3(t)\|_{L^1(\R^2_+)} \,&\lesssim\, \alpha^2 + \alpha t^{-1/2} \|\hat\omega_1
    + \omega_2\|_\star  + t^{-1/2} \|\hat\omega_1 + \omega_2\|_\star^2\,,\\
  \|F_4(t)\|_{L^1(\R^2_+)} \,&\lesssim\, \alpha^2 e^{-c/t} + \alpha t^{-1/4} \|\hat\omega_1
    + \omega_2\|_\star  + t^{-1/2} \|\hat\omega_1 + \omega_2\|_\star^2\,,
  \end{split}
\end{equation}
where $c > 0$ is a positive constant depending on the radius $r_0$ in \eqref{chidef}. 
\end{lem}

\begin{rem}
It is readily verified that the Lamb--Oseen vortex
$\bar\omega_1(t) = \alpha e^{t\Delta}\delta_z$ satisfies
$\|\bar\omega_1\|_\star = C\alpha$ for some constant $C > 0$. As a consequence,
if the vorticity $\omega$ and the associated velocity $u$ are defined as in
\eqref{omudecomp}, we have
\[
  \|u(t)\|_{L^4} \,\lesssim\, \|\omega(t)\|_{L^{4/3}} \,\lesssim\, t^{-1/4}
  \bigl(\alpha + \|\hat\omega_1\|_\star + \|\omega_2\|_\star\bigr)\,, \qquad 0 <
  t < T\,,
\]
hence $\|u(t)\omega(t)\|_{L^1} \lesssim t^{-1/2} \bigl(\alpha + \|\hat\omega_1\|_\star
+ \|\omega_2\|_\star\bigr)^2$.  Unfortunately, when $\alpha$ is large, this
rough bound is not sufficient to solve the system
\eqref{def:hatom1}--\eqref{def:om2} by a fixed point argument, but we claim in
Lemma~\ref{lem:estF} that the quadratic terms $F_j$ actually satisfy better estimates.
More precisely, the largest terms in the right-hand sides of inequalities
\eqref{est:Fj} are $\mathcal O(t^{-1/2})$ as $t \to 0$, as in the crude
estimate above.  Among those, the quadratic quantities
will be harmless, because we will work with solutions that are small in $\cZ_T$,
so that all nonlinearities can be treated perturbatively. The term
$ t^{-1/2} \|\hat\omega_1 + \omega_2\|_\star$ in $F_3$ does not create any
problem either, because by \eqref{def:om2} it appears in connection with the operator
$S_2(t)$ which satisfies better estimates than what is expected from the scaling,
see Lemma~\ref{lem:Stokesbounds}. It follows that the right-hand side of
\eqref{def:om2} will be small in the norm \eqref{starnorm} if $T$ sufficiently
small. This is not the case of \eqref{def:hatom1}, because of the off-diagonal
term $t^{-1/2}\|\omega_2\|_\star$, but a suitable choice of the norm in
$\cZ_T$ eliminates the problem and allows us to close the fixed point argument.
\end{rem}

\begin{proof}
We first estimate the terms $F_1$ and $F_4$, which are supported away from
the center $z$ of the Lamb--Oseen vortex. In view of \eqref{chidef}, if $x$ belongs
to the support of $1-\chi$, then $|x-z| \ge r_0/2 > 0$. Moreover the formulas
\eqref{def:Oseen} show that $|\bar u_1(x,t)| \lesssim \alpha |x-z|^{-1}$ and
$|\bar \omega_1(x,t)| \lesssim \alpha t^{-1}e^{-|x-z|^2/(4t)}$. It follows that 
\begin{equation}\label{uom-away}
  \|(1-\chi)\bar u_1(t)\|_{L^4} \,\lesssim\, \alpha\,, \qquad
  \|(1-\chi)\bar \omega_1(t)\|_{L^{4/3}} \,\lesssim\, \alpha\,e^{-c/t}\,,  
\end{equation}
for some constant $c > 0$. Since $\|\hat u_1(t)\|_{L^4} \lesssim \|\hat \omega_1(t)
\|_{L^{4/3}} \lesssim t^{-1/4}\|\hat\omega_1\|_\star$, we thus obtain
\[
  \|F_1(t)\|_{L^1} \,\le\, \|(1-\chi)\bar u_1(t)\|_{L^4}\|\hat\omega_1(t)\|_{L^{4/3}}
  + \|\hat u_1(t)\|_{L^4}\|(1-\chi)\bar\omega_1(t)\|_{L^{4/3}} \,\lesssim\,
  \alpha t^{-1/4} \|\hat\omega_1\|_\star\,,
\]
which is the first estimate in \eqref{est:Fj}. 

On the other hand, as is easily verified, the quantity $F_4$ can be written in
the form
\begin{equation}\label{F4bis}
  F_4 \,=\, (1-\chi)\Bigl\{\bar v_1 \bar\omega_1 + \BS[\bar\omega_1]
  (\hat\omega_1 + \omega_2) + \BS[\hat\omega_1 + \omega_2]
  (\bar\omega_1 + \hat\omega_1 + \omega_2)\Bigr\}\,,
\end{equation}
where the correction term $\bar v_1$ defined in \eqref{vcordef} satisfies
\begin{equation}\label{corrvort}
  \bar v_1(x,t) \,=\, -\frac{1}{2\pi}\int_{\R^2_-}\frac{(x-y)^\perp}{|x-y|^2}
  \bigl(\bar\omega_1(y,t) + \bar\omega_1(y^*,t)\bigr)\dd y\,,
  \qquad x \in \R^2_+\,, \quad t \in (0,T)\,.
\end{equation}
Since $\bar\omega_1(y,t)$ (resp. $\bar\omega_1(y^*,t)$) is exponentially small when
$y \in \R^2_-$ (resp. $y \in \R^2_+$), we have
\begin{equation}\label{corrvort2}
  \bar v_1(x,t) \,=\, -\frac{\alpha}{\sqrt{t}}\,v^G\Bigl(\frac{x - z^*}{\sqrt{t}}
  \Bigr) \,+\, \cO_{L^4}\bigl(\alpha\,e^{-c/t}\bigr)\,, \qquad \text{hence}\quad
  \|\bar v_1(t)\|_{L^4(\R^2_+)} \,\lesssim\, \alpha\,.
\end{equation}
To bound the other terms in \eqref{F4bis}, we observe that $\BS[\bar\omega_1] =
\bar u_1 + \bar v_1$ and we use again \eqref{uom-away} and \eqref{corrvort2}.
This gives
\begin{align*}
  \|F_4(t)\|_{L^1} \,&\le\, \|\bar v_1\|_{L^4} \|(1{-}\chi)\bar\omega_1\|_{L^{4/3}}
  + \|(1{-}\chi)(\bar u_1 + \bar v_1)\|_{L^4}\|\hat\omega_1 + \omega_2\|_{L^{4/3}} \\[1mm]
  & \quad~ + \|\BS[\hat\omega_1 + \omega_2]\|_{L^4}\|(1{-}\chi)\bar\omega_1\|_{L^{4/3}}
  + \|\BS[\hat\omega_1 + \omega_2]\|_{L^4}\|\hat\omega_1 + \omega_2\|_{L^{4/3}} \\[1mm]
  \,&\lesssim\, \alpha^2 e^{-c/t} + \alpha t^{-1/4} \|\hat\omega_1
    + \omega_2\|_\star  + t^{-1/2} \|\hat\omega_1 + \omega_2\|_\star^2\,,
\end{align*}
which is the last estimate in \eqref{est:Fj}. 

We now consider the terms $F_2$ and $F_3$, which are supported near the center $z$
of the Lamb--Oseen vortex. As in \eqref{F4bis} we can write
\begin{equation}\label{F3bis}
  F_3 \,=\, \chi\Bigl\{\bar v_1 \bar\omega_1 + \BS[\bar\omega_1]
  (\hat\omega_1 + \omega_2) + \BS[\hat\omega_1 + \omega_2]
  (\bar\omega_1 + \hat\omega_1 + \omega_2)\Bigr\}\,.
\end{equation}
In view of \eqref{corrvort} we have $\|\chi\bar v_1(t)\|_{L^\infty} \lesssim
\|\bar\omega_1(t)\|_{L^1} \lesssim \alpha$, hence $\|\chi\bar v_1(t)
\bar\omega_1(t)\|_{L^1} \lesssim \alpha^2$. Unlike in \eqref{uom-away}, 
the cut-off function $\chi$ does not improve the estimates of $\bar\omega_1$ and
$\bar u_1$, so we simply use
\[
  \|\chi\bar u_1(t)\|_{L^4} \,\le\, \|\bar u_1(t)\|_{L^4} \,\lesssim\, \alpha t^{-1/4}\,,
  \qquad
  \|\chi\bar\omega_1(t)\|_{L^{4/3}} \,\le\, \|\bar\omega_1(t)\|_{L^{4/3}} \,\lesssim\,
  \alpha t^{-1/4}\,.
\]
Proceeding as above, we thus find the deteriorated bound
\[
  \|F_3(t)\|_{L^1} \,\lesssim\, \alpha^2 + \alpha t^{-1/4} \|\omega(t) -
  \bar\omega_1(t)\|_{L^{4/3}} + \|\omega(t) - \bar\omega_1(t)\|_{L^{4/3}}^2\,,
\]
which is \eqref{F3bound}, and the third estimate in \eqref{est:Fj} easily
follows. 

Finally, we deduce from \eqref{omudecomp} and \eqref{Fjdef} that
\begin{equation}\label{F2bis}
  F_2 \,=\, \chi\Bigl\{(\bar v_1+\hat v_1)(\bar\omega_1 + \hat\omega_1)
  + \hat u_1 \hat\omega_1 + \BS[\bar\omega_1 + \hat\omega_1]\omega_2 +
  \BS[\omega_2](\bar\omega_1 + \hat\omega_1 + \omega_2)\Bigr\}\,,
\end{equation}
where we recall that $\BS[\bar\omega_1] = \bar u_1 + \bar v_1$. We already know
that $\|\chi\bar v_1\bar\omega_1\|_{L^1} \lesssim \alpha^2$, and for the other
terms involving $\bar v_1$ or $\hat v_1$ we use the bounds
\[
  \|\chi \bar v_1(t)\|_{L^4} \,\lesssim\, \|\chi \bar v_1(t)\|_{L^\infty}
  \,\lesssim\, \alpha\,, \qquad
  \|\chi \hat v_1(t)\|_{L^4} \,\lesssim\, \|\chi \hat v_1(t)\|_{L^\infty}
  \,\lesssim\, t^{-1/4}\|\hat\omega_1\|_\star\,,
\]
which follow easily from the formula \eqref{corrvort} and its analogue for
$\hat v_1$. We thus find
\[
  \|F_2(t)\|_{L^1} \,\lesssim\, \bigl(\alpha + t^{-1/4}\|\hat\omega_1\|_\star\bigr)^2
  + \alpha t^{-1/2}\|\omega_2\|_\star + t^{-1/2}\bigl(\|\hat\omega_1\|_\star^2
    + \|\omega_2\|_\star^2\bigr)\,,
\]
which concludes the proof of \eqref{est:Fj}.   
\end{proof}
	
We now consider the map $\Phi = (\Phi_1, \Phi_2) : \cZ_T \to \cZ_T$
defined by 
\begin{align}\label{FP1}
  \Phi_1(\hat\omega_1,\omega_2)(t) \,&=\, \int_0^t \Sigma_\alpha(t,s)
  \div\bigl(F_1(s) - F_2(s)\bigr)\dd s\,,\\ \label{FP2}
  \Phi_2(\hat\omega_1,\omega_2)(t) \,&=\, \alpha S_2(t)\delta_z - \int_0^t S_2(t{-}s)
  \div F_3(s)\dd s - \int_0^t S(t{-}s)\div F_4(s)\dd s\,,
\end{align}
for $0 < t \le T$, where the nonlinear terms $F_j(\hat\omega_1,\omega_2)$ are
defined in \eqref{Fjdef} for $j = 1,2,3,4$. Combining Lemmas~\ref{lem:S1} and
\ref{lem:Stokesbounds}, Corollary~\ref{cor:Sigma-alpha}, and Lemma~\ref{lem:estF},
we obtain the following useful estimates.
	
\begin{lem}\label{lem:Phibounds}
There exists positive constants $C_1, C_2$ and $K$, with $C_2$ independent of $\alpha$,
such that for all sufficiently small $T > 0$ and all $(\hat \omega_1,\omega_2) \in \cZ_T$,
the following estimates hold\:
\begin{align}\label{Phibd1}
  \|\Phi_1(\hat\omega_1,\omega_2)\|_\star \,&\le\, C_1\alpha^2 T^{1/2} + \alpha
  \bigl(C_1T^{1/4}\|\hat \omega_1\|_\star + K\|\omega_2\|_\star\bigr) + C_1
  \bigl(\|\hat \omega_1\|_\star^2 + \|\omega_2\|_\star^2\bigr)\,, \\[1mm] \label{Phibd2}
  \|\Phi_2(\hat\omega_1,\omega_2)\|_\star \,&\le\, C_2\alpha^2 T^{9/8} +
  C_2\alpha\bigl( T^{1/8} + T^{1/4}  \|\hat\omega_1 + \omega_2\|_\star\bigr) + C_2
  \|\hat\omega_1 + \omega_2\|_\star^2\,.
\end{align}
\end{lem}

\begin{rem}
In the previous inequalities, we have kept the powers of $\alpha$ in the
right-hand side of \eqref{Phibd1} even though $C_1$ and $K$ may depend on
$\alpha$.  We believe that this form will make the computations easier to follow
for the reader, and also preserves the symmetry between $\Phi_1$ and $\Phi_2$.
Note that the specific dependency with respect to $\alpha$ will be irrelevant
for our purposes.
\end{rem}	

\begin{proof}
To bound $\Phi_1$ we use the semigroup estimate \eqref{SigGamBd2} with
$p = 4/3$ and $0 < \kappa < 1/8$, as well as the bounds on $F_1$ and 
$F_2$ in \eqref{est:Fj}. For any $t \in (0,T]$, we thus obtain
\begin{align*}
  \|\Phi_1(&\hat\omega_1,\omega_2)(t)\|_{L^{4/3}} \,\le\, \int_0^t \frac{C}{(t-s)^{3/4}}
  \,\Bigl(\frac{t}{s}\Bigr)^{\kappa}\,\bigl(\|F_1(s)\|_{L^1}+ \|F_2(s)\|_{L^1}\bigr)\dd s \\
  & \,\le\, \int_0^t \frac{C}{(t-s)^{3/4}s^{1/2}}\,\Bigl(\frac{t}{s}\Bigr)^{\kappa}
  \Bigl(\alpha^2s^{1/2} + \alpha s^{1/4} \|\hat\omega_1\|_\star
    + \alpha \|\omega_2\|_\star + \|\hat\omega_1\|_\star^2 + \|\omega_2\|_\star^2
    \Bigr)\dd s\\
  & \,\le \frac{C}{t^{1/4}} \Bigl(\alpha^2 t^{1/2} + \alpha t^{1/4}\|\hat\omega_1\|_\star
    + \alpha \|\omega_2\|_\star + \|\hat \omega_1\|_\star^2 + \|\omega_2\|_\star^2\Bigr)\,,
\end{align*}  
where the constant $C$ depends on $\kappa$ and $\alpha$. Multiplying both sides by
$t^{1/4}$ and taking the supremum for  $t \in (0,T]$, we arrive at
\eqref{Phibd1}.

On the other hand, in view of \eqref{S12def}, we have $\bigl(S_2(t)\delta_z\bigr)(x) =
K_2(x,z,t)$ for all $x \in \R^2_+$, and we deduce from Corollary~\ref{cor:K1bounds} that
\begin{equation}\label{FP9}
  \|\alpha S_2(t)\delta_z\|_{L^{4/3}} \,\le\, \frac{C\alpha}{t^{1/8}}\,, \qquad t > 0\,.
\end{equation}
To bound the second term in the right-hand side of \eqref{FP2}, we recall
that the nonlinear term $F_3$ is supported away from boundary $x_2 = 0$,
thanks to the cut-off function $\chi$. This makes it possible to apply the
improved estimates \eqref{S2est} on the operator $S_2(t)$. Using in addition the
bound on $F_3$ in \eqref{est:Fj} we thus find
\begin{align}\nonumber
  \Bigl\|\int_0^t &S_2(t{-}s)\div F_3(s)\dd s\Bigr\|_{L^{4/3}} \,\le\,
    \int_0^t \frac{C}{(t-s)^{1/8}} \|F_3(s)\|_{L^1}\dd s\\ \nonumber
  \,&\le\, \int_0^t \frac{C}{(t-s)^{1/8}s^{1/2}}\,\Bigl(\alpha^2s^{1/2}
   + \alpha \|\hat\omega_1 + \omega_2\|_\star
   + \|\hat\omega_1 + \omega_2\|_\star^2 \Bigr)\dd s \\[1mm] \label{FP10}
  \,&\le\, C t^{3/8}\Bigl(\alpha^2 t^{1/2} + \alpha\|\hat\omega_1 + \omega_2\|_\star
   + \|\hat\omega_1 + \omega_2\|_\star^2\Bigr)\,.
\end{align}
Finally, using \eqref{S1est} and the bound on $F_4$ in \eqref{est:Fj}, we obtain
\begin{align}\nonumber
  \Bigl\|\int_0^t &S(t{-}s)\div F_4(s)\dd s\Bigr\|_{L^{4/3}} \,\le\,
    \int_0^t \frac{C}{(t-s)^{3/4}} \|F_4(s)\|_{L^1}\dd s\\ \nonumber
  \,&\le\, \int_0^t \frac{C}{(t-s)^{3/4}s^{1/2}}\,\Bigl(\alpha^2 e^{-c/s}
      + \alpha s^{1/4}\|\hat\omega_1 + \omega_2\|_\star + 
      \|\hat\omega_1 + \omega_2\|_\star^2\Bigr)\dd s\\[1mm] \label{FP11}
  \,&\le\, \frac{C}{t^{1/4}} \Bigl(\alpha^2 e^{-c/t} + \alpha t^{1/4}
      \|\hat\omega_1 + \omega_2\|_\star + \|\hat\omega_1 + \omega_2\|_\star^2\Bigr)\,.
\end{align}
In view of \eqref{FP9}, \eqref{FP10}, \eqref{FP11}, we have
\[
   t^{1/4}\|\Phi_2(\hat\omega_1,\omega_2)(t)\|_{L^{4/3}} \,\le\, C\Bigl(\alpha t^{1/8}
   + \alpha^2 t^{9/8} + \alpha t^{1/4}\|\hat\omega_1 + \omega_2\|_\star
   + \|\hat\omega_1 + \omega_2\|_\star^2\Bigr)\,.
\]
Taking the supremum for $t \in (0,T]$, we arrive at \eqref{Phibd2}.
\end{proof}

In view of \eqref{Phibd1} and \eqref{Phibd2}, we equip the space $\cZ_T$ with
the unbalanced norm
\begin{equation}\label{def:cZnorm}
  \|(\hat\omega_1,\omega_2)\|_\cZ \,:=\, \|\hat\omega_1\|_\star + C_0\|\omega_2\|_\star\,,
  \qquad \text{where} \quad C_0 \,=\, \max(1,2K\alpha)\,.
\end{equation}
Given any $\rho > 0$ we consider the ball
\begin{equation}\label{def:Brho}
  \cB_\rho\,:=\, \bigl\{(\hat\omega_1,\omega_2) \in \cZ_T\,;\, \|(\hat\omega_1,\omega_2)\|_\cZ
  \le \rho\bigr\}\,.
\end{equation}
Our goal is to show that, if $\rho$ and $T$ are taken sufficiently small (depending on $\alpha$),
the function $\Phi = (\Phi_1,\Phi_2)$ defined by \eqref{FP1}, \eqref{FP2} maps the ball
$\cB_\rho$ into itself and is a strict contraction in $\cB_\rho$. Indeed, assume
that $\bom := (\hat\omega_1,\omega_2) \in \cB_\rho$ and $\bw := (\hat w_1,w_2)
\in \cB_\rho$. Proceeding exactly as in the proof of Lemma~\ref{lem:Phibounds},
and using the fact that the nonlinear terms \eqref{Fjdef} are quadratic polynomials,
we easily obtain the Lipschitz estimate
\begin{equation}\label{LipPhi}
  \|\Phi(\bom) - \Phi(\bw)\|_\cZ \,\le\, \Bigl(\frac12 + C_1\rho + C_2T^{1/4}\Bigr)
  \|\bom - \bw\|_\cZ\,,
\end{equation}
for some positive constants $C_1,C_2$ depending on $\alpha$. Note that the
constant $C_0$ in the definition \eqref{def:cZnorm} was chosen so as to obtain
the term $\frac12 \|\bom - \bw\|_\cZ$ in the right-hand side of \eqref{LipPhi}.
On the other hand, we see from \eqref{Phibd1} and \eqref{Phibd2} that
$\|\Phi(0)\|_\cZ \le C_3 T^{1/8}$ for some positive constant $C_3$ depending on
$\alpha$. Now if we choose $\rho$ and $T$ sufficiently small so that
\begin{equation}\label{rhoT}
  C_1 \rho \,\le\, \frac18\,, \qquad C_2 T^{1/4} \,\le\, \frac18\,, \qquad
  \text{and}\quad C_3 T^{1/8} \,\le\, \frac{\rho}{4}\,,
\end{equation}
it follows that $\Phi$ maps $\cB_\rho$ into itself and that
$\mathrm{Lip}(\Phi) \le 3/4$ in $\cB_\rho$.  As a consequence, $\Phi$ has a
unique fixed point within $B_\rho$, which is by definition a solution of the
pair of integral equations \eqref{def:hatom1}--\eqref{def:om2}.

\subsection{Existence and uniqueness of mild solutions of \eqref{2Dvort}}
\label{ssec43}

It is now a relatively easy task to conclude the proof of the existence and
uniqueness claims in Theorem~\ref{thm:main}, together with inequality
\eqref{firstasym}.  The proof of the scale invariant estimates \eqref{mainest}
is postponed to Section~\ref{ssec51}.

\medskip\noindent{\bf Existence part\:} Given $\alpha > 0$ and $T > 0$, we
consider the function space $\cZ_T$ defined by \eqref{cZdef} endowed with the
norm \eqref{def:cZnorm}.  According to the previous section, if $T>0$ is chosen
sufficiently small, there exists $\rho>0$ such that the system
\eqref{def:hatom1}-\eqref{def:om2} has a unique solution within the ball
$\cB_\rho$, which we now denote $\bom = (\hat \omega_1, \omega_2)$. By
construction, we have
\[
  \|\bom\|_{\cZ} \,=\, \|\Phi(\bom)\|_{\cZ} \,\le\, \|\Phi(\bom) - \Phi(0)\|_{\cZ}
  + \|\Phi(0)\|_{\cZ} \,\le\, \frac34\,\|\bom\|_{\cZ} + C_3 T^{1/8}\,,
\]
so that 
\[
  \|\bom\|_{\cZ} \,=\, \sup_{t \in (0,T)} t^{1/4} \|\hat\omega_1(t)\|_{L^{4/3}}
  + C_0\,\sup_{t \in (0,T)} t^{1/4} \|\omega_2(t)\|_{L^{4/3}} \,\le\, 4 C_3 T^{1/8}\,.
\]
Since $T > 0$ can be replaced by any smaller time and the constant $C_3$
does not depend on $T$, it follows that 
\begin{equation}\label{bomest}
  \|\hat\omega_1\|_\times \,:=\, \sup_{t \in (0,T)} t^{1/8} \|\hat\omega_1(t)\|_{L^{4/3}}
  < \infty\,, \qquad 
  \|\omega_2\|_\times \,:=\, \sup_{t \in (0,T)} t^{1/8}\|\omega_2(t)\|_{L^{4/3}}
  \,<\, \infty\,.
\end{equation}

We also need to control the $L^1$ norm of our solution. Since $\hat\omega_1 =
\Phi_1(\hat\omega_1,\omega_2)$, we can estimate the right-hand side of \eqref{FP1}
as in the proof of Lemma~\ref{lem:Phibounds}. Using \eqref{bomest} and the bounds 
on $F_1, F_2$ established in Lemma~\ref{lem:estF}, we find
\begin{align*}
  \|\hat\omega_1(t)\|_{L^1} \,&\le\, \int_0^t \frac{C}{(t-s)^{1/2}}
  \bigl(\|F_1(s)\|_{L^1} + \|F_2(s)\|_{L^1}\bigr)\dd s \\
  \,&\le\, \int_0^t \frac{C}{(t-s)^{1/2}}\Bigl(\alpha^2 + \alpha s^{-1/8}
  \|\hat\omega_1\|_\times + \alpha s^{-3/8}\|\omega_2\|_\times + 
  s^{-1/4}\bigl(\|\hat\omega_1\|_\times + \|\omega_2\|_\times\bigr)^2\Bigr)\dd s\,,
\end{align*}
so that $\|\hat\omega_1(t)\|_{L^1} \le Ct^{1/8}$ for some constant depending
on $\alpha$. Similarly,
\begin{align*}
  \|\omega_2(t) - \alpha S_2(t)\delta_z\|_{L^1} \,&\le\, C \int_0^t \Bigl\{\|F_3(s)\|_{L^1}
  + \frac{\|F_4(s)\|_{L^1}}{(t-s)^{1/2}}\Bigr\}\dd s \\
  \,&\le\, C\alpha^2 t + C \alpha t^{3/8}\bigl(\|\hat\omega_1\|_\times + \|\omega_2\|_\times
  \bigr) + C t^{1/4}\bigl(\|\hat\omega_1\|_\times + \|\omega_2\|_\times\bigr)^2\,,
\end{align*}
hence $\|\omega_2(t) - \alpha S_2(t)\delta_z\|_{L^1} \le C t^{1/4}$. Note
that $S_2(t)\delta_z$ is uniformly bounded in $L^1(\R^2_+)$, but does not
converge strongly to zero as $t \to 0$.

According to \eqref{omudecomp}, we define $\omega(t) = \bar\omega_1(t) +
\hat\omega_1(t) + \omega_2(t)$, where $\bar\omega_1(t) = \alpha S_1(t)\delta_z$.
By construction $\omega \in C^0((0,T),L^1(\R^2_+) \cap L^{4/3}(\R^2_+))$, and
we have just shown that $\|\omega(t)\|_{L^1}$ is uniformly bounded for
$t \in (0,T)$, which is the first condition in \eqref{maincond}. Moreover it
follows from \eqref{bomest} that
\[
  \sup_{t \in (0,T)} t^{1/8} \bigl\|\omega(t) - \alpha G_t(\cdot-z)\bigr\|_{L^{4/3}}
  \,\le\, \|\hat\omega_1\|_\times  + \|\omega_2\|_\times  \,<\, \infty\,,
\]
which is the second condition in \eqref{maincond}, with $\beta = 1/8$. 
Finally, in view of Lemma~\ref{lem:Mcont}, we have
\[
  \omega(t) \,=\, \alpha S(t)\delta_z + \hat\omega_1(t) + \bigl(\omega_2(t) 
  - \alpha S_2(t)\delta_z\bigr) \,\weakto\, \alpha \delta_z\,, \qquad
  \text{as }\, t \to 0\,,
\]
because $\hat\omega_1(t)$ and $\omega_2(t) - \alpha S_2(t)\delta_z$
converge strongly to zero in $L^1(\R^2_+)$. Therefore $\omega(t)$ satisfies all conditions
stated in Theorem~\ref{thm:main}, and the calculations above also show that
estimate \eqref{firstasym} holds. 

It remains to verify that $\omega(t)$ is indeed a solution of the vorticity
equation in the upper half-plane, which satisfies the integral condition
\eqref{noslip}. This can be done by undoing the various changes of variables
that were used earlier in this section to make the integral equation
\eqref{IntEqIni} amenable to a fixed point argument. The procedure is rather
classical, and the details can be left to the reader; in particular, we take
here for granted that the solution $\bom(t)$ of \eqref{def:hatom1}--\eqref{def:om2}
is smooth for positive times. Starting from the integral equation
\eqref{def:hatom1} and recalling the definition of the two-parameter semi-group
$\Sigma_\alpha$, it is straightforward to verify that $\hat \omega_1$ satisfies
the linearized equation \eqref{hatR2} which, according to \eqref{F12id}, can be
written in the equivalent form
\[
  \partial_t \hat \omega_1(x,t) \,=\, \Delta \hat \omega_1 (x,t)- \div (\chi(x)\,
  u (x,t)\,\omega(x,t)),\qquad x\in \R^2\,,\quad t\in (0,T)\,. 
\]
Now defining $\omega_1 := \bar{\omega}_1 + \hat{\omega}_1$, we deduce that
$\omega_1$ is a solution of the evolution equation \eqref{R2vort} in the whole
plane $\R^2$, which implies that the restriction of $\omega_1$ to the upper
half-plane satisfies the integral equation \eqref{om1int}. Finally, 
summing up the integral relations \eqref{om1int} and \eqref{om2int}, we
conclude that $\omega := \omega_1+ \omega_2$ satisfies the integral equation
\eqref{IntEqIni} for all $ \in (0,T)$, and is therefore a mild solution of
\eqref{2Dvort} in $L^1_\perp(\R^2_+)$. 
 
\medskip\noindent{\bf Uniqueness part\:}
Conversely, let $\omega\in C^0((0,T), L^1_\perp(\R^2_+) \cap L^{4/3}(\R^2_+))$
be a solution of \eqref{2Dvort} in the sense of Definition~\ref{def:sol}, and
assume that $\omega$ satisfies the hypotheses of Theorem~\ref{thm:main}.  Using
Lemma~\ref{lem:inteq}, we decompose $\omega = \bar \omega_1 + \hat \omega_1 +
\omega_2$ as in \eqref{omudecomp}, and we know from Lemma~\ref{lem:decomp}
that the pair $\bom := (\hat\omega_1,\omega_2)$ satisfies the integral
equations \eqref{def:hatom1}--\eqref{def:om2}.  It follows from the second assumption
in \eqref{maincond} that the pair $\bom$ belongs to the space $\cZ_T$ defined
in \eqref{cZdef}, and by construction $\bom$ is a fixed point of the map $\Phi$
given by \eqref{FP1}, \eqref{FP2}.

Now, the second assumption in \eqref{maincond} implies that 
\[
  \|\hat\omega_1 + \omega_2\|_\star \,\equiv\, \sup_{t\in (0,T)} t^{1/4}
  \|\omega(t) -\bar\omega_1(t)\|_{L^{4/3}} \,\lesssim\, T^{1/4-\beta}\,,
\]
so that $\|\hat\omega_1 + \omega_2\|_\star$ can be made arbitrarily small by taking
$T$ small enough. In addition, since $\omega_2 = \Phi_2(\hat\omega_1,\omega_2)$, 
it follows from \eqref{Phibd2} that
\[
  \|\omega_2\|_\star \,\lesssim\, \alpha^2 T^{9/8} + \alpha \bigl(T^{1/8}
  + T^{1/4} \|\hat\omega_1 + \omega_2\|_\star\bigr) + \|\hat\omega_1 + \omega_2\|_\star^2\,,
\]
which means that $\|\omega_2\|_\star$, hence also $\|\hat\omega_1\|_\star$,
can be made arbitrarily small by taking $T$ small enough. So, if $T > 0$
is sufficiently small, the solution $\bom = (\hat\omega_1,\omega_2)$ is
contained in the ball $\cB_\rho$ in which the map $\Phi$ is a strict
contraction. Summarizing, we have shown that any solution of \eqref{IntEq}
satisfying the assumptions of Theorem~\ref{thm:main} coincides, for sufficiently
small times, with the particular solution constructed by our fixed point
argument. This proves uniqueness of the solution for short times, hence
for all times because for any $t > 0$ the solution belongs to the
space $L^1_\perp(\R^2_+)$ where local uniqueness is known to hold,
see \cite{Abe}. \QED

%%%%%%%%%%%%%%%%%%%%%%%%%%%%%%%%%%%%%%%%%%%%%%%%%%%%%%%%%%%%%%%%%%%%%%%%%%%
%%%%%%%%%%%%%%%%%%%%%%%%%%%%%%%%%%%%%%%%%%%%%%%%%%%%%%%%%%%%%%%%%%%%%%%%%%%

\section{Additional properties of the solution}
\label{sec5}

In this final section we establish a few additional properties of the
solution of \eqref{IntEqIni} whose existence and uniqueness is established
in Section~\ref{sec4}. In particular, we conclude the proof of
Theorem~\ref{thm:main} in Section~\ref{ssec51} and of Proposition~\ref{prop:timedecay}
in Section~\ref{ssec52}. The proof of Proposition~\ref{prop:speed}
is a relatively straightforward calculation that is presented in Section~\ref{ssec53}. 

\subsection{Scale invariant estimates and short time behavior}\label{ssec51}

We first show that the solution of \eqref{IntEqIni} satisfies the
usual scale invariant estimates not only in $L^1(\R^2_+)$ and $L^{4/3}(\R^2_+)$,
but in all Lebesgue spaces. 

\begin{lem}\label{lem:genprop}
Under the assumptions of Theorem~\ref{thm:main}, the solution of
\eqref{IntEqIni} satisfies
\begin{equation}\label{genom}
  \Omega_p \,:=\, \sup_{t \in (0,T)} t^{1-1/p}\|\omega(t)\|_{L^p} \,<\, \infty\,, \qquad
 U_q \,:=\, \sup_{t \in (0,T)} t^{1/2-1/q}\|u(t)\|_{L^q} \,<\, \infty\,,
\end{equation}
for all $p \in [1,\infty]$ and all $q \in (2,\infty]$.
%In addition $\|\omega(t) - \alpha S(t)\delta_z\|_{L^1} \to 0$ as $t \to 0$. 
\end{lem}

\begin{proof}
This is a classical argument that we reproduce here for completeness.
We start from the integral equation \eqref{IntEqIni} which takes the
form $\omega(t) = \alpha S(t)\delta_z - I(t)$, where $I(t)$ denotes
the integral term. We know from \eqref{Stokes3} that $\sup_{t \in (0,T)}
t^{1-1/p}\|S(t)\delta_z\|_{L^p} < \infty$, so we only need to bound
$\|I(t)\|_{L^p}$ for $p \in [1,\infty]$. As was already observed in
Remark~\ref{rem:2D}, the second inequality in \eqref{maincond} implies
that $\Omega_{4/3} < \infty$, so that $U_4 < \infty$ by Lemma~\ref{lem:HLS}. 

We first estimate the integral term $I(t) = \int_0^t S(t-s)\div\bigl(u(s)\omega(s)\bigr)
\dd s$ for $p \in [1,2)$. Using the bound \eqref{Stokes2} on the Stokes
semigroup, we find as in Remark~\ref{rem:mild}\:
\[
  \|I(t)\|_{L^p} \,\lesssim\, \int_0^t \frac{\|u(s)\omega(s)\|_{L^1}}{(t-s)^{3/2-1/p}}
  \dd s \,\lesssim\, \int_0^t \frac{U_4\,\Omega_{4/3}}{(t-s)^{3/2-1/p}s^{1/2}}\dd s
  \,\lesssim\, \frac{\Omega_{4/3}^2}{t^{1-1/p}}\,, \qquad t \in (0,T)\,.
\]
This shows that $\Omega_p < \infty$ for $p \in [1,2)$, and Lemma~\ref{lem:HLS}
implies that $U_q < \infty$ for $q \in (2,\infty)$. 

Next, we assume that $p \in [2,\infty)$ and we take $r \in (1,2)$ such that
$1/r < 1/2 + 1/p$. Using the $L^p$--$L^r$ estimate for the operator $S(t)\nabla$, see
Remark~\ref{rem:Stokes2}, we obtain as before
\[
  \|I(t)\|_{L^p} \,\lesssim\, \int_0^t \frac{\|u(s)\omega(s)\|_{L^r}}{(t-s)^{1/2+1/r-1/p}}
  \dd s \,\lesssim\, \int_0^t \frac{U_q\,\Omega_\theta}{(t-s)^{1/2+1/r-1/p}
  s^{3/2-1/r}}\dd s \,\lesssim\, \frac{\Omega_\theta^2}{t^{1-1/p}}\,,
\]
where we used H\"older's inequality $\|u \omega\|_{L^r} \le \|u\|_{L^q}\|\omega\|_{
L_\theta}$ and the Hardy--Littlewood--Sobolev bound $\|u\|_{L^q} \lesssim \|\omega\|_{L_\theta}$
with
\[
  \frac{1}{q} + \frac{1}{\theta} \,=\, \frac{1}{r} \quad\text{and}\quad
  \frac{1}{q} \,=\, \frac{1}{\theta} - \frac{1}{2}\,, \qquad \text{which leads to}
  \qquad  q \,=\, \frac{4r}{2-r} \quad\text{and}\quad \theta \,=\, \frac{4r}{2+r}\,.
\]
We have thus shown that $\Omega_p < \infty$ for all $p < \infty$, and using
Lemma~\ref{lem:HLS} again we deduce that $U_\infty < \infty$. 

Finally, if $p = \infty$, we observe that 
\[
  \|I(t)\|_{L^\infty} \,\lesssim\, \int_0^{t/2} \frac{\|u(s)\omega(s)\|_{L^1}}{(t-s)^{3/2}}
  \dd s + \int_{t/2}^t \frac{\|u(s)\omega(s)\|_{L^4}}{(t-s)^{3/4}}\dd s
  \,\lesssim\, \frac{1}{t}\bigl(\Omega_{4/3}^2 + U_\infty\Omega_4\bigr)\,,
\]
which shows that $\Omega_\infty < \infty$. The proof is thus complete. 
\end{proof}

\subsection{Localization properties}\label{ssec52}

For $0 < t < T$ the vorticity $\omega(x,t)$ given by Theorem~\ref{thm:main} is a
smooth solution of the advection-diffusion equation \eqref{2Dvort} in $\R^2_+$, which
can be studied using energy estimates. Since $\omega$ does not satisfy any simple
boundary condition, we use spatial weights that vanish at the boundary together
with their first order derivatives. As a first example, we show that the vorticity
is strongly localized in the vertical direction. 

\begin{prop}\label{prop:weight}
Under the assumptions of Theorem~\ref{thm:main}, the solution of
\eqref{IntEqIni} satisfies, for any $m \in \N$, 
\begin{equation}\label{est:weight}
  \sup_{t \in (0,T)} \int_{\R^2_+} x_2^m\,|\omega(x,t)|\dd x \,<\, \infty\,.
\end{equation}
\end{prop}

\begin{proof}
Since $\omega(t)$ is uniformly bounded in $L^1(\R^2_+)$ for $t \in (0,T)$, 
we can assume without loss of generality that $m \ge 2$. Given $\lambda > 0$,
let $\psi_\lambda : \R \to \R_+$ be the function defined by
\[
  \psi_\lambda(\omega) \,=\, \sqrt{\lambda^2 + \omega^2} - \lambda\,, \qquad
  \forall\,\omega \in \R\,.
\]
Clearly $\psi_\lambda$ is strictly convex and $0 \le \psi_\lambda(\omega) \le
|\omega|$ for all $\omega \in \R$. Given $\epsilon > 0$ we also introduce
the spatial weight $\chi_\epsilon : \R^2_+ \to \R_+$ defined by
\begin{equation}\label{defchieps}
  \chi_\epsilon(x) \,=\, \frac{x_2^m}{1 + \epsilon x_2^m}\,, \qquad
  \forall\,x = (x_1,x_2) \in \R^2_+\,.
\end{equation}
Obviously $0 \le \chi_\epsilon(x) \le \min(x_2^m,1/\epsilon)$, and there exists
a constant $C_0 = C_0(m) > 0$ such that
\begin{equation}\label{chiprop}
  |\nabla \chi_\epsilon(x)| \,\le\, C_0\,\chi_\epsilon(x)^{1-1/m}\,, \qquad
  \Delta \chi_\epsilon(x) \,\le\, C_0\,\chi_\epsilon(x)^{1-2/m}\,, \qquad
   \forall\,x \in \R^2_+\,,\quad \forall\,\epsilon > 0\,.
\end{equation}
Our main tool is the energy function defined by 
\begin{equation}\label{def:Eepslam}
  E_{\epsilon,\lambda}(t) \,=\, \int_{\R^2_+} \chi_\epsilon(x)\,\psi_\lambda(
  \omega(x,t))\dd x\,, \qquad t \in (0,T)\,.
\end{equation}

We differentiate \eqref{def:Eepslam} with respect to $t$ and integrate by
parts, using the important observation that the weight $\chi_\epsilon$ and
its gradient $\nabla\chi_\epsilon$ vanish on the boundary. This gives
\begin{align*}
  E_{\epsilon,\lambda}'(t) \,&=\, \int_{\R^2_+} \chi_\epsilon\psi_\lambda'(\omega)
  \partial_t \omega\dd x \,=\, \int_{\R^2_+} \chi_\epsilon\psi_\lambda'(\omega)
  \bigl(\Delta\omega - (u\cdot\nabla)\omega\bigr)\dd x \\
  \,&=\, \int_{\R^2_+} \bigl(\Delta \chi_\epsilon + u\cdot\nabla\chi_\epsilon\bigr)
  \psi_\lambda(\omega)\dd x - \int_{\R^2_+} \chi_\epsilon \psi_\lambda''(\omega)
  |\nabla\omega|^2\dd x \\
  \,&\le\, \int_{\R^2_+} \bigl(\Delta \chi_\epsilon + u\cdot\nabla\chi_\epsilon\bigr)
  \psi_\lambda(\omega)\dd x\,.
\end{align*}
The right-hand side will be estimated in terms of the quantities
\begin{equation}\label{mincond}
  \cM \,:=\, \sup_{t \in (0,T)}\|\omega(t)\|_{L^1}\,, \qquad \text{and}\qquad
  \cU \,:=\, \sup_{t \in (0,T)}t^{1/2}\|u(t)\|_{L^\infty}\,.
\end{equation}
Using \eqref{chiprop}, H\"older's inequality, and the bound $\psi_\lambda(\omega) \le
|\omega|$, we easily find
\begin{align*}
  \int_{\R^2_+} \Delta \chi_\epsilon\,\psi_\lambda(\omega)\dd x \,&\le\,
  C_0 \int_{\R^2_+} \chi_\epsilon^{1-2/m}\psi_\lambda(\omega)\dd x \,\le\,
  C_0\,E_{\epsilon,\lambda}^{1-2/m}\cM^{2/m}\,,\\
  \int_{\R^2_+} |u\cdot\nabla \chi_\epsilon|\,\psi_\lambda(\omega)\dd x \,&\le\,
  C_0\,\|u\|_{L^\infty}\int_{\R^2_+} \chi_\epsilon^{1-1/m}\psi_\lambda(\omega)\dd x \,\le\,
  C_0\,\|u\|_{L^\infty}\,E_{\epsilon,\lambda}^{1-1/m}\cM^{1/m}\,.
\end{align*}
Since $\|u(t)\|_{L^\infty} \le \cU t^{-1/2}$, we arrive at the differential
inequality
\begin{equation}\label{Ediffeq}
  E_{\epsilon,\lambda}'(t) \,\le\, C_0\Bigl(E_{\epsilon,\lambda}(t)^{1-2/m}\cM^{2/m}
  + \frac{\cU}{\sqrt{t}}\,E_{\epsilon,\lambda}(t)^{1-1/m}\cM^{1/m}\Bigr)\,,
  \qquad t \in (0,T)\,.
\end{equation}

To simplify the analysis, we define $f(t) = E_{\epsilon,\lambda}(t)^{2/m}\cM^{-2/m}$
and we observe that
\[
  f'(t) \,\le\, \frac{2C_0}{m}\Bigl(1 + \frac{\cU}{\sqrt{t}}\,f(t)^{1/2}\Bigr)
  \,\le\, \frac{C_1}{\sqrt{t}}\bigl(1 + f(t)\bigr)^{1/2}\,, \qquad t \in (0,T)\,,
\]
where $C_1 = (4C_0/m) \max\bigl(\sqrt{T},\cU\bigr)$. Integrating over
the time interval $(t_0,t)$, we obtain
\[
  \bigl(1 + f(t)\bigr)^{1/2} \,\le\, \bigl(1 + f(t_0)\bigr)^{1/2} + C_1\bigl(\sqrt{t}
  - \sqrt{t_0}\bigr)\,, \qquad 0 < t_0 < t < T\,,
\]
and returning to the energy \eqref{eq:Eepslam} we arrive at
\begin{equation}\label{eq:Eepslam}
  \bigl(\cM^{2/m} + E_{\epsilon,\lambda}(t)^{2/m}\bigr)^{1/2} \,\le\,
  \bigl(\cM^{2/m} + E_{\epsilon,\lambda}(t_0)^{2/m}\bigr)^{1/2} 
  + C_1\,\cM^{1/m}\bigl(\sqrt{t} - \sqrt{t_0}\bigr)\,.
\end{equation}

It remains to derive \eqref{est:weight} from \eqref{eq:Eepslam} by sending
the parameters $t_0$, $\lambda$, and $\epsilon$ to zero. We first observe that
\begin{equation}\label{eq:Eepslam2}
  E_{\epsilon,\lambda}(t_0) \,\le\, \int_{\R^2_+} \chi_\epsilon |\omega(t_0)|\dd x
  \,\le\, \int_{\R^2_+} \chi_\epsilon |\omega(t_0) - \alpha S(t_0)\delta_z|\dd x
  + \alpha \int_{\R^2_+} \chi_\epsilon |S(t_0)\delta_z|\dd x\,.
\end{equation}
The first integral in the right-hand side converges to zero as $t_0 \to 0$
because $\chi_\epsilon$ is bounded and we know from the short-time estimate
\eqref{firstasym} that $\|\omega(t_0) - \alpha S(t_0)\delta_z\|_{L^1} \to
0$. Using the expression \eqref{Kdecomp}, the limit of the second integral can
be computed explicitly and is found to be equal to
$\chi_\epsilon(z) = (1+\epsilon)^{-1}$. Note that, in view of \eqref{K2bd1} and
\eqref{defchieps}, the contribution of the boundary term $K_2(\cdot,z,t_0)$ to
this calculation is $\cO(t_0^{m/2})$ as $t_0 \to 0$. We thus obtain
\begin{equation}\label{eq:Eepsbd}
  \bigl(\cM^{2/m} + E_{\epsilon,\lambda}(t)^{2/m}\bigr)^{1/2} \,\le\, \bigl(\cM^{2/m}
  + \alpha\bigr)^{1/2} + C_1\,\cM^{1/m}t^{1/2}\,, \qquad 0 < t < T\,.
\end{equation}
In a second step we invoke the monotone convergence theorem to take the limits
$\lambda \to 0$, and then $\epsilon \to 0$. Since
\[
  E_{\epsilon,\lambda}(t) \,\xrightarrow[\genfrac{}{}{0pt}{}{\lambda \to 0}{\epsilon \to 0}]{}\,
  \int_{\R^2_+}x_2^m |\omega(x,t)|\dd x\,, \qquad 0 < t < T\,,
\]
we deduce \eqref{est:weight} from \eqref{eq:Eepsbd}.  
\end{proof}

\begin{cor}\label{cor:ulp}
Under the assumptions of Theorem~\ref{thm:main}, the velocity field
satisfies
\begin{equation}\label{ulpbd}
  \sup_{t \in (0,T)}\|u(t)\|_{L^p} \,<\, \infty\,, \qquad \forall\, p \in (1,2)\,.
\end{equation}
As a consequence, the kinetic energy $\|u(t)\|_{L^2}^2$ is finite for all
$t \in (0,T)$. 
\end{cor}

\begin{proof}
For any $p \in (1,2)$, it follows from \eqref{uLp} that
\[
  \|u(t)\|_{L^p} \,\lesssim\, \int_{\R^2_+} x_2^{2/p-1} \,|\omega(x,t)|\dd x
  \,\lesssim\, \int_{\R^2_+} (1+x_2)\,|\omega(x,t)|\dd x\,,
\]
and we know from Proposition~\ref{prop:weight} that the right-hand side is uniformly
bounded for $t \in (0,T)$, which gives \eqref{ulpbd}. Since we already know
that $u(t) \in L^q(\R^2_+)$ for all $q > 2$, we conclude that $u(t) \in L^2(\R^2_+)$
by interpolation. 
\end{proof}

\begin{rem}\label{rem:energy}
It is not difficult to prove that $\|u(t)\|_{L^2}^2 = \cO(|\log t|)$ as
$t \to 0$, which is in contrast with the uniform bound \eqref{ulpbd} for
$p \in (1,2)$. Indeed, a direct calculation in Appendix~\ref{appB} shows that
the kinetic energy of the Stokes solution $\alpha S(t)\delta_z$ diverges
logarithmically as $t \to 0$, and using the short time estimate
\eqref{firstasym} it is easy to verify that the kinetic energy of the nonlinear
correction $\omega(t) - \alpha S(t)\delta_z$ vanishes as $t \to 0$.
\end{rem}

\begin{proof}[\bf Proof of Proposition~\ref{prop:timedecay}]
Since the Cauchy problem for \eqref{2Dvort} is locally well-posed in
$L^1_\perp(\R^2_+)$, see \cite{Abe}, we can consider the maximal extension 
$\omega \in C^0\bigl((0,T_*),L^1_\perp(\R^2_+)\cap L^{4/3}(\R^2_+)\bigr)$
of the solution given by Theorem~\ref{thm:main}, where $T_* \in [T,+\infty]$
is the maximal existence time. We know from Corollary~\ref{cor:ulp}
that the associated velocity $u(t) = \BS[\omega(t)]$ belongs to $L^2(\R^2_+)$ for all
$t \in (0,T_*)$, see also Remark~\ref{rem:energy}. Given any $t_0 \in (0,T)$,
this means that $u(t)$ coincides for $t \ge t_0/2$ with the unique finite-energy
solution $\bar u(t)$ of the Navier-Stokes equations in $L^2(\R^2_+)$ such that
$\bar u(t_0/2) = u(t_0/2)$. This solution $\bar u$ is defined on the time interval
$[t_0/2,+\infty)$, and both the energy $\|\bar u(t)\|_{L^2}$ and the enstrophy
$\|\nabla \bar u(t)\|_{L^2}$ are uniformly bounded for $t \ge t_0$; see, for instance,
\cite[Chapter~9]{CF}. Returning to the solution of \eqref{IntEqIni}, we observe
that, for $p \in [1,2)$ and $t \in (0,T_*)$, 
\begin{equation}\label{omglobal}
  \|\omega(t)\|_{L^p} \,\lesssim\, \alpha\|S(t)\delta_z\|_{L^p} +
  \int_0^{t_0}\frac{\|u(s)\|_{L^4}\|\omega(s)\|_{L^{4/3}}}{(t-s)^{3/2-1/p}}\dd s + 
  \int_{t_0}^t\frac{\|\bar u(s)\|_{L^2}\|\nabla\bar u(s)\|_{L^2}}{(t-s)^{3/2-1/p}}\dd s\,.
\end{equation}
This inequality shows in particular that the norms $\|\omega(s)\|_{L^1}$ and
$\|\omega(s)\|_{L^{4/3}}$ cannot blow up in finite time, which implies that
$T_* = +\infty$. The local solution in Theorem~\ref{thm:main} can therefore be
extended to a global solution. 

We also know from Corollary~\ref{cor:ulp} that $\bar u(t_0) \equiv u(t_0) \in
L^p(\R^2_+)$ for any $p > 1$. It thus follows from \cite[Theorem~1]{BM}
that $\|\bar u(t)\|_{L^2} = \cO(t^{-1/2+\epsilon})$ as $t \to +\infty$, for any
$\epsilon > 0$, and classical energy estimates then imply that
$\|\nabla \bar u(t)\|_{L^2}= \cO(t^{-1+\epsilon})$ as $t \to +\infty$.
Returning to \eqref{omglobal} with $p = 1$ and using the estimate
\eqref{StokesLp} for the first term in the right-hand side, we see that
$\|\omega(t)\|_{L^1} = \cO(t^{-1/2})$ as $t \to +\infty$. Since the kinetic
energy of the Stokes solution $\|\alpha S(t)\delta_z\|_{L^2}^2$ decays like $t^{-1}$ as
$t \to +\infty$, as computed in Appendix~\ref{appB}, similar arguments show
that $\|\BS[\omega(t)]\|_{L^2} = \cO(t^{-1/2})$ as $t\to +\infty$. This concludes
the proof of \eqref{timedecay}.
\end{proof}

Using the same techniques as in Proposition~\ref{prop:weight}, we
next show that the solution of \eqref{IntEqIni} is concentrated
for small times in a tiny neighborhood of the point $z = (0,1)$ or
of the boundary. 

\begin{prop}\label{prop:weight2}
Under the assumptions of Theorem~\ref{thm:main}, the solution of
\eqref{IntEqIni} satisfies, for any $\epsilon > 0$ and any $m \in \N$,
\begin{equation}\label{est:weight2}
  \int_{\Omega_\epsilon} \,|\omega(x,t)|\dd x \,=\, \cO(t^{m/2})\,,
  \qquad \text{as }t \to 0\,,
\end{equation}
where $\Omega_\epsilon = \bigl\{x \in \R^2_+\,;\, |x-z| \ge \epsilon
\text{ and } x_2 \ge \epsilon\bigr\}$. 
\end{prop}

\begin{proof}
We proceed as in Proposition~\ref{prop:weight}, with the difference that
$\epsilon > 0$ is now a fixed parameter, that we take sufficiently small. Given
$m \in \N$ with $m \ge 2$, we choose a localization function $\eta : [0,+\infty) \to [0,1]$
of class $C^{m-1,1}$ such that $\eta(r) = 0$ for $r \le 1/2$, $\eta(r) = 1$ for
$r \ge 1$, and
\begin{equation}\label{etaweight}
  |\eta'(r)| \,\le\, C\,\eta(r)^{1-1/m}\,, \qquad |\eta''(r)| \,\le\,
  C\,\eta(r)^{1-2/m}\,, \qquad \forall\, r \ge 0\,,
\end{equation}
for some positive constant $C$. Such a function exists, and one can take
$\eta(r) = (r-1/2)_+^m$ in a neighborhood of $r = 1/2$. Our weight function
$\chi_\epsilon : \R^2_+ \to [0,1]$ is defined by
\[
  \chi_\epsilon(x) \,=\, \eta\Bigl(\frac{x_2}{\epsilon}\Bigr)\,
  \eta\Bigl(\frac{|x-z|}{\epsilon}\Bigr)\,, \qquad \forall x \in \R^2_+\,.
\]
By construction $\chi_\epsilon$ is equal to $1$ on $\Omega_\epsilon$, and
vanishes identically in a neighborhood of size $\epsilon/2$ of the boundary
and in a similar neighborhood of the point $z = (0,1)$. Moreover, in view of
\eqref{etaweight}, $\chi_\epsilon$ satisfies \eqref{chiprop} for some positive
constant $C_0 = C_0(m,\epsilon)$. Therefore, the energy function defined by
\eqref{def:Eepslam} again satisfies the differential inequality \eqref{Ediffeq}.

The main difference is that the energy function $E_{\epsilon,\lambda}(t_0)$ now
converges to zero as $t_0 \to 0$. Indeed, in the right-hand side of
\eqref{eq:Eepslam2}, the first term is $\cO(t_0^{1/8})$ thanks to
\eqref{firstasym} and the last term converges to zero by direct calculation,
because the Stokes solution $S(t_0)\delta_z$ is asymptotically concentrated near
the point $z$ or near the boundary, where the weight $\chi_\epsilon$ vanishes
identically. As a consequence, the function $f(t) = E_{\epsilon,\lambda}
(t)^{2/m}\cM^{-2/m}$ vanishes at the origin and satisfies
\[
  f'(t) \,\le\, C_2\Bigl(1 + \frac{\cU}{\sqrt{t}}\,f(t)^{1/2}\Bigr)\,, \qquad
  t \in (0,T)\,, 
\]
where $C_2 = 2C_0/m$. Using Lemma~\ref{lem:Gron} below, we deduce that
$f(t) \le C_3t$ for all $t \in [0,T]$, where the constant satisfies 
$C_3 = C_2\bigl(1+\cU C_3^{1/2}\bigr)$, and this means that 
$E_{\epsilon,\lambda}(t) \le \cM\bigl(C_3t\bigr)^{m/2}$. Finally,
recalling that $\chi_\epsilon = 1$ on $\Omega_\epsilon$, we conclude
that 
\[
  \int_{\Omega_\epsilon} \,|\omega(x,t)|\dd x \,\le\, \int_{\R^2_+}
  \chi_\epsilon(x)\,|\omega(x,t)|\dd x \,=\, \lim_{\lambda \to 0}
  E_{\epsilon,\lambda}(t) \,\le\, \cM\bigl(C_3t\bigr)^{m/2}\,,
  \qquad t \in (0,T)\,,
\]
which is the desired result. 
\end{proof}

The localization estimate \eqref{est:weight2} shows that, up to
corrections of size $\cO(t^\infty)$, the solution of \eqref{IntEqIni}
can be decomposed for small times into a vortex that is concentrated
near the initial position $z$, and a boundary layer term. The next
statement makes a link with the decomposition of the vorticity
that was used in Section~\ref{sec4}. 

\begin{cor}\label{cor:local}
If the solution of \eqref{IntEqIni} is decomposed as $\omega(t) =
\omega_1(t) + \omega_2(t)$, where $\omega_1(t)$ is defined in 
\eqref{om1int} and $\omega_2(t)$ in \eqref{om2int}, the following
estimate holds for any $\epsilon > 0$ and any $m \in \N$\:
\begin{equation}\label{est:om1om2}
  \int_{|x-z| \ge \epsilon} \,|\omega_1(x,t)|\dd x +
  \int_{x_2 \ge \epsilon} \,|\omega_2(x,t)|\dd x \,=\, \cO(t^{m/2})\,,
  \qquad \text{as }t \to 0\,.
\end{equation}
\end{cor}

\begin{proof}
The cut-off function \eqref{chidef} depends on a parameter $r_0 \in (0,1)$ that
was fixed arbitrarily, but Proposition~\ref{prop:weight2} shows that the choice
of $r_0$ is irrelevant for the short-time asymptotics. Indeed, if $\chi_1$
denotes the same function as in \eqref{chidef} with $r_0$ replaced by $r_1 \neq r_0$,
then
\[
  \Bigl\|\int_0^t S_1(t{-}s)\div \bigl((\chi{-}\chi_1) u(s) \omega(s)\bigr)\dd s
  \Bigr\|_{L^1} \,\lesssim\, U_\infty\int_0^t \frac{\|(\chi{-}\chi_1)\omega(s)\|_{L^1}
  }{(t-s)^{1/2}s^{1/2}}\dd s \,=\, \cO(t^\infty)\,,
\]
as $t \to 0$, because the function $\chi-\chi_1$ is supported away from the
point $z$ and of the boundary. The above estimate implies that the difference
between two possible functions $\omega_1$ associated with different choices of
the parameter $r_0$ remains $\cO(t^{\infty})$ for small times. In particular,
given any small $\epsilon > 0$, one can use the freedom above to take
$r_0 = \epsilon/2$. Now the integral equation \eqref{om1int} can be written
in the form
\[
  \omega_1(x,t) \,=\, \alpha G_t (x-z) - \int_0^t\int_{\R^2} \chi(y)\,
  \omega_1(y,t)\,u(y,t) \cdot \nabla G_{t-s}(x-y) \dd y \dd s\,.
\]
If $|x-z| \ge \epsilon$, then $|x-y| \geq \epsilon/2$ for all
$y\in \mathrm{supp}\; \chi_1$, and it is then straightforward to check that both
terms in the right-hand side are bounded by $C e^{-c|x-z|^2/t}$ for some
positive constants $C$ and $c$. This proves the estimate on $\omega_1$ in
\eqref{est:om1om2}. The same strategy can be applied to $\omega_2$ in
\eqref{om2int}, where the term involving $(1-\chi)$ can be localized in an
$(\epsilon/2)$-neighborhood of the boundary. Using in addition the properties of
the operator $S_2(t)$, one finds that the $L^1$ norm of $\omega_2(t)$ outside an
$\epsilon$-neighborhood of the boundary is negligible, as asserted in in
\eqref{est:om1om2}.
\end{proof}

\subsection{Translation speed for short times}\label{ssec53}

This section is devoted to the proof of Proposition~\ref{prop:speed}. 
Differentiating \eqref{def:Z} with respect to time, we find
\[
  Z'(t) \,=\, V(t) + \frac{1}{\alpha} \int_{\R^2_+} \omega
  \Bigl\{x\,\bigl(\Delta\chi + u\cdot\nabla\chi\bigr) + 2\nabla\chi\Bigr\}\dd x\,,
  \qquad V(t) \,=\, \frac{1}{\alpha} \int_{\R^2_+} \chi\,u\,\omega\dd x\,.
\]
Using Proposition~\ref{prop:weight2} and the rough estimate $\|u(t)\|_{L^\infty}
\le C t^{-1/2}$, we see that the terms involving derivatives of $\chi$ are
$\cO(t^\infty)$ as $t \to 0$, so that we only need to compute $V(t)$. We decompose
$\omega(t) = \omega_1(t) + \omega_2(t)$ as in Section~\ref{sec4}, and applying
Corollary~\ref{cor:local} we obtain
\[
  V(t) \,=\, \frac{1}{\alpha}\int_{\R^2_+} \chi(x)\,u(x,t)\,\omega_1(x,t)\dd x +
  \cO(t^\infty)\,,
\]
because by \eqref{est:om1om2} the $L^1$ norm of $\chi \omega_2(t)$ is extremely
small as $t \to 0$. Next, applying Lemma~\ref{lem:2BS}, we further decompose
$u = \BSS[\widetilde{\omega}_1] + \BSS[\widetilde{\omega}_2]$, where $\BSS$
denotes the Biot--Savart operator in the whole space $\R^2$ and
$\widetilde{\omega}_i$ the extension by zero of $\omega_i$ to the whole
space. Note that the difference between $\widetilde \omega_1 $ and $\omega_1$ is
supported in $\R^2_-$ and of size $\cO(t^\infty)$ thanks to Corollary \ref{cor:local}.
We observe that
\[
  \int_{\R^2_+} \chi \BSS[\widetilde{\omega}_1]\,\widetilde{\omega}_1\dd x \,=\, \int_{\R^2}
  \BSS[\widetilde{\omega}_1]\,\widetilde{\omega}_1\dd x - \int_{\R^2} (1-\chi)
  \BSS[\widetilde{\omega}_1]\,\widetilde{\omega}_1\dd x \,=\, \cO(t^\infty)\,,
\]
because the first integral in the right-hand side vanishes identically thanks to
Lemma~\ref{lem:2BS}, whereas the second one is small by Corollary~\ref{cor:local}.
We thus arrive at the convenient expression
\begin{equation}\label{Vexp1}
  V(t) \,=\, \frac{1}{\alpha}\int_{\R^2_+} \chi(x)\,\BSS[{\omega}_2(t)](x)
  \,\omega_1(x,t)\dd x \,+\, \cO(t^\infty)\,,
\end{equation}
which is the starting point of our analysis. 

The next step is to replace $\omega_1(t)$ by $\alpha S_1(t)\delta_z$ and
$\omega_2(t)$ by $\alpha S_2(t)\delta_z$ in the right-hand side of
\eqref{Vexp1}, to obtain a more explicit formula. We take $\epsilon > 0$
small enough so that $\epsilon < 1-r_0$, and defining $\hat\omega_2(t)
= \omega_2(t) - \alpha S_2(t)\delta_z$ we observe that
\begin{align*}
  \bigl\|\chi \BSS[\hat\omega_2(t)]\bigr\|_{L^\infty} \,&\le\,
  \bigl\|\chi \BSS[\hat\omega_2(t) \1_{\{x_2 \le \epsilon\}}]\bigr\|_{L^\infty}
  +   \bigl\|\chi \BSS[\hat\omega_2(t) \1_{\{x_2 > \epsilon\}}]\bigr\|_{L^\infty} \\
  \,&\lesssim\, \bigl\|\hat\omega_2(t) \1_{\{x_2 \le \epsilon\}}\bigr\|_{L^1}
  +  \bigl\|\hat\omega_2(t) \1_{\{x_2 > \epsilon\}}\|_{L^1}^{1/2}\,
  \bigl\|\hat\omega_2(t)\|_{L^\infty}^{1/2}\,\lesssim\, t^{1/8}\,.
\end{align*}
Here, to bound the term involving $\1_{\{x_2 \le \epsilon\}}$, we used the fact
that the Biot--Savart kernel is uniformly bounded when the first argument lies
in the support of $\chi$ and the second one in an $\epsilon$-neighborhood of the
boundary; this allows us to bound the $L^\infty$ norm of the velocity in terms
of the $L^1$ norm of the vorticity, the latter one being controlled by the
short-time estimate \eqref{firstasym}.  For the term involving
$\1_{\{x_2 > \epsilon\}}$, a crude estimate is sufficient since the quantity
$\bigl\|\hat\omega_2(t) \1_{\{x_2 > \epsilon\}} \|_{L^1}$ is $\cO(t^\infty)$ by
Corollary~\ref{cor:local}.  Since $\|\omega_1(t)\|_{L^1}$ is uniformly bounded,
we deduce from \eqref{Vexp1} that
\[
  V(t) \,=\, \int_{\R^2_+} \chi(x)\,\BSS[S_2(t)\delta_z](x)\,\omega_1(x,t)\dd x
  + \cO(t^{1/8})\,.
\]
A similar argument shows that $\|\chi\BSS[S_2(t)\delta_z]\|_{L^\infty}$ is
uniformly bounded, and since we know from Corollary~\ref{cor:local} that
$\|\omega_1(t) - \alpha S_1(t)\delta_z\|_{L^1} = \cO(t^{1/8})$ we conclude that
\begin{equation}\label{Vexp2}
  V(t) \,=\, \alpha\int_{\R^2_+} \chi \,\BSS[S_2(t)\delta_z]
  \,S_1(t)\delta_z\dd x + \cO(t^{1/8})\,.
\end{equation}

From now on the proof is essentially a direct calculation based on the explicit
expressions given in Section~\ref{sec2}. The most complicated task is to compute
the leading term of the velocity field created by the boundary layer term
$S_2(t)\delta_z$. We start from the formula
\[
  \bigl(S_2(t)\delta_z\bigr)(y) \,=\, K_2(y,z,t) \,=\, -G_t(y-z^*) -
  K_0(y,z,t)\,,
\]
see \eqref{Kdecomp} and \eqref{S12def}. The mirror vortex term $G_t(y-z^*)$ is
completely negligible for small times, so we concentrate on the contribution of
the boundary layer term $K_0(y,z,t)$. As in the argument before, when using the
Biot--Savart formula, we can restrict the integration domain to an
$\epsilon$-neighborhood of the boundary, because the vorticity is extremely
small outside that neighborhood. Therefore, for $x$ in the support of $\chi$,
we have the asymptotic formula
\begin{equation}\label{Vexp3}
  \bigl(\BSS[S_2(t)\delta_z]\bigr)(x) \,=\, W(x,t) + \cO(t^\infty)\,,
\end{equation}
where
\[
  W(x,t) \,=\, \frac{1}{2\pi}\int_\R\int_0^\epsilon
  \frac{1}{|x-y|^2}\,\begin{pmatrix} x_2 - y_2 \\ y_1 - x_1\end{pmatrix}
  \,K_0(y,z,t)\dd y_2\dd y_1\,.
\]
We recall that $K_0(x,y,t)$ can be decomposed as in \eqref{newdec},
where the correction term $\widetilde{K}_0$ is defined in Lemma~\ref{lem:K0decomp}.
The contribution of $\widetilde{K}_0$ to the speed $W(x,t)$ can be estimated
as follows:
\[
  \bigl|\widetilde W(x,t)\bigr| \,\lesssim\, \int_\R\int_0^\epsilon \frac{1}{|x-y|}
  \,\frac{1}{t}\,\Bigl|\widetilde{K}_0\Bigl(\frac{y}{\sqrt{t}},\frac{z}{\sqrt{t}}
  ,1\Bigr)\Bigr|\dd y_2\dd y_1 \,\lesssim\, \varphi(t^{-1/2}) \,=\,
  \cO(t^{1/2})\,,
\]
where $\varphi$ is defined in \eqref{vphidef}. Here and in what follows,
we use the observation that $|x-y|$ is bounded away from zero when
$x \in \supp(\chi)$ and $y_2 \le \epsilon$. Evaluating the contribution
of the explicit term in \eqref{newdec}, we thus find
\begin{align*}
  W(x,t) \,&=\, \frac{1}{\pi^2}\int_\R\int_0^\epsilon
  \frac{1}{|x-y|^2}\,\begin{pmatrix} x_2 - y_2 \\ y_1 - x_1\end{pmatrix}
  \,\frac{1}{\sqrt{t}}\,g\Bigl(\frac{y_2}{\sqrt{t}}\Bigr)
  \,\frac{1}{1+y_1^2}\dd y_2\dd y_1 + \cO(t^{1/2}) \\ 
  \,&=\, \frac{1}{\pi^2}\int_\R\int_0^{\epsilon t^{-1/2}}
  \frac{1}{(x_1-y_1)^2 + (x_2-y_2\sqrt{t})^2}\,\begin{pmatrix} x_2 - y_2\sqrt{t}
  \\ y_1 - x_1\end{pmatrix} \,\frac{g(y_2)}{1+y_1^2}\dd y_2\dd y_1 + \cO(t^{1/2})\,.
\end{align*}
For $x \in \supp(\chi)$, we expand the integrand in powers of $\sqrt{t}$
to extract the leading order term. Recalling that $\int_0^\infty g(y_2)\dd y_2 = 1/2$,
we obtain $W(x,t) = \overline{W}(x) + \cO(t^{1/2})$, where
\begin{equation}\label{Vexp4}
  \overline{W}(x) \,=\, \frac{1}{2\pi^2}\int_\R \frac{1}{(x_1{-}y_1)^2 + x_2^2}
  \,\begin{pmatrix} x_2 \\ y_1 {-} x_1\end{pmatrix} \,\frac{1}{1+y_1^2}\dd y_1
  \,=\, \frac{1}{2\pi } \frac{1}{x_1^2 + (1+x_2)^2}\,\begin{pmatrix}
    1+x_2 \cr -x_1 \end{pmatrix}\,. 
\end{equation}
Here the last equality can be established using the residue theorem, or by taking
the gradient with respect to $x$ of the following scalar identity which can be
found in \cite[formula~4.296.2]{GR}\:
\[
  \frac{1}{\pi}\int_{\R} \log\bigl((x_1{-}y_1)^2 + x_2^2\bigr) \,\frac{1}{1+y_1^2}\dd y_1
  \,=\, \log\bigl(x_1^2 + (1+x_2)^2\bigr)\,, \qquad \forall\,x \in \R^2_+\,.
\]

Since $|W(x,t) - \overline{W}(z)| \le |W(x,t) - \overline{W}(x)| + |
\overline{W}(x) - \overline{W}(z)| \lesssim |x-z| + \sqrt{t}$ on $\supp(\chi)$,
it follows from \eqref{Vexp2}--\eqref{Vexp4} that 
\[
  V(t) \,=\, \alpha \overline{W}(z) + \alpha \int_{\R^2_+} \chi(x)\bigl(
  W(x,t) - \overline{W}(z)\bigr) S_1(t)\delta_z\dd x + \cO(t^{1/8})
  \,=\, \frac{\alpha}{4\pi} \begin{pmatrix} 1 \\ 0 \end{pmatrix} +
  \cO(t^{1/8})\,,
\]  
because the integral in the right-hand side is $\cO(t^{1/2})$ due to
the explicit form of the Lamb--Oseen vortex $S_1(t)\delta_z$. This
concludes the proof of \eqref{Vasym}. \QED

%%%%%%%%%%%%%%%%%%%%%%%%%%%%%%%%%%%%%%%%%%%%%%%%%%%%%%%%%%%%%%%%%%%%%%%%%%%
%%%%%%%%%%%%%%%%%%%%%%%%%%%%%%%%%%%%%%%%%%%%%%%%%%%%%%%%%%%%%%%%%%%%%%%%%%%

\appendix

\section{Derivation of the Stokes semigroup formula}\label{appA}

In this section we give a short and self-contained derivation of the semigroup
formula \eqref{Stokes}. In contrast with \cite{Abe}, we do not use any information
on the Stokes semigroup in velocity formulation. Instead, following Ukai \cite{Ukai},
we introduce a scalar quantity which has the same scaling as the velocity and
satisfies also the Dirichlet boundary condition. Here and in the sequel, we denote
by $|\partial_1|$ the nonlocal differential operator on $\R$ corresponding to the
multiplication by $|k|$ at the level of the Fourier transform. 

\begin{lem}\label{lem:wdef}
For any $\omega \in L^1(\R^2_+)$, the nonlocal equation $\bigl(\partial_2 -|\partial_1|\bigr)
v(x_1,x_2) = \omega(x_1,x_2)$ in $\R^2_+$ has a unique solution which is a bounded and
continuous function of $x_2 \in [0,+\infty)$ with values in the space $L^1(\R,\D x_1)$.
Moreover $v(\cdot,0) + \gamma[\omega] = 0$, where $\gamma$ is the boundary trace operator
defined in \eqref{gammadef}. 
\end{lem}

\begin{proof}
We denote by $\hat\omega(k,x_2)$ the partial Fourier transform of $\omega \in L^1(\R^2_+)$ 
with respect to the horizontal variable $x_1$, namely
\begin{equation}\label{ParFour}
  \hat\omega(k,x_2) \,=\, \int_\R \omega(x_1,x_2)\,e^{-ikx_1}\dd x_1\,, \qquad
  k \in \R\,, \quad x_2 > 0\,.
\end{equation}
In the new variables $(k,x_2) \in \R^2_+$, the equation we want to solve is
$\bigl(\partial_2 -|k|\bigr)\hat v(k,x_2) =\hat\omega(k,x_2)$, and the general
solution takes the form
\begin{equation}\label{gensol}
  \hat v(k,x_2) \,=\, -\int_{x_2}^\infty e^{|k|(x_2-y_2)}\,\hat\omega(k,y_2)\dd y_2
  \,+\, A(k)\,e^{|k|x_2}\,, \qquad k \in \R\,, \quad x_2 > 0\,,
\end{equation}
for some function $A : \R \to \C$.  Since
$\hat\omega\in L^\infty(\R, L^1(\R_+))$, the first term belongs to
$L^\infty(\R^2_+)$. Clearly, we have to take $A(k) = 0$ for all $k \in \R$,
otherwise the right-hand side of \eqref{gensol} is not the Fourier transform of
a bounded function of $x_2 > 0$ with values in $L^1(\R,\D x_1)$.  Returning to
the original variables $x = (x_1,x_2) \in \R^2_+$, we obtain the representation
formula
\begin{equation}\label{vrep}
  v(x) \,=\, \frac{1}{\pi}\int_{\R^2_+} \frac{x_2 - y_2}{|x-y|^2}\,
  \1_{\{y_2 \ge x_2\}}\,\omega(y)\dd y\,, \qquad x \in \R^2_+\,.
\end{equation}
If we assume that $\omega \in L^1(\R^2_+)$, it is straightforward to verify using
\eqref{vrep} that $x_2 \mapsto v(\cdot,x_2)$ is a bounded and continuous
function from $[0,+\infty)$ into $L^1(\R)$. Moreover, setting $x_2 = 0$,
we see that $v(\cdot,0) + \gamma[\omega] = 0$ where $\gamma[\omega]$ is the
boundary trace operator \eqref{gammadef}. 
\end{proof}
  
\begin{df}
We define the Stokes semigroup $\bigl(S(t)\bigr)_{t\ge0}$ by
\begin{equation}\label{Sdef}
  S(t) \,=\, \bigl(\partial_2 - |\partial_1|\bigr)\,e^{t\Delta_D} \bigl(\partial_2 -
  |\partial_1|\bigr)^{-1}\,, \qquad t \ge 0\,,
\end{equation}
where $\bigl(\partial_2 - |\partial_1|\bigr)^{-1}$ is the operator introduced in
Lemma~\ref{lem:wdef} and $e^{t\Delta_D}$ is the heat semigroup in $\R^2_+$ with
Dirichlet boundary condition.
\end{df}

Let us check that this definition coincides with the vorticity formulation of
the Stokes system with no-slip boundary condition. It follows immediately from
the definition~\eqref{Sdef} that $S(0) = \1$ and $S(t_1+t_2) = S(t_1)S(t_2)$ for
all $t_1, t_2 \ge 0$. Moreover, if $\omega_0 \in L^1(\R^2_+)$, it is clear that
$\omega(t) := S(t)\omega_0$ solves the heat equation
$\partial_t\omega(t) = \Delta\omega(t)$ for $t > 0$, because the derivatives
$\partial_t,\Delta$ commute with $\partial_2$ and $|\partial_1|$.  Finally, for
any $t > 0$, we have by Lemma~\ref{lem:wdef}
\[
  \gamma[\omega(t)] \,=\, -\bigl(\partial_2 - |\partial_1|\bigr)^{-1}\omega(t)
  \,\Big|_{x_2 = 0} \,=\, - e^{t\Delta_D}\bigl(\partial_2 - |\partial_1|\bigr)^{-1}
  \omega_0\,\Big|_{x_2 = 0} \,=\, 0\,,
\]
because $e^{t\Delta_D}$ is the heat semigroup in $L^1(\R^2_+)$ with Dirichlet boundary condition.
It remains to compute the integral kernel of the operator $S(t)$ for $t > 0$, which will
allow us to establish the properties stated in Proposition~\ref{prop:Stokes}.  

\begin{lem}\label{lem:Skernel}
For any $t > 0$ and any $\omega_0 \in L^1(\R^2_+)$ we have the integral representation
\begin{equation}\label{Stokesbis}
  \bigl(S(t)\omega_0\bigr)(x) \,=\, \int_{\R^2_+} K(x,y,t)\,\omega_0(y)\dd y\,,
  \qquad x \in \R^2_+\,,\quad t > 0\,,
\end{equation}
where the kernel $K(x,y,t)$ is defined in \eqref{Kdecomp}, \eqref{K0def}. 
\end{lem}

\begin{proof}
In view of \eqref{Sdef} we have the decomposition $S(t) = e^{t\Delta_D} - S_0(t)$, where
\begin{equation}\label{S0def}
  S_0(t) \,=\, \bigl[\,e^{t\Delta_D}\,,\,\partial_2 - |\partial_1|\bigr]\,
  \bigl(\partial_2 - |\partial_1|\bigr)^{-1} \,=\, \bigl[e^{t\Delta_D}\,,\,\partial_2\bigr]\,
  \bigl(\partial_2 - |\partial_1|\bigr)^{-1}\,.
\end{equation}
In the last equality, we used the fact that the semigroup $e^{t\Delta_D}$ is a convolution
operator in the horizontal variable $x_1$, which therefore commutes with the (nonlocal)
differential operator $|\partial_1|$. Since
\begin{equation}\label{HeatDir}
  \Bigl(e^{t\Delta_D}f\Bigr)(x) \,=\, \int_{\R^2_+} \Bigl(G_t(x-y) - G_t(x-y^*)\Bigr)
  f(y) \dd y\,, \qquad x \in \R^2_+\,, \quad t > 0\,, 
\end{equation}
where $G$ is the Gaussian kernel \eqref{GPdef}, it remains to verify that
\begin{equation}\label{Stokester}
  \bigl(S_0(t)\omega_0\bigr)(x) \,=\, \int_{\R^2_+} K_0(x,y,t)\,\omega_0(y)\dd y\,,
  \qquad x \in \R^2_+\,,\quad t > 0\,,
\end{equation}
where $K_0(x,y,t)$ is defined by \eqref{K0def}. Take $\omega_0 \in L^1(\R^2_+)$
and consider $v_0 = \bigl(\partial_2 - |\partial_1|\bigr)^{-1}\omega_0$. According to
Lemma~\ref{lem:wdef} we have
\begin{equation}\label{vform}
  v_0(z) \,=\, -\frac{1}{\pi}\int_{\R^2_+} \frac{y_2 - z_2}{|y-z|^2}\,
  \1_{\{y_2 \ge z_2\}}\,\omega(y)\dd y \,=\, -\int_\R \int_{z_2}^\infty P(y-z)
  \,\omega(y)\dd y_2\dd y_1\,,
\end{equation}
where $P$ is the Poisson kernel defined in \eqref{GPdef}. On the other hand, a straightforward
calculation shows that
\begin{equation}\label{DirCom}
  \Bigl(\bigl[e^{t\Delta_D}\,,\,\partial_2\bigr]v_0\Bigr)(x) \,=\, 2\partial_{x_2}
  \int_{\R^2_+} G_t\bigl(x-z^*\bigr)\,v_0(z)\dd z\,.
\end{equation}
Combining \eqref{vform}, \eqref{DirCom} and using Fubini's theorem, we arrive at the
representation formula \eqref{Stokester} with $K_0(x,y,t)$ defined by \eqref{K0def}. 
\end{proof}

\section{Properties of the Stokes solution}\label{appB}

We collect here, for easy reference, some estimates on the Stokes solution
$\omega(t) = S(t)\delta_z$ where $z = (0,1)$. According to \eqref{Stokes},
\eqref{Kdecomp} we have the decomposition $\omega(t) = w_d(t) - w_0(t)$
where
\[
  w_d(x,t) \,=\, G_t(x-z) - G_t(x-z^*)\,, \qquad w_0(x,t) \,=\, K_0(x,z,t)\,,
  \qquad x\in\R^2_+\,,\quad t > 0\,.
\]
The dipole term $w_d(t)$ can be estimated by direct calculations, which
give (we omit the details)
\[
  \|w_d(t)\|_{L^1} \,\lesssim\, \frac{1}{(1+t)^{1/2}}\,, \qquad 
  \|w_d(t)\|_{L^\infty} \,\lesssim\, \frac{1}{t(1+t)^{1/2}}\,, \qquad
  t > 0\,.
\]
On the other hand, using \eqref{Kbd1}, we can control the boundary layer
term $w_0(t)$ as follows
\[
  \|w_0(t)\|_{L^1} \,\lesssim\, \frac{1}{(1+t)^{1/2}}\,, \qquad 
  \|w_0(t)\|_{L^\infty} \,\lesssim\, \frac{1}{t^{1/2}(1+t)}\,, \qquad
  t > 0\,.
\]
Altogether, we arrive at
\begin{equation}\label{StokesLp}
  \|S(t)\delta_z\|_{L^1} \,\lesssim\, \frac{1}{(1+t)^{1/2}}\,, \qquad 
  \|S(t)\delta_z\|_{L^\infty} \,\lesssim\, \frac{1}{t(1+t)^{1/2}}\,, \qquad
  t > 0\,. 
\end{equation}

We next consider the velocity field $u(t) = \BS[\omega(t)] = v_d(t) - v_0(t)$,
where $v_d(t) = \BS[w_d(t)]$ and $v_0(t) = \BS[w_0(t)]$. The dipole term $v_d(t)$
is given by the explicit formula
\[
  v_d(x,t) \,=\, \frac{1}{\sqrt{t}}\biggl\{v^G\Bigl(\frac{x-z}{\sqrt{t}}\Bigr) -
  v^G\Bigl(\frac{x-z^*}{\sqrt{t}}\Bigr)\biggr\}\,, \qquad x\in\R^2_+\,,\quad t > 0\,,
\]
where $v^G$ is defined in \eqref{GvGdef}. Again, a direct calculation shows
that
\[
  \|v_d(t)\|_{L^2}^2 \,\sim\, \frac{1}{4\pi}\,\log\frac{1}{t} \quad
  \text{as }t\to 0\,, \qquad
  \|v_d(t)\|_{L^2}^2 \,\lesssim\, \frac{1}{t} \quad\text{when }t\ge 1\,.
\]
To estimate the kinetic energy of the boundary layer term $v_0(t)$, a possibility
is to use Lemmas~\ref{lem:HLS}, \ref{lem:BS} and an interpolation argument. 
If $1 < p < 2 < q < \infty$ and $1/q = 1/p - 1/2$, then $\|v_0(t)\|_{L^2} \le
\|v_0(t)\|_{L^p}^{1-2/q} \,\|v_0(t)\|_{L^q}^{2/q}$ and we know that
\begin{align*}
  \|v_0(t)\|_{L^p} \,&\lesssim\, \int_{\R^2_+}x_2^{(2/p)-1}|w_0(x,t)|\dd x
  \,\lesssim\, \frac{t^{1/p-1/2}}{(1+t)^{1/2}}\,, \\[1mm]
  \|v_0(t)\|_{L^q} \,&\lesssim\, \|w_0(t)\|_{L^p} \,\le\, \|w_0(t)\|_{L^1}^{1/p}
  \,\|w_0(t)\|_{L^\infty}^{1-1/p} \,\lesssim\, \frac{t^{1/(2p)-1/2}}{(1+t)^{1-1/(2p)}}\,.
\end{align*}
It follows that
\[
  \|v_0(t)\|_{L^2} \,\lesssim\, \frac{t^\theta}{(1+t)^{1/2 + \theta}}\,,
   \qquad \text{where}\quad \theta \,=\, \frac{1}{q}\Bigl(\frac12 - \frac1q\Bigr) \,>\, 0\,.
\]
We conclude that the total kinetic energy $\|u(t)\|_{L^2}^2$ diverges logarithmically
as $t \to 0$, and decays to zero like $1/t$ as $t \to +\infty$. 

\section{A Gr\"onwall-type lemma}\label{appC}

Here we state and prove an elementary result that is used in section~\ref{ssec52}. 

\begin{lem}\label{lem:Gron}
Let $f : [0,T] \to \R_+$ be a continuous function satisfying
\begin{equation}\label{eqf}
  f(t) \,\le\, \int_0^t\Bigl(a + \frac{b}{\sqrt{s}}\,f(s)^{1/2}\Bigr)\dd s\,,
  \qquad\,\forall t \in [0,T]\,,
\end{equation}
for some positive constants $a$ and $b$. Then $f(t) \le ct$ for all $t \in [0,T]$,
where the constant $c$ is defined by the implicit formula $c = a + b c^{1/2}$. 
\end{lem}

\begin{proof}
We claim that, for all $n \in \N$, there exists a constant $C_n \ge 0$ such
that   
\begin{equation}\label{fupper}
  f(t) \,\le\, C_n\,t^{\gamma_n}\,, \qquad \forall\, t \in [0,T]\,,
\end{equation}
where $\gamma_n = 1 - 2^{-n}$. Indeed, if $n = 0$, we can take
$C_0 = \sup\{f(t)\,;\, t \in [0,T]\}$. Suppose now that $f$ satisfies the bound
\eqref{fupper} for some $n \in \N$. Using \eqref{eqf} and noticing that
$\gamma_{n+1}=(1+ \gamma_n)/2$, we find
\[
  f(t) \,\le\,  \int_0^t \Bigl(a + \frac{b}{\sqrt{s}}\,C_n^{1/2}
  s^{\gamma_n/2}\Bigr)\dd s \,=\, at + \frac{b\,C_n^{1/2}}{\gamma_{n+1}}
  \,t^{\gamma_{n+1}}\,\le\, C_{n+1}\,t^{\gamma_{n+1}}\,,
\]
for all $t \in [0,T]$, where 
\begin{equation}\label{recur}
  C_{n+1} \,=\, a\,T^{1-\gamma_{n+1}} + \frac{b\,C_n^{1/2}}{\gamma_{n+1}} \,=:\, \phi_n(C_n)\,.
\end{equation}
We observe that the function $\phi_n$ is strictly increasing and has a unique
fixed point in $(0,+\infty)$, which we denote by $M_n$. Note also that
$\phi_n\leq \overline{\phi}$ for all $n\in \N$, where
$\overline{\phi}(x) := a\max(T,1) + 2 b x^{1/2}$. Denoting by $\overline{ M}$ the
unique fixed point of $\overline{\phi}$, we deduce that $M_n\leq \overline{ M}$
for all $n\in \N$. Now, let $M=\max (\overline{ M} ,C_0)$; assuming that
$C_n\leq M$, the monotonicity of $\phi_n$ and the inequality $M_n\leq M$ imply
that $C_{n+1}\leq \phi_n(M) \leq M$. By induction, we infer that $C_n\leq M$ for
all $n\in \N$.

It follows that the sequence $(C_n)_{n \in \N}$ converges as $n \to +\infty$ to
its $\omega$-limit set $\cE$, which is a nonempty compact subset of
$[0,M]$. Now, since $\gamma_n \to 1$ as $n \to \infty$, a standard result in
dynamical systems theory asserts that $\phi(\cE) = \cE$ where
$\phi(x) = \lim_{n\to\infty}\phi_n(x) = a + b x^{1/2}$. As the map $\phi$ has a
unique fixed point $c$, which is globally attracting, the only nonempty subset
of $[0,M]$ that is invariant under $\phi$ is $\{c\}$, and this implies that
$C_n \to c$ as $n \to +\infty$. Thus taking the limit $n \to +\infty$ in
\eqref{fupper} we obtain the desired result.
\end{proof}

%%%%%%%%%%%%%%%%%%%%%%%%%%%%%%%%%%%%%%%%%%%%%%%%%%%%%%%%%%%%%%%%%%%%%%%%%%%
%%%%%%%%%%%%%%%%%%%%%%%%%%%%%%%%%%%%%%%%%%%%%%%%%%%%%%%%%%%%%%%%%%%%%%%%%%%

%%%%%%%%%%%%%%%%%%%%%%%%%%%%%%%%%%%%%%%%%%%%%%%%%
% \bibliographystyle{abbrv}
% \bibliography{biblio.bib}
%%%%%%%%%%%%%%%%%%%%%%%%%%%%%%%%%%%%%%%%%%%%%%%%%%

\bigskip\noindent
{\bf Anne-Laure Dalibard}\\
Sorbonne Université, Université Paris Cité,
CNRS, INRIA\\
Laboratoire Jacques-Louis Lions, LJLL, EPC ANGE\\
F-75005 Paris, France\\
Email\: {\tt anne-laure.dalibard@sorbonne-universite.fr}

\bigskip\noindent
{\bf Thierry Gallay}\\
Universit\'e Grenoble Alpes, CNRS, Institut Universitaire de France\\
Institut Fourier, 100 rue des Maths, F-38610 Gi\`eres, France\\
Email\: {\tt thierry.gallay@univ-grenoble-alpes.fr}

\end{document}